\documentclass[10pt]{article}
\usepackage{mathrsfs,amsthm,graphicx,color,verbatim,bbm,amsmath,amsfonts,amssymb,newclude,nicefrac,amsfonts,graphicx,color,enumerate,hyperref,mathtools}

\theoremstyle{plain}
\newtheorem{theorem}{Theorem}[section]
\newtheorem{lemma}[theorem]{Lemma}
\newtheorem{prop}[theorem]{Proposition}
\newtheorem{cor}[theorem]{Corollary}

\newtheorem{definition}[theorem]{Definition}
\newtheorem{setting}[theorem]{Setting}

\usepackage[nocompress]{cite}

\newcommand{\E}{\mathbb{E}}
\renewcommand{\P}{\mathbb{P}}
\newcommand{\Q}{\mathbb{Q}}
\newcommand{\R}{\mathbb{R}}
\newcommand{\N}{\mathbb{N}}
\newcommand{\smallsum}{\textstyle\sum}
\newcommand{\smallprod}{\textstyle\prod}
\newcommand{\scp}[3]{\langle#1,#2\rangle_{#3}}

\newcommand{\WE}{W}
\newcommand{\BI}{B}
\newcommand{\af}{\mathbf{a}}
\newcommand{\Af}{\mathbf{A}}

\begin{document}
\title{Uniform error estimates for artificial neural network approximations for heat equations}

\author{
Lukas Gonon$^{1}$,
Philipp Grohs$^2$, 
Arnulf Jentzen$^{3,4}$, \\
David Kofler$^5$,
and David \v{S}i\v{s}ka$^6$ \bigskip
\\
\small{$^1$ Faculty of Mathematics and Statistics, University of St.\ Gallen,} \\
\small{Switzerland, e-mail: \href{mailto:lukas.gonon@unisg.ch}{lukas.gonon@unisg.ch}}\smallskip \\
\small{$^2$ Faculty of Mathematics and Research Platform Data Science, } \\
\small{University of Vienna, Austria, e-mail: \href{mailto:philipp.grohs@univie.ac.at}{philipp.grohs@univie.ac.at}}\smallskip \\
\small{$^3$ Department of Mathematics, ETH Zurich,} \\
\small{Switzerland, e-mail: \href{mailto:arnulf.jentzen@sam.math.ethz.ch}{arnulf.jentzen@sam.math.ethz.ch}}\smallskip \\
\small{$^4$ Faculty of Mathematics and Computer Science, University of M\"unster,} \\
\small{Germany, e-mail: \href{mailto:ajentzen@uni-muenster.de}{ajentzen@uni-muenster.de}}\smallskip \\
\small{$^5$ Department of Mathematics, ETH Zurich, }\\
\small{Switzerland, e-mail: \href{mailto:david.kofler@zenai.ch}{david.kofler@zenai.ch}}\smallskip \\
\small{$^6$ School of Mathematics, University of Edinburgh,}\\
\small{United Kingdom, e-mail: \href{mailto:d.siska@ed.ac.uk}{d.siska@ed.ac.uk}}
\smallskip \\
}

\maketitle
\begin{abstract}
Recently, artificial neural networks (ANNs) in conjunction with stochastic gradient descent optimization methods have been employed to approximately compute solutions of possibly rather high-dimensional partial differential equations (PDEs). Very recently, there have also been a number of rigorous mathematical results in the scientific literature which examine the approximation capabilities of such deep learning based approximation algorithms for PDEs. These mathematical results from the scientific literature prove in part that algorithms based on ANNs are capable of overcoming the curse of dimensionality in the numerical approximation of high-dimensional PDEs. In these mathematical results from the scientific literature usually the error between the solution of the PDE and the approximating ANN is measured in the $L^p$-sense with respect to some $p \in [1,\infty)$ and some probability measure.  In many applications it is, however, also important to control the error in a uniform $L^\infty$-sense. The key contribution of the main result of this article is to develop the techniques to obtain error estimates between solutions of PDEs and approximating ANNs in the uniform $L^\infty$-sense.
In particular, we prove that the number of parameters 
of an ANN
to uniformly approximate the classical solution of the heat equation 
in a region $ [a,b]^d $
for a fixed time point $ T \in (0,\infty) $
grows at most polynomially 
in the dimension $ d \in \N $ 
and 
the reciprocal of the approximation precision $ \varepsilon >  0 $. This verifies that  ANNs can overcome the curse of dimensionality 
in the numerical approximation of the heat equation when the error is measured in the uniform $L^\infty$-norm.
\end{abstract}
\tableofcontents

\section{Introduction}
\label{sec:intro}
Artificial neural networks (ANNs) 
play a central role in machine learning applications
such as  computer vision  (cf., e.g., \cite{Huang2017DenselyCC,NIPS2012_4824,DBLP:journals/corr/SimonyanZ14a}),
 speech recognition (cf., e.g., \cite{graves2013speech,38131,wu2016stimulated}), 
game intelligence
 (cf., e.g., 
 \cite{Silver2016,Silver2017}),
and finance (cf., e.g., \cite{Becker2019,Buehler2018,Sirignano2019}).
Recently, ANNs in conjunction with stochastic gradient descent optimization methods have also been employed to approximately compute solutions of possibly rather high-dimensional partial differential equations (PDEs); cf., for example, \cite{FujiiTakahashi2019,Becker2018,BeckJentzenE2019,Becker2019,BeckerCheriditoJentzen2019,PhamWarin2019,Magill2018NeuralNT,BeckBeckerCheridito2019,Berg2018AUD,ChanMikaelWarin2019,chen2019deep,EHanJentzen2017CMStat,GoudenegeMolent2019,HanLong2018,Han2018,HurePhamWarin2019,JacquierOumgari2019,LyeMishraRay2019,LongLuMaDong2018,Sirignano2018,weinan2018deep,Farahmand2017DeepRL,henry2017deep,Dockhorn2019} and the references mentioned therein. The numerical simulation results in the above named references indicate that such deep learning based approximation methods for PDEs have the fundamental power to overcome the {\it{curse of dimensionality}} (cf., e.g., Bellman~\cite{MR2641641}) in the sense that the precise number of the real parameters of the approximating ANN grows at most polynomially in both  the dimension $d \in \mathbb{N} = \{1,2,3,\ldots \}$ of the PDE under consideration and the reciprocal $\varepsilon^{-1}$ of the prescribed approximation precision $\varepsilon > 0$.
Very recently, there have also been a number of rigorous mathematical results examining the approximation capabilities of these deep learning based approximation algorithms for PDEs (see, e.g., \cite{HanLong2018,Sirignano2018,BernerGrohsJentzen2018,ElbraechterSchwab2018, HornungJentzen2018,grohs2019deep,HutzenthalerJentzenKruse2019,JentzenSalimovaWelti2018,KutyniokPeterseb2019,ReisingerZhang2019}). These works prove in part that algorithms based on ANNs are capable of overcoming the curse of dimensionality in the numerical approximation of high-dimensional PDEs. In particular, the works
 \cite{BernerGrohsJentzen2018,ElbraechterSchwab2018, HornungJentzen2018,grohs2019deep,HutzenthalerJentzenKruse2019,JentzenSalimovaWelti2018,KutyniokPeterseb2019,ReisingerZhang2019} provide mathematical convergence results of such deep learning based numerical approximation methods for PDEs with dimension-independent error constants and convergence rates which depend on the dimension only polynomially.

 Except of in the article Elbr\"achter et al.~\cite{ElbraechterSchwab2018}, in each of the approximation results in the above cited articles \cite{HanLong2018,Sirignano2018,BernerGrohsJentzen2018, HornungJentzen2018,grohs2019deep,HutzenthalerJentzenKruse2019,JentzenSalimovaWelti2018,KutyniokPeterseb2019,ReisingerZhang2019} the error between the solution of the PDE and the approximating ANN is measured in the $L^p$-sense with respect to some $p \in [1,\infty)$ and some probability measure.  In many applications it is, however, also important to control the error in a uniform $L^\infty$-sense. This is precisely the subject of this article. More specifically, it is the key contribution of Theorem~\ref{thm:heat-eq} in Subsection~\ref{ssec:ql-heat} below, which is the main result of this article, to prove that ANNs can overcome the curse of dimensionality 
 in the numerical approximation of the heat equation when the error is measured in the uniform $L^\infty$-norm. The arguments used to prove the approximation results in the above cited articles, where the error between the solution of the PDE and the approximating ANN is measured in the $L^p$-sense with respect to some $p \in [1,\infty)$ and some probability measure, can not be employed for the uniform $L^\infty$-norm approximation and the article Elbr\"achter et al.~\cite{ElbraechterSchwab2018} is concerned with a specific class of PDEs so that the PDEs can essentially be solved analytically and the error analysis in \cite{ElbraechterSchwab2018} strongly exploits this explicit solution representation. The key contribution of the main result of this article, Theorem~\ref{thm:heat-eq} in Subsection~\ref{ssec:ql-heat} below,  is to develop the techniques to obtain error estimates between solutions of PDEs and approximating ANNs in the uniform $L^\infty$-sense.
 To illustrate the findings of the main result of this article in more detail, we formulate in the next result a particular case of Theorem~\ref{thm:heat-eq}.

\begin{theorem}
\label{thm:intro}
Let $a \in \R$, $b \in (a,\infty)$,   
$c,T \in (0,\infty)$, $ \mathbf a \in C^1(\R,\R)$, let
$\Af_d \colon \R^d \to \R^d$, $ d \in \N$, satisfy for all  $d \in \N$, $x=(x_1,\ldots,x_d) \in \R^d$ that 
$ \Af_d(x) = (\mathbf a(x_1),\ldots,\mathbf a(x_d)) $, 
let $
\mathbf{N} = \cup_{L \in \N \cap [2,\infty)}\cup_{\substack{l_0,\ldots,l_L \in  \N 
}} \big( 
\times_{k=1}^L (\R^{l_k\times l_{k-1}}\times\R^{l_k})
\big)$, 
let $\mathcal{P}\colon \mathbf{N} \to \N$ 
and
$\mathcal{R} \colon \mathbf{N} \to \cup_{m,n \in \N} C(\R^m,\R^n)$ satisfy for all 
$L \in \N \cap [2,\infty)$, 
$l_0,\ldots,l_L \in \N $,  
$\Phi = 
(
(\WE_1,\BI_1),
\ldots,
(\WE_L,  \BI_L)
)
\in (\times_{k=1}^{L} (\R^{l_k \times l_{k-1}} \times \R^{l_k}))$, 
$x_0 \in \R^{l_0},\ldots, x_{L-1}\in \R^{l_{L-1}}$ with 
$\forall \, k \in \N \cap (0,L) \colon x_k = \Af_{l_k}(\WE_k x_{k-1} + \BI_k)$ that
$\mathcal{P}(\Phi) =  \smallsum_{k=1}^{L} l_k(l_{k-1} + 1)$,
$(\mathcal R \Phi) \in C(\mathbb R^{l_0}, \mathbb R^{l_L})$,
and
$(\mathcal{R}\Phi)(x_0) =  \WE_L x_{L-1} + \BI_L$,
let
$ \varphi_d \in C(\R^d, \R)$, $d \in \N$,
let $(\phi_{\varepsilon,d})_{(\varepsilon,d) \in (0,1]\times \N}  \subseteq \mathbf{N}$, let  $\left \| \cdot \right \| \colon \left(\cup_{d \in \mathbb{N}} \R^d\right) \to [0,\infty)$ satisfy for all $d \in \N$, $x=(x_1,\ldots,x_d) \in \R^d$ that  $\left \| x \right \| = [ \sum_{j=1}^d |x_j|^2 ]^{1/2} $, 
and assume for all $\varepsilon \in (0,1]$, $d \in \N$, $ x \in \R^d$ that
\begin{equation}
\label{eq:heat-eq3-ass1}
(\mathcal{R} \phi_{\varepsilon,d}) \in C(\R^d,\mathbb R)\,, \qquad \left| (\mathcal{R}\phi_{\varepsilon,d})(x)\right| 
+
\| (\nabla (\mathcal{R}\phi_{\varepsilon,d}))(x)\|
\leq
c  d^c(1+\|x\|^{c})\,,
\end{equation}
\begin{equation}
\text{and} \qquad \mathcal{P}(\phi_{\varepsilon,d})  \leq c  d^c\varepsilon^{-c}\,,
\qquad
| \varphi_d(x) - (\mathcal{R}\phi_{\varepsilon,d})(x)| 
\leq 
\varepsilon c  d^c(1+\|x\|^{c})
\,.
\end{equation}
Then
\begin{enumerate}[(i)]
\item \label{it:heat-eq3-1} 
there exist unique at most polynomially growing 
$u_d \in C([0,T] \times \R^d, \R)$,
$d \in \N$, 
which satisfy for all
$d \in \N$, 
$ t \in (0,T] $,
$x \in \R^d$ 
that 
$ u_d|_{(0,T] \times \R^d } \in C^{1,2}((0,T] \times \R^d, \R ) $,
$u_d(0,x) = \varphi_d(x)$,
and 
\begin{align} \label{eq:intro-heat}
(\tfrac{\partial }{\partial t}u_d)(t,x) 
&= 
(\Delta_x u_d)(t,x)
\end{align}
and
\item \label{it:heat-eq3-2}
there exist $(\psi_{\varepsilon,d})_{(\varepsilon,d) \in (0,1]\times \N} \subseteq \mathbf{N}$ and
$\kappa \in \R$ such that 
for all $\varepsilon \in (0,1]$, $d \in \N$ we have that
$
\mathcal{P}( \psi_{ \varepsilon,d  } )
\leq \kappa d^{\kappa} \varepsilon^{-\kappa}
$,
$
(\mathcal{R}\psi_{\varepsilon,d}) \in C(\R^d,\R)
$,
and
\begin{equation}
\sup_{x \in [a,b]^d} \left| u_d(T,x) -  ( \mathcal{R} \psi_{ \varepsilon, d  } )( x ) \right| \leq \varepsilon.
\end{equation}
\end{enumerate}
\end{theorem}

Theorem~\ref{thm:intro}
follows directly from Corollary~\ref{cor:heat-eq2} in Subsection~\ref{ssec:ql-heat}  below. Corollary~\ref{cor:heat-eq2}, in turn, is a consequence of Theorem~\ref{thm:heat-eq} in Subsection~\ref{ssec:ql-heat}, the main result of the article. Let us add a few comments on some of the mathematical objects appearing in Theorem~\ref{thm:intro} above. The real number $T \in (0,\infty)$ denotes the time horizon on which we consider the heat equations in \eqref{eq:intro-heat}. The function  $ \mathbf a \in C^1(\R,\R)$ describes the activation function which we employ for the considered ANN approximations. In particular, Theorem~\ref{thm:intro} applies to ANNs with the standard logistic function as the activation function in which case the function  $ \mathbf a \in C^1(\R,\R)$ in Theorem~\ref{thm:intro} satisfies that for all $x \in \R$ we have that $\mathbf a(x)=(1+e^{-x})^{-1}$. The set $\mathbf{N}$ contains all possible ANNs, where each ANN is described abstractly in terms of the number of hidden layers, the number of nodes in each layer, and the values of the parameters (weights and biases in each layer), the function  $\mathcal{P}\colon \mathbf{N} \to \N$ maps each ANN $\Phi \in \mathbf{N}$ to its total number of parameters $\mathcal{P}(\Phi)$, and the function
$\mathcal{R} \colon \mathbf{N} \to \cup_{m,n \in \N} C(\R^m,\R^n)$ maps each ANN $\Phi \in \mathbf{N}$ to the actual function $(\mathcal{R}\Phi)$ (its realization) associated to  $\Phi$ (cf., e.g.,  Grohs et al.~\cite[Section~2.1]{GrohsHornungJentzen2019} and Petersen \& Voigtlaender~\cite[Section~2]{petersen2017optimal}). Item~(ii) in Theorem~\ref{thm:intro} above establishes under the hypotheses of Theorem~\ref{thm:intro} that the solution $u_d \colon [0,T] \times \R^d \to \R$ of the heat equation can at time $T$ be approximated by means of an ANN without the curse of dimensionality. To sum up, roughly speaking, Theorem~\ref{thm:intro} verifies that  if the initial conditions of the
heat equations can be approximated well by ANNs, then the number of parameters 
of an ANN
to uniformly approximate the classical solution of the heat equation 
in a region $ [a,b]^d $
for a fixed time point $ T \in (0,\infty) $
grows at most polynomially 
in the dimension $ d \in \N $ 
and 
the reciprocal of the approximation precision $ \varepsilon >  0 $.

In our proof of 
Theorem~\ref{thm:intro} and Theorem~\ref{thm:heat-eq}, respectively, 
we employ several probabilistic and analytic arguments.
In particular,
we use
a Feynman-Kac type formula for PDEs of the Kolmogorov type
(cf., for example, Hairer et al.\  \cite[Corollary 4.17]{hairer2015}),
Monte Carlo approximations
(cf. Proposition~\ref{prop:kahane}),
Sobolev type estimates for Monte Carlo approximations
(cf. Lemma~\ref{lem:Sobolev2}),
the fact that stochastic differerential equations 
with affine linear coefficient functions
are affine linear in the initial condition
(cf. Grohs et al.~\cite[Proposition~2.20]{HornungJentzen2018}),
as well as
an existence result for realizations of random variables
(cf. Grohs et al.~\cite[Proposition 3.3]{HornungJentzen2018}).

The rest of this research paper is structured in the following way. 
In Section~\ref{sec:sob} below preliminary results 
on 
Monte Carlo approximations together with
Sobolev type estimates are established. Section~\ref{sec:sde} contains preliminary results
on
stochastic differential equations.
In Section~\ref{sec:error} we employ 
the results from Sections~\ref{sec:sob}--\ref{sec:sde}
to obtain uniform error estimates for ANN approximations.
In Section~\ref{sec:heat-eq} these uniform error estimates 
are used to prove Theorem~\ref{thm:heat-eq} in Subsection~\ref{ssec:ql-heat} below, the main result of this article.

\section{Sobolev and Monte Carlo estimates}
\label{sec:sob}

\subsection{Monte Carlo estimates}
In this subsection we recall in Lemma~\ref{lem:kahane-const} below an estimate for the $p$-Kahane--Khintchine constant from the scientific literature
(cf., for example, Cox et al.~\cite[Definition 5.4]{Coxetal2016}
or
Grohs et al.~\cite[Definition 2.1]{HornungJentzen2018}).
Lemma~\ref{lem:kahane-const}, in particular, 
ensures that the $p$-Kahane--Khintchine constant
grows at most polynomially in $p$.
Lemma~\ref{lem:kahane-const}
will be employed  in the proof of Corollary~\ref{cor:pde-nnet-ap4} in Subsection \ref{ssec:qee} below.
Our proof of Lemma~\ref{lem:kahane-const} is based on an application of Hyt\"onen et al.\ \cite[Theorem 6.2.4]{MR3752640}
and is a slight extension of Grohs et al.~\cite[Lemma 2.2]{HornungJentzen2018}. For completeness we also recall in Definition~\ref{def:p2-kahane-khintchine} below the notion of the Kahane--Khintchine constant (cf., e.g.,  Cox et al.~\cite[Definition 5.4]{Coxetal2016}). 
Proposition~\ref{prop:kahane} below
 is an  $L^p $-approximation result 
for Monte-Carlo approximations. 
This $L^p $-approximation result 
for Monte-Carlo approximations is one of the main ingredients in our proof of
Lemma~\ref{lem:Sobolev2} in Subsection~\ref{ssec:smc} below.
Proposition~\ref{prop:kahane} 
is well-known in the literature and
is proved, e.g., as Corollary~5.12 in Cox et al.~\cite{Coxetal2016}.

\begin{definition}
\label{def:p2-kahane-khintchine}
Let $p\in(0,\infty)$. 
Then we denote by $\mathfrak K_p \in [0, \infty]$ the quantity given
by
\begin{equation}
\begin{split}
& \mathfrak K_p = \\		
& \sup \left \{
c \in [0,\infty) \colon 
\left[
\begin{array}{l}
\text{$\exists$ probability space $(\Omega, \mathcal F, \mathbb P)\colon$} \\
\text{$\exists$ $\mathbb R$-Banach space $(E, \|\cdot\|_E) \colon$} \\	
\text{$\exists$ $k\in \N$: $\exists x_1, \ldots,x_k \in E \setminus \{0\}\colon$}\\
\text{$\exists$ $\mathbb P$-Rademacher family $r_j: \Omega \to \{-1,1\}$, $j\in \N \colon$}\\
\left(\E\left[\left\|\sum_{j=1}^k r_j x_j\right\|_E^p \right]\right)^{1/p} = c \left(\E\left[\left\|\sum_{j=1}^k r_j x_j\right\|_E^2 \right]\right)^{1/2}
\end{array}
\right]
\right\}
\end{split}
\end{equation}
and we refer to $\mathfrak K_p$ as the $p$-Kahane--Khintchine constant.  	
\end{definition}

\begin{lemma}
\label{lem:kahane-const}
Let $ p \in [1,\infty)$. Then
$\mathfrak{K}_p \leq \sqrt{\max\{1,p-1\}}$ (cf.\ Definition~\ref{def:p2-kahane-khintchine}).
\end{lemma}
\begin{proof}[Proof of Lemma~\ref{lem:kahane-const}]
Throughout this proof 
let $ (\Omega, \mathcal{F}, \P)$ be a probability space, 
let $(E, \left\| \cdot \right\|_E)$ be a $\R$-Banach space, let $ k \in \N$,
$ x_1, \ldots, x_k \in E \backslash \{ 0 \}$,
and let $r_j \colon \Omega \to \{-1,1\} $, $ j \in \N$, be i.i.d.\ random variables which satisfy that
\begin{equation}
  \P( r_1 = -1 ) = \P (r_1 = 1) = \frac{1}{2}.
\end{equation}
Observe that Jensen's inequality ensures for all $ q \in [1,2]$ that
\begin{equation}
\label{eq:kahane-pf1}
  \left(
  \E\left[ 
  \left \|
  \sum_{j = 1}^k r_j x_j
  \right \|_E^q
  \right]
  \right)^{ \! \nicefrac{1}{q} }
=
\left(
  \E\left[ 
  \left \|
  \sum_{j = 1}^k r_j x_j
  \right \|_E^{2\frac{q}{2}}
  \right]
  \right)^{ \! \nicefrac{1}{q} }
\leq 
\left(
  \E\left[ 
  \left \|
  \sum_{j = 1}^k r_j x_j
  \right \|_E^{ 2 }
  \right]
  \right)^{ \! \nicefrac{1}{2} }
  .
\end{equation}
In addition, observe that \cite[Theorem 6.2.4]{MR3752640}
(applied with $ q \leftarrow q $, $ p \leftarrow 2 $ for  $ q \in (2,\infty)$
in the notation of \cite[Theorem 6.2.4]{MR3752640})
ensures for all 
$ q \in (2, \infty) $ that
\begin{equation}
\label{eq:kahane-pf2}
  \left(
  \E\left[ 
  \left \|
  \sum_{j = 1}^k r_j x_j
  \right \|_E^q
  \right]
  \right)^{ \! \nicefrac{1}{q} }
\leq 
\sqrt{q-1}
\left(
  \E\left[ 
  \left \|
  \sum_{j = 1}^k r_j x_j
  \right \|_E^{ 2 }
  \right]
  \right)^{ \! \nicefrac{1}{2} }
  .
\end{equation}
Combining this with \eqref{eq:kahane-pf1} demonstrates
that for all $ q \in [1,\infty) $ we have that
\begin{equation}
  \left(
  \E\left[ 
  \left \|
  \sum_{j = 1}^k r_j x_j
  \right \|_E^q
  \right]
  \right)^{ \! \nicefrac{1}{q} }
\leq 
\sqrt{\max\{1,q-1\}}
\left(
  \E\left[ 
  \left \|
  \sum_{j = 1}^k r_j x_j
  \right \|_E^{ 2 }
  \right]
  \right)^{ \! \nicefrac{1}{2} }
  .
\end{equation}
This completes the proof of Lemma~\ref{lem:kahane-const}.
\end{proof}

\begin{prop}
\label{prop:kahane}
Let $ d, n \in \N $, $p \in [2,\infty)$, 
let $\left \| \cdot \right \| \colon \R^d \to [0,\infty)$ be the standard norm on $\R^d$,
let $ ( \Omega, \mathcal{F}, \P ) $ be a probability space, 
and 
let $ X_{ i } \colon \Omega \to \R^d $,
$ i \in \{ 1, \dots, n \} $,
be i.i.d.\ random variables with 
$
\E\big[ \| X_1 \| \big] < \infty
$.
Then 
\begin{equation}
\left(
\E\!\left[ 
\left\|
\E[ X_1 ]
-
\frac{ 1 }{ n }
\left(  
\sum_{ i = 1 }^{ n }
X_{ i }
\right)
\right\|^p
\right]
\right)^{ \!\! \nicefrac{ 1 }{ p } }
\leq
\frac{
2 \, \mathfrak{K}_p
\left(
\E\!\left[ 
\left\|
X_1 
-
\E[ X_1 ]
\right\|^p
\right]
\right)^{\nicefrac{1}{p} }
}{
\sqrt{ n }
}
.
\end{equation}
(cf.\ Definition~\ref{def:p2-kahane-khintchine}).
\end{prop}

\subsection{Volumes of the Euclidean unit balls}
In this subsection we provide in  Corollary~\ref{cor:vol-ball} below an elementary and well-known  upper bound for the volumes of the Euclidean unit balls. 
Corollary~\ref{cor:vol-ball} will be used in our proof of Corollary~\ref{cor:sob-const2} 
in Subsection~\ref{ssec:ssf} below.
Our proof of Corollary~\ref{cor:vol-ball} employs the elementary 
and well-known results in Lemmas~\ref{lem:gamma}--\ref{lem:ball-gamma} below.
For completeness we also provide in this subsection detailed proofs for Lemmas~\ref{lem:gamma}--\ref{lem:ball-gamma}
and
Corollary~\ref{cor:vol-ball}.

\begin{lemma}
\label{lem:gamma}
Let
$\Gamma \colon (0,\infty) \to (0,\infty)$
and
$ \mathrm{B} \colon (0,\infty)^2 \to (0,\infty)$ satisfy for all $ x,y \in (0,\infty) $ that
$ \Gamma(x) = \int_0^{\infty} t^{x-1} e^{-t} \, dt $
and
$ \mathrm{B}(x,y) = \int_0^1 t^{x-1} (1-t)^{y-1} \, dt $
and
let 
$ x,y \in (0,\infty)$.
Then
\begin{enumerate}[(i)]
\item \label{it:gamma3} we have that $\Gamma(x+1) = x\Gamma(x)$,
\item \label{it:gamma2} we have that $\mathrm{B}(x,y) = \frac{\Gamma(x)\Gamma(y)}{\Gamma(x+y)}$,
\item \label{it:gamma1} we have that $\Gamma(1) = 1$ and  $\Gamma(\frac12) = \sqrt{\pi}$,
and
\item \label{it:gamma4} we have that
\begin{equation}
\sqrt{\frac{2\pi}{x}}\left[\frac{x}{e} \right]^x 
\leq \Gamma(x)
\leq
\sqrt{\frac{2\pi}{x}} \left[ \frac{x}{e} \right]^xe^{\frac{1}{12x}} 
.
\end{equation}
\end{enumerate}
\end{lemma}
\begin{proof}[Proof of Lemma~\ref{lem:gamma}]
Throughout this proof 
let $ \Phi \colon (0,\infty) \times (0,1) \to (0,\infty)^2 $ satisfy for all 
$ s \in (0,\infty)$,
$ t \in (0,1) $
that
\begin{equation}
\label{eq:gamma-phi}
\Phi(s,t) = (s(1-t),st)
\end{equation}
and let $ f \colon (0,\infty)^2 \to (0,\infty)$ satisfy for all 
$s,t \in (0,\infty)$ that
\begin{equation}
f(s,t) = s^{x-1}t^{y-1}e^{-(s+t)}.
\end{equation}
Observe that integration by parts verifies that 
\begin{equation}
\begin{split}
\Gamma(x+1) 
& = 
\int_0 ^{ \infty } t^{ x } e^{ - t } \, dt 
= 
\left[ - t^x e^{ - t } \right]_{ t = 0 } ^{ t = \infty } 
+ x \int_0 ^{ \infty } t^{ x-1 } e^{ -t } \, dt 
 = 
x \Gamma(x).
\end{split}
\end{equation}
This establishes item \eqref{it:gamma3}.
Next note that Fubini's theorem ensures that
\begin{equation}
\label{eq:betagamma}
\begin{split}
\Gamma(x) \Gamma(y) 
& = 
\int_0^{\infty} s^{x-1} e^{-s} \, ds \int_0^{\infty} t^{y-1} e^{-t} \,dt
\\ & = 
\int_0^{\infty}\int_0^{\infty} s^{x-1}t^{y-1}e^{-(s+t)} \,ds\,dt
 = 
\int_0^{\infty}\int_0^{\infty} f(s,t) \,ds \,dt\,.
\end{split}
\end{equation}
Moreover, note that for all 
$ s \in (0,\infty) $,
$ t \in (0,1)$
we have that
\begin{equation}
\det(\Phi^{\prime}(s,t)) = s \in (0,\infty).
\end{equation}
This, 
\eqref{eq:betagamma}, 
the integral transformation theorem
(cf., for example, \cite[Theorem 6.1.7]{MR3098996}),
and Fubini's theorem
prove that
\begin{equation}
\label{eq:gamma-prod}
\begin{split}
\Gamma(x) \Gamma(y) 
& = 
\int_0^\infty \int_0^1 f(\Phi(s,t)) \, |\!\det(\Phi^{\prime}(s,t)) | \, dt\,ds 
\\ & = 
\int_0^{\infty}\int_0^1 (s(1-t))^{x-1} (st)^{y-1}e^{-(s(1-t)+st)} s \,dt\,ds
\\ & = 
\int_0^{\infty} s^{x+y-1}e^{-s} \,ds \int_0^1 t^{y-1}(1-t)^{x-1} \,dt 
 =
\Gamma(x+y) \mathrm{B}(y,x).
\end{split}
\end{equation}
Next note that
the integral transformation theorem
with the diffeomorphism
$ (0,1) \ni t \mapsto (1-t) \in (0,1) $
ensures that
\begin{equation}
\mathrm{B}(x,y) 
=
\int_0^1 t^{x-1} (1-t)^{y-1} \, dt 
=
\int_0^1 (1-t)^{x-1} t^{y-1} \, dt 
= \mathrm{B}(y,x).
\end{equation}
Combining this with \eqref{eq:gamma-prod} establishes item \eqref{it:gamma2}.
Next note that
\begin{equation}
\label{eq:gamma1}
\Gamma(1) 
=
\int_0^{ \infty } t^{1 - 1} e^{ - t } \, dt 
=
\int_0^{ \infty } e^{ - t } \, dt 
=
1.
\end{equation}
Item \eqref{it:gamma2}
and
the integral transformation theorem
with the diffeomorphism
$ (0, \frac{\pi}{2} ) \ni t \mapsto [\sin(t)]^2 \in (0,1) $ 
therefore
verify that 
\begin{equation}
\begin{split}
\frac{
\left[
\Gamma\!\left( \tfrac{1}{2} \right)
\right]^2
}
{ \Gamma(1) }
= 
\mathrm{B}\!\left(\tfrac{1}{2}, \tfrac{1}{2} \right) 
=
\int_0^1 t^{ -\nicefrac{1}{2} } (1-t)^{-\nicefrac{1}{2} } \, dt 
= 
\int_0^{ \frac{\pi}{2} }
2 
\,dt
= 
\pi
.
\end{split}
\end{equation}
Combining this with \eqref{eq:gamma1} establishes item \eqref{it:gamma1}.
Next note that  
Artin~\cite[Chapter 3, (3.9)]{MR0165148}
ensures that there exists 
$\mu \colon (0,\infty) \to \R$ which satisfies for all $t\in (0,\infty)$ 
that $0 < \mu(t) < \frac{1}{12t}$ and 
\begin{equation}
\Gamma(t) = \sqrt{2\pi}t^{t-1/2}e^{-t} e^{\mu(t)} \,.
\end{equation}
Hence, we obtain that
\begin{equation}
\sqrt{2\pi}x^{x-1/2}e^{-x} 
\leq \Gamma(x)
\leq
\sqrt{2\pi}x^{x-1/2} e^{-x}e^{\frac{1}{12x}}\,.
\end{equation}
This establishes item \eqref{it:gamma4}.
This completes the proof of Lemma~\ref{lem:gamma}.
\end{proof}
\begin{lemma}
\label{lem:beta}
Let $\mathrm{B} \colon (0,\infty)^2 \to (0,\infty)$ satisfy for all
$ x,y \in (0,\infty) $ 
that
$ \mathrm{B}(x,y) = \int_0^1 t^{x-1} (1-t)^{y-1} \, dt $.
Then it holds for all $p \in [0,\infty)$ that
\begin{equation}
\int_0^{\frac{\pi}{2}} \left[ \sin(t) \right]^{ p} dt 
=
\frac{\mathrm{B}\!\left( \tfrac{p+1}{2},\tfrac12\right)}{2}
.
\end{equation}
\end{lemma}
\begin{proof}[Proof of Lemma~\ref{lem:beta}]
First, note that for all $t \in (0,1)$ we have that
\begin{equation}
\arcsin^{\prime}(t) = (1-t^2)^{-\nicefrac12}.
\label{eq:beta1}
\end{equation}
This and  the integral transformation theorem
with the diffeomorphism $(0,1) \ni t \mapsto \arcsin(t) \in (0,\frac{\pi}{2})$ 
ensure for all $p \in [0,\infty)$ that
\begin{equation}
  \int_0^{\frac{\pi}{2}} \left[ \sin(t) \right]^{ p} dt  = \int_0^1 t^p (1-t^2)^{-\nicefrac12} \, dt
  .
\end{equation}
The integral transformation theorem   
with the diffeomorphism $(0,1) \ni t \mapsto \sqrt{t} \in (0,1)$
hence implies for all $p \in [0,\infty)$ that
\begin{equation}
\begin{split}
& \int_0^{\frac{\pi}{2}} \left[ \sin(t) \right]^{ p} dt
= 
\frac{1}{2} \int_0^1 t^{\nicefrac{ p}{2}-\nicefrac12}(1-t)^{-\nicefrac12}\, dt 
\\ & = 
\frac{1}{2} \int_0^1 t^{\nicefrac{ (p+1)}{2}-1}(1-t)^{\nicefrac12-1}\, dt
= 
\frac{ \mathrm{B}\!\left(\tfrac{p+1}{2},\tfrac12\right)}{2}.
\end{split}
\end{equation}
This completes the proof of Lemma~\ref{lem:beta}.
\end{proof}
\begin{lemma}
\label{lem:polar}
Let $ R \in (0,\infty] $,
for every $d \in \N$ let $\left \| \cdot \right \|_{\R^d} \colon \R^d \to [0,\infty)$ be the standard norm on $\R^d$,
for every $d \in \{2,3,\ldots\}$ let 
$
B_d = \{ x \in \R^d \colon  \| x \|_{\R^d}  < R \}
$
and
\begin{equation}
S_d
=
\begin{cases}
(0,2\pi) & \text{for $d = 2$ }  \\
(0,2\pi) \times (0,\pi)^{d-2} & \text{for $d \in \{3,4,\ldots\}$},
\end{cases}
\end{equation}
and let $ T_d \colon (0,R) \times S_d \to \R^d$, $d\in \{2, 3, \ldots\}$,  satisfy for all 
$d \in \{2,3,\ldots\}$, 
$ r \in (0,R) $,
$ \varphi \in (0,2\pi) $,
$ \vartheta_1,\ldots,\vartheta_{d-2} \in (0,\pi) $ 
that if $d=2$ then $T_2(r,\varphi ) = r (\cos(\varphi),\sin(\varphi))$
and if $d\geq 3$ then
\begin{equation}
\begin{split}
T_d(r, \varphi, & \vartheta_1,\ldots,\vartheta_{d-2}) 
= 
r\Big(
\cos(\varphi) \left[\smallprod_{i=1}^{d-2} \sin(\vartheta_i)\right],
\sin(\varphi)
\left[ \smallprod_{i=1}^{d-2} \sin(\vartheta_i) \right],
\\ & 
\cos(\vartheta_1) \left[ \smallprod_{i=2}^{d-2} \sin(\vartheta_i) \right],
\ldots,
\cos(\vartheta_{d-3})\sin(\vartheta_{d-2}),
\cos(\vartheta_{d-2})
\Big)
.
\end{split}
\end{equation}
Then
\begin{enumerate}[(i)]
\item \label{it:polar1} it holds for all 
$ r \in (0,R) $,
$ \varphi \in (0,2\pi)$
that
\begin{equation}
| \det((T_2)^{\prime}(r,\varphi)) |
=
r,
\end{equation}
\item \label{it:polar2} it holds for all 
$ d \in \{3, 4, \ldots\} $,
$ r \in (0,R) $,
$ \varphi \in (0,2\pi)$,
$ \vartheta_1,\ldots,\vartheta_{d-2} \in (0,\pi) $
that
\begin{equation}
| \det ((T_d)^{\prime}(r,\varphi,\vartheta_1,\ldots,\vartheta_{d-2})) |
=
r^{d-1} \left[\smallprod_{i=1}^{d-2} [\sin(\vartheta_{i})]^{i}\right],
\end{equation}
and
\item \label{it:polar3} it holds for all 
$ d \in \{2,3,\ldots\}$
and
all 
$\mathcal{B}(\R^d) / \mathcal{B}([0,\infty))$-measurable functions
$f \colon \R^d \to [0,\infty) $ 
that
\begin{equation}
\label{eq:polar}
\begin{split}
& \int_{B_d} f(x) \, dx
= 
\int_0^R \int_{S_d} 
f(T_d(r,\phi))\,
| \det ((T_d)^{\prime}(r,\phi)) |
\, d\phi\,dr\,
.
\end{split}
\end{equation}
\end{enumerate}
\end{lemma}
\begin{proof}[Proof of Lemma~\ref{lem:polar}]
Throughout this proof for every $ d \in \{2,3,\ldots\}$ let 
$\lambda_{\R^d}:\mathcal{B}(\R^d) \to [0,\infty]$ be the Lebesgue--Borel measure on $\R^d$.
Observe that for all 
$ r \in (0,R) $,
$ \varphi \in (0,2\pi) $ 
we have that
\begin{equation}
(T_2)^{\prime}(r, \varphi ) 
= 
\begin{pmatrix}
\cos(\varphi) & -r\sin(\varphi) \\
\sin(\varphi) & r\cos(\varphi)
\end{pmatrix}
.
\end{equation}
Hence, we obtain that
for all 
$ r \in (0,R) $,
$ \varphi \in (0,2\pi) $ 
we have that
\begin{equation}
\left| \det \big( (T_2)^{\prime}(r,\varphi) \big) \right|
=
r [\cos(\varphi)]^2 + r [\sin(\varphi)]^2 = r.
\end{equation}
This establishes item \eqref{it:polar1}.
Next observe that Amann \& Escher \cite[Ch. X, Lemma 8.8]{MR2500068} establishes item \eqref{it:polar2}.
To establish item \eqref{it:polar3} we distinguish between
the case $ d = 2 $
and
the case $d \in \{3,4,\ldots\} $.
First, we consider the case $ d = 2 $.
Note that $ T_2 \colon (0,R) \times (0,2\pi) \to T_2((0,R) \times (0,2\pi)) $ is a bijective function.
This and item \eqref{it:polar1} verify that  
$ T_2 \colon (0,R) \times (0,2\pi) \to T_2((0,R)\times(0,2\pi)) $
is a diffeomorphism.
Next observe that
\begin{equation}
T_2((0,R)\times(0,2\pi)) = B_2 \backslash \{ (x,0) \colon x \in [0,\infty)\}.
\end{equation}
The fact that $ T_2 \colon (0,R) \times (0,2\pi) \to T_2((0,R)\times(0,2\pi)) $ is a diffeomorphism,
the fact that $ \lambda_{\R^2}(\{ (x,0) \colon x \in [0,\infty)\}) = 0 $,
and
the integral transformation theorem 
hence yield that  for all 
$ \mathcal{B}(\R^2) / \mathcal{B}([0,\infty))$-measurable functions
$f \colon \R^2 \to [0,\infty) $ 
we have that
\begin{equation}
\int_{B_2} f(x) \, dx 
=
\int_0^R \int_{S_2} f(T_2(r,\phi))\, \big|\det\big( (T_2)^{\prime}(r,\phi)\big) \big|
\, d\phi\,dr\,
.
\end{equation}
This establishes item \eqref{it:polar3} in the case $ d = 2 $.
Next we consider the case $ d \in \{3,4,\ldots\} $.
Note that Amann \& Escher \cite[Ch. X, Lemma 8.8]{MR2500068} 
implies that 
$ T_d \colon (0,R) \times (0,2\pi) \times (0,\pi)^{d-2} \to T_d((0,R) \times (0,2\pi) \times (0,\pi)^{d-2}) $
is a diffeomorphism with 
\begin{equation}
T_d((0,R) \times (0,2\pi) \times (0,\pi)^{d-2})
=
B_d \backslash 
\big(
[0,\infty) \times \{0\} \times \R^{d-2}
\big).
\end{equation}
The fact that 
$
\lambda_{\R^d}\big([0,\infty) \times \{0\} \times \R^{d-2}\big) 
=
0
$
and
the integral transformation theorem 
hence verify that  for all 
$ \mathcal{B}(\R^d) / \mathcal{B}([0,\infty))$-measurable functions
$f \colon \R^d \to [0,\infty) $ 
we have that
\begin{equation}
\int_{B_d} f(x) \, dx 
=
\int_0^R \int_{S_d} f(T_d(r,\phi)) \, \big| \det \big( (T_d)^{\prime}(r,\phi) \big)\big|
\, d\phi \, dr\,
.
\end{equation}
This establishes item \eqref{it:polar3} in the case $ d \in \{3,4,\ldots\}$.
This completes the proof of Lemma~\ref{lem:polar}.
\end{proof}

\begin{lemma}
\label{lem:ball-gamma}
Let $d \in \N$,  
let $\left\|  \cdot \right \| \colon \R^d \to [0,\infty)$ be the standard norm on $\R^d$, 
let $\lambda:\mathcal{B}(\R^d)\to [0,\infty]$ be the Lebesgue--Borel measure on $\R^d$,
let $\mathbbm{B} \subseteq \R^d$ be the set given by $\mathbbm{B} =  \{ x \in \R^d \colon \|x\| < 1\}$,
and let $\Gamma \colon (0,\infty) \to (0,\infty)$ satisfy for all 
$ x \in (0,\infty) $
that
$ \Gamma(x) = \int_0^{\infty} t^{x-1} e^{-t} \, dt $. 
Then
\begin{enumerate}[(i)]
\item  \label{it:ball1} for $d\in \{2,3,\ldots\}$ we have that
\begin{equation}
\label{eq:ball-int}
\lambda(\mathbbm{B})
= \frac{2\pi}{d}
\left[\prod_{i=1}^{d-2}  \int_0^\pi [\sin(\vartheta_{i})]^{i} \, d\vartheta_i\right] 	
\end{equation}	
and
\item \label{it:ball2} we have that
\begin{equation}
\label{eq:ball-gamma}
\lambda(\mathbbm{B})
= \frac{\pi^{\nicefrac{d}{2}}}{\Gamma\!\left( \frac{d}{2}+1 \right)}
.
\end{equation}
\end{enumerate}
\end{lemma}
\begin{proof}[Proof of Lemma~\ref{lem:ball-gamma}]
To establish~\eqref{eq:ball-int} and~\eqref{eq:ball-gamma} we distinguish between the case $d=1$, the case $d=2$, and the case $d \geq 3$.
First, we consider the case $d=1$. 
Note that 
items \eqref{it:gamma3} 
and
\eqref{it:gamma1}
in Lemma~\ref{lem:gamma} verify that 
\begin{equation}
\Gamma\!\left(\tfrac{1}{2} +1 \right) 
=
\frac{\Gamma\!\left(\tfrac12\right)}{2}
=
\frac{\pi^{\nicefrac12}}{2}.
\label{}
\end{equation}
This implies that
\begin{equation}
\frac{\pi^{\nicefrac12}}{\Gamma\!\left(\frac{1}{2}+1\right)} = 2.
\end{equation}
Combining this and the fact that $\lambda(\mathbbm{B}) = \lambda((-1,1)) = 2$ establishes \eqref{eq:ball-gamma} in the case $d=1$.
Next we consider the case $d=2$. 
Note that items \eqref{it:polar1} and \eqref{it:polar3} in Lemma~\ref{lem:polar} 
and
Fubini's theorem 
prove that
\begin{equation}
\label{eq:ball-gamma-pf1}
\begin{split}
\lambda(\mathbbm{B}) = \int_{\mathbbm{B}} \, dx 
=
\int_{0}^{2\pi}\int_{0}^1 r \, dr \, d\varphi 
=
\pi.
\end{split}
\end{equation}
Next note that items \eqref{it:gamma3}
and
\eqref{it:gamma1} 
in Lemma~\ref{lem:gamma} verify that 
\begin{equation}
  \Gamma( 2) = \Gamma(1+1) = \Gamma(1) = 1.
  \label{}
\end{equation}
This implies that
\begin{equation}
  \frac{\pi}{\Gamma(1+1)} = \pi.
\end{equation}
Combining this with \eqref{eq:ball-gamma-pf1} establishes~\eqref{eq:ball-int}  
and~\eqref{eq:ball-gamma} in the case $d=2$.
Next we consider the case $d \geq 3$. 
Note that
items \eqref{it:polar2}--\eqref{it:polar3} in Lemma~\ref{lem:polar} 
and
Fubini's theorem
ensure that 
\begin{equation}
\label{eq:ball-int-pf}
\begin{split}
  \lambda(\mathbbm{B}) 
&  =
  \int_{\mathbbm{B}} 
  \, dx 
\\ & =
  \int_{0}^{\pi} \cdots\int_{0}^{\pi}   \int_{0}^{2\pi} \int_{0}^1  r^{d-1} 
  \left[\prod_{i=1}^{d-2} [\sin(\vartheta_{i})]^{i}\right]  
  dr \, d\varphi \, d\vartheta_1 \, \cdots \, d\vartheta_{d-2}  
\\ & =
  \frac{1}{d} 
  \int_{0}^{2\pi} \, d\varphi
  \left[\prod_{i=1}^{d-2} 
  \int_{0}^{\pi} [\sin(\vartheta_{i})]^{i} \, d\vartheta_i\right] 
.
\end{split}
\end{equation}
This establishes~\eqref{eq:ball-int} in the case $d\in \{ 3, 4, \ldots \}$.
Moreover, note that~\eqref{eq:ball-int-pf} and the fact that for all $k \in \N$ we have that $\int_{0}^{ \pi } [\sin(t)]^k \, dt = 2 \int_0^{\frac{\pi}{2}} [\sin(t)]^k \, dt $
verify that 
\begin{equation}
\begin{split}
\lambda(\mathbbm{B})  
=
\frac{4}{d}
\int_{0}^{\frac{\pi}{2}} \, d\varphi 
\left[\prod_{i=1}^{d-2} 2 \int_{0}^{\frac{\pi}{2}} [\sin(\vartheta_{i})]^{i} \, d\vartheta_i\right] 
.
\end{split}
\end{equation}
Combining 
this, 
Lemma~\ref{lem:beta}
(applied with
$ p \leftarrow i $
for $ i \in \{0,\ldots,d-2\} $
in the notation of Lemma~\ref{lem:beta}), 
and
item \eqref{it:gamma2} in Lemma~\ref{lem:gamma}  yields that
\begin{equation}
\begin{split}
\lambda(\mathbbm{B}) & = 
\frac{2}{d} 
\mathrm{B}\!\left(\tfrac12,\tfrac12  \right)
\left[\prod_{i=1}^{d-2} \mathrm{B}\!\left(\tfrac{i+1}{2},\tfrac12  \right) \right] 
\\ & = 
\frac{2}{d} 
\frac{\Gamma\!\left(\frac12 \right)\Gamma\!\left(\frac12\right)}{\Gamma(1)}
\left[\prod_{i=1}^{d-2} \frac{\Gamma\!\left(\frac{i+1}{2}  \right)\Gamma\!\left(\frac{1}{2}  \right)}{\Gamma\!\left(\frac{i+2}{2}  \right)} 
  \right]   
=
\frac{2}{d} \frac{\big[\Gamma\!\left(\frac12\right)\big]^{d}
}{\Gamma\!\left( \frac{d}{2} \right)
} 
.
\end{split}
\end{equation}
Items 
\eqref{it:gamma3} 
and
\eqref{it:gamma1}
in Lemma~\ref{lem:gamma} 
hence verify that 
\begin{equation}
  \lambda(\mathbbm{B}) 
  = \frac{\big[\Gamma\!\left(\frac12\right)\big]^{d}}{\frac{d}{2}\Gamma\!\left( \frac{d}{2} \right)}
  = \frac{\pi^{\nicefrac{d}{2}}}{ \Gamma\!\left( \frac{d}{2}+1 \right)}
  .
\end{equation}
This establishes~\eqref{eq:ball-gamma} in the case $d \in \{3, 4, \ldots\}$.
This completes the proof of Lemma~\ref{lem:ball-gamma}.
\end{proof}
\begin{cor}
\label{cor:vol-ball}
Let $d \in \N$, let $\left\|  \cdot \right \| \colon \R^d \to [0,\infty)$ be the standard norm on $\R^d$, and let $\lambda:\mathcal{B}(\R^d)\to [0,\infty]$ be the Lebesgue--Borel measure on $\R^d$.
Then
\begin{equation}
\lambda\big(\{x \in \R^d \colon \|x\| < 1\}\big) 
\leq 
\frac{1}{\sqrt{d\pi}} \Big[\frac{2 \pi e}{d} \Big]^{\nicefrac{d}{2}}.
\end{equation}
\end{cor}
\begin{proof}[Proof of Lemma~\ref{cor:vol-ball}]
Throughout this proof 
let $\Gamma \colon (0,\infty) \to (0,\infty)$ satisfy for all 
$ x \in (0,\infty) $
that
$ \Gamma(x) = \int_0^{\infty} t^{x-1} e^{-t} \, dt $.
Note that Lemma~\ref{lem:ball-gamma} verifies that  
\begin{equation}
\label{eq:vol-ball-pf1}
\lambda\big(\{x \in \R^d \colon \|x\| < 1\}\big)  = \frac{\pi^{\frac{d}{2}}}{
\Gamma\!\left(\frac{d}{2} + 1\right)}
.
\end{equation}
Moreover, note that 
items \eqref{it:gamma3}
and
\eqref{it:gamma4} in Lemma~\ref{lem:gamma} 
imply that for all $x \in (0,\infty)$ we have that
\begin{equation}
{\Gamma(x+1)} = x \Gamma(x) \geq \sqrt{2\pi x} \left[\frac{x }{e}\right]^{x}
.
\end{equation}
Hence, we obtain that
\begin{equation}
\frac{1}{\Gamma\!\left(\frac{d}{2} + 1\right)} \leq 
\frac{1}{\sqrt{d \pi}} \left[ \frac{2e}{d} \right]^{\nicefrac{d}{2}}
.
\end{equation}
Combining this with \eqref{eq:vol-ball-pf1}  yields that
\begin{equation}
\lambda\big(\{x \in \R^d \colon \|x\| < 1\}\big) 
\leq
\frac{1}{\sqrt{d\pi}} \Big[\frac{2 \pi e}{d} \Big]^{\nicefrac{d}{2}}.
\end{equation}
This completes the proof of Lemma~\ref{cor:vol-ball}.
\end{proof}
\subsection{Sobolev type estimates for smooth functions}
\label{ssec:ssf}
In this subsection we present in Corollary~\ref{cor:sob-const3} below  a Sobolev type estimate
for smooth functions with an explicit and dimension-independent constant in the Sobolev type estimate.
Corollary~\ref{cor:sob-const3} will be employed in our proof of
Corollary~\ref{cor:pde-nnet-ap4} in Subsection~\ref{ssec:qee} below. 
Corollary~\ref{cor:sob-const3} is a consequence 
of the elementary results in Lemma~\ref{lem:sob-1d}
and
Corollary~\ref{cor:sob-const2} below.
Corollary~\ref{cor:sob-const2}, in turn, follows from 
Proposition~\ref{prop:sobolev-const1} below.
Proposition~\ref{prop:sobolev-const1} is a special case 
of Mizuguchi et al. \cite[Theorem 2.1]{MR3736595}.
For completeness we also provide in this subsection a proof for Proposition~\ref{prop:sobolev-const1}. Our proof of Proposition~\ref{prop:sobolev-const1}
 employs the well-known results in Lemmas~\ref{lem:int-bound}--\ref{lem:sob-const0} below. 
Results similar to Lemmas~\ref{lem:int-bound}--\ref{lem:sob-const0} can, e.g., be found in Mizuguchi et al.~\cite[Lemma 3.1, Theorem 3.3, and Theorem 3.4]{MR3736595} 
and Gilbarg \& Trudinger~\cite[Lemma 7.16]{MR1814364}.

\begin{lemma}
\label{lem:sob-1d}
Let $ \Phi \in C^1((0,1),\R)$ satisfy that 
$\int_0^1
\left(
\left|
\Phi(x)
\right|
+
\left|
\Phi^{ \prime } ( x )
\right|
\right)
dx
<
\infty.
$
Then
\begin{equation}
\label{eq:sob-1d}
\sup_{ x \in (0,1) }
\left| \Phi(x) \right|
\leq
\int_0^1
\left(
\left|
\Phi(x)
\right|
+
\left|
\Phi^{ \prime } ( x )
\right|
\right)
dx
\,
.
\end{equation}
\end{lemma}
\begin{proof}[Proof of Lemma~\ref{lem:sob-1d}]
First, note that the fundamental theorem of calculus ensures that
for all $ x \in (0,1) $ we have that
\begin{equation}
\Phi(x) = \int_0^1
\left[
\Phi(s) - \int_{x}^s\Phi^{\prime}(t) \, dt 
\right]
ds
.
\end{equation}
The triangle inequality 
and
the hypothesis that 
$\int_0^1
\left(
\left|
\Phi(x)
\right|
+
\left|
\Phi^{ \prime } ( x )
\right|
\right)
dx
<
\infty
$
hence verify that 
for all
$ x \in (0,1) $
we have that
\begin{equation}
\label{eq:sob-1d-pf1}
\begin{split}
\left |
\Phi(x)
\right |
& = 
\left|
\int_0^1
\left[
\Phi(s) - \int_{x}^s\Phi^{\prime}(t) \, dt 
\right]
ds
\right| 
\\ & \leq
\int_0^1
\left[
\left| \Phi(s) \right| 
+ \int_{x}^s\left| \Phi^{\prime}(t) \right|  dt 
\right]
ds 
\\ & \leq
\int_0^1
\left| \Phi(s) \right| ds 
+
\int_0^1\left| \Phi^{\prime}(t) \right|  dt 
<
\infty
.
\end{split}
\end{equation}
This implies~\eqref{eq:sob-1d}.
This completes the proof of Lemma~\ref{lem:sob-1d}.
\end{proof}
\begin{lemma}
\label{lem:int-bound}
Let $d \in \{2,3,\ldots\} $, $p \in (d,\infty)$, let $\left\|  \cdot \right \| \colon \R^d \to [0,\infty)$ be the standard norm on $\R^d$, 
let $\lambda:\mathcal{B}(\R^d)\to[0,\infty]$ be the Lebesgue--Borel measure on $\R^d$,
and let $W \subseteq \R^d$ be a non-empty, bounded, and open set. 
Then 
\begin{equation}
\label{eq:hoelder}
\begin{split}
&  \int_{\cup_{x \in W}\{y-x \colon y \in W\}} \|z\|^{(1-d)\frac{p}{p-1}} \, dz 
\\ & \leq 
  \frac{d(p-1)}{p-d}  \left[ \sup_{v,w \in W} \|v-w\|\right]^{\frac{p-d}{p-1}}
  \lambda\big(\{x \in \R^d \colon \|x\| < 1\}\big)
.
\end{split}
\end{equation}
\label{lem:hoelder}
\end{lemma}
\begin{proof}[Proof of Lemma~\ref{lem:hoelder}]
Throughout this proof 
let $\rho \in (0,\infty) $ satisfy that $ \rho = \sup_{v,w \in W} \| v-w\|$,
let $ V \subseteq \R^d$ be the set given by $V = \cup_{x \in W}\{y-x \colon y \in W\}$,
let $S \subseteq \R^{d-1}$ be the set given by
\begin{equation}
S = 
\begin{cases}
(0,2\pi) & \colon d=2 \\
(0,2\pi) \times (0,\pi)^{d-2} & \colon d \in \{3,4,\ldots\},
\end{cases}
\end{equation}
let $B_{r} \subseteq \R^d$,
$r \in (0,\infty)$,
be the sets which satisfy for all $ r \in (0,\infty)$ that
$B_{r} = \{x \in \R^d \colon \|x \| < r\}$,
and let $T_R \colon (0,R) \times S \to \R^d$,
$ R \in (0,\infty] $, satisfy
for all $ R \in (0,\infty] $, 
$ r \in (0,R) $,
$ \varphi \in (0,2\pi) $,
$ \vartheta_1,\ldots,\vartheta_{d-2} \in (0,\pi) $
that if $d=2$ then $T_R(r,\varphi)= r(\cos(\varphi),\sin(\varphi))$ 
and if $d \in \{3,4,\ldots\}$ then
\begin{equation}
\label{eq:hoelder-tn}
\begin{split}
& T_R(r, \varphi,\vartheta_1,\ldots,\vartheta_{d-2}) 
= 
r\Big(
\cos(\varphi) \left[\smallprod_{i=1}^{d-2} \sin(\vartheta_i)\right],
\sin(\varphi)
\left[ \smallprod_{i=1}^{d-2} \sin(\vartheta_i) \right],
\\ & 
\cos(\vartheta_1) \left[ \smallprod_{i=2}^{d-2} \sin(\vartheta_i) \right],
\ldots,
\cos(\vartheta_{d-3})\sin(\vartheta_{d-2}),
\cos(\vartheta_{d-2})
\Big)
.
\end{split}
\end{equation}
Observe that 
\begin{equation}
\label{eq:hoelder-pf2}
(1-d)\frac{p}{p-1} + d = \frac{(1-d)p + d(p-1)}{p-1} = \frac{p-d}{p-1}\,
.
\end{equation}
Next note that
items \eqref{it:polar1}--\eqref{it:polar3} in Lemma~\ref{lem:polar}
and
the fact that for all $ r \in (0,\rho) $, $ \phi \in S $ 
we have that 
$
\| T_{\rho}(r,\phi)  \| = r
$
verify that 
\begin{equation}
\begin{split}
  \int_{B_{\rho}} \|x\|^{(1-d)\frac{p}{p-1}} \, dx 
& = 
\int_0^\rho \int_S  r^{(1-d)\frac{p}{p-1} } 
\left| \det \big((T_{\rho})^{\prime}(r,\phi)\big)\right| \,  d\phi\, dr \\
& = 
\int_0^\rho \int_S  r^{(1-d)\frac{p}{p-1} } r^{d-1}
\left| \det \big((T_\infty)^{\prime}(1,\phi)\big)\right| \,  d\phi\, dr  \,.
\end{split}
\end{equation}
The fact that $V \subseteq B_{\rho}$,
Fubini's theorem,
and~\eqref{eq:hoelder-pf2}
hence yield that 
\begin{equation}
\begin{split}
\int_{V} \|x\|^{(1-d)\frac{p}{p-1}} \, dx 
& \leq
\int_{B_{\rho}} \|x\|^{(1-d)\frac{p}{p-1}} \, dx
\\ & =
\int_S \bigg(\int_0^\rho r^{(1-d)\frac{p}{p-1}} r^{d-1} \,dr\bigg)
\left| \det \big((T_\infty)^{\prime}(1,\phi)\big)\right| \, d\phi \\
& = \rho^{\frac{p-d}{p-1}} \frac{(p-1)}{p-d} \int_{S} \left| \det\big((T_{\infty})^{\prime}(1,\phi)\big)\right|\, d\phi\,. 
\end{split}
\end{equation}
This, Lemma~\ref{lem:polar}, 
and Lemma~\ref{lem:ball-gamma} hence prove that
\begin{equation}
\begin{split}
\int_{V} & \|x\|^{(1-d)\frac{p}{p-1}} \, dx 
\leq
\rho^{\frac{p-d}{p-1}} \frac{(p-1)}{p-d} \int_{S} \left| \det\big((T_{\infty})^{\prime}(1,\phi)\big)\right|\, d\phi \\
& = \rho^{\frac{p-d}{p-1}} \frac{(p-1)}{p-d} \int_0^{2\pi} \int_0^\pi \cdots \int_0^\pi \left[\smallprod_{i=1}^{d-2} [\sin(\vartheta_{i})]^{i}\right] \,d\vartheta_1 \cdots d\vartheta_{d-2}\, d\varphi \\
& =  \rho^{\frac{p-d}{p-1}} \frac{(p-1)}{p-d} d \frac{2\pi}{d} \prod_{i=1}^{d-2} \int_0^\pi [\sin(\vartheta_{i})]^{i} \, d\vartheta_i\, 
 =  \rho^{\frac{p-d}{p-1}} \frac{d(p-1)}{p-d} 
\lambda(B_1)\,
.
\end{split}
\end{equation}
This completes the proof of Lemma~\ref{lem:hoelder}.
\end{proof}

\begin{lemma}
\label{lem:poincare}
Let 
$d \in \{2,3,\ldots\} $,
$ p \in (d,\infty) $, 
let $W \subseteq \R^d$ be a non-empty, open, bounded, and convex set, 
let $ \Phi \in C^1(W,\R) $,
let $\left\|  \cdot \right \| \colon \R^d \to [0,\infty)$ be the standard norm on $\R^d$, 
let $\lambda:\mathcal{B}(\R^d)\to [0,\infty]$ be the Lebesgue--Borel measure on $\R^d$,
and
assume that
$
\int_{ W }(|
\Phi(x)
|^p
+
\|
( \nabla \Phi )( x )
\|^{p})
\, 
dx
<
\infty
$.
Then it holds for all $x \in W$ that
\begin{equation}
\begin{split}
\label{eq:poincare}
  & \left| \lambda(W)\Phi(x) -  \int_{W} \Phi(y)\, dy \right|
\\ & \leq 
  \frac{1}{d} \left[\sup_{v,w \in W} \|v-w\|\right]^d  
  \int_{W}  \left \| (\nabla \Phi)(y) \right \| \left \|x-y \right\|^{1-d} dy.
\end{split}
\end{equation}
\end{lemma}
\begin{proof}[Proof of Lemma~\ref{lem:poincare}]
Throughout this proof let 
$ x \in W $,
let $\rho \in (0,\infty)$ satisfy that $ \rho = \sup_{v,w \in W} \| v-w \|$, 
let $\scp{\cdot}{\cdot}{} \colon \R^d \times \R^d \to \R$ be the $d$-dimensional Euclidean scalar product,
let $S \subseteq \R^{d-1}$ be the set given by
\begin{equation}
S = 
\begin{cases}
(0,2\pi) & \colon d=2 \\
(0,2\pi) \times (0,\pi)^{d-2}  & \colon d \in \{3,4,\ldots\},
\end{cases}
\end{equation}
let $(\omega_{v,w})_{(v,w) \in W^2} \subseteq \R^d$ satisfy for all $ v,w \in W $ that 
\begin{equation}
  \omega_{v,w} = 
  \begin{cases}
  \tfrac{w-v}{\|w-v\|}  & \colon v \neq w \\
  0 & \colon v = w,
  \end{cases}
\end{equation}
let $B_r \subseteq \R^d$,
$ r \in (0,\infty)$, satisfy for all 
$ r \in (0,\infty)$
that
$B_{r} = \{y \in \R^d \colon \|y-x \| < r\}$,
let $E \colon \R^d \to \R^d$ satisfy for all 
$ y \in \R^d $
that
\begin{equation}
\label{eq:poincare-defe}
E(y) = 
\begin{cases}
(\nabla \Phi)(y) & \colon y \in W \\
0  & \colon y \in \R^d \backslash W,
\end{cases}
\end{equation}
and let $T_{R} \colon (0,R) \times S \to \R^d$,
$ R \in (0,\infty] $, satisfy
for all
$ R \in (0,\infty] $, 
$ r \in (0,R)$,
$ \varphi \in (0,2\pi)$,
$ \vartheta_1,\ldots,\vartheta_{d-2} \in (0,\pi) $
that if $d=2$ then 
$T_R(r,\varphi)= r(\cos(\varphi),\sin(\varphi))$
and if $d\geq 3$ then 
\begin{equation}
\label{eq:poincare-tn}
\begin{split}
& T_R(r, \varphi,\vartheta_1,\ldots,\vartheta_{d-2})
= 
r\Big(
\cos(\varphi) \left[\smallprod_{i=1}^{d-2} \sin(\vartheta_i)\right],
\sin(\varphi)
\left[ \smallprod_{i=1}^{d-2} \sin(\vartheta_i) \right],
\\ & 
\cos(\vartheta_1) \left[ \smallprod_{i=2}^{d-2} \sin(\vartheta_i) \right],
\ldots,
\cos(\vartheta_{d-3})\sin(\vartheta_{d-2}),
\cos(\vartheta_{d-2})
\Big)
.
\end{split}
\end{equation}
Observe that the hypothesis that 
$
\int_{ W }
(
|
\Phi(y)
|^{p}
+
\|
( \nabla \Phi )( y )
\|^{p}
)
\,
dy
<
\infty
$,
the hypothesis that $ W $ is bounded,
and
H\"older's inequality
ensure that 
\begin{equation}
\label{eq:poincare-wp}
\begin{split}
& \int_{ W }
|
\Phi(y)
|
+
\|
( \nabla \Phi )( y )
\|
\, 
dy
\\ & \leq
[\lambda(W)]^{\nicefrac{(p-1)}{p}}
\left(
\left[
\int_{ W }
|
\Phi(y)
|^p
\, 
dy
\right]^{\nicefrac{1}{p}}
+
\left[
\int_{ W }
\|
( \nabla \Phi )( y )
\|^p
\, 
dy
\right]^{\nicefrac{1}{p}}
\right)
< 
\infty
.
\end{split}
\end{equation}
Next note  that the assumption that $W$ is convex and the fundamental theorem of calculus 
yield that  for all $ y \in W$  we have that
\begin{equation}
\label{eq:poincare-diff}
\Phi(x) - \Phi(y) 
=
- (\Phi(x + r\omega_{x,y}))|_{r=0}^{r = \|y-x\|} 
=
- \int_{0}^{\|y-x\|}  \frac{d}{dr}\Phi(x+r\omega_{x,y})  \, dr 
.
\end{equation}
The fact that $ \lambda(W) < \infty $,
\eqref{eq:poincare-wp},
the Cauchy-Schwarz inequality,
and
the fact that for all $y \in W\setminus \{x\}$ we have that $\|\omega_{x,y} \| = 1$
hence prove that
\begin{equation}
\begin{split}
& \left | \lambda(W) \Phi(x) - \int_{W} \Phi(y) \, dy \right | 
  = 
\left| 
\int_W \int_{0}^{\|y-x\|} \frac{d}{dr}\Phi(x+r\omega_{x,y})  \, dr \, dy 
\right | 
\\ &  = 
\left| 
\int_W \int_{0}^{\|y-x\|} \scp{ (\nabla \Phi)(x+r\omega_{x,y})}{ \omega_{x,y}}{}
\, dr \, dy 
\right |
\\ & \leq
\int_W \int_{0}^{\|y-x\|} 
\left| \scp{ (\nabla \Phi)(x+r\omega_{x,y}) }{ \omega_{x,y} }{} \right| \, dr \, dy 
\\ & \leq 
\int_W \int_{0}^{\|y-x\|} \left\| (\nabla \Phi)(x+r\omega_{x,y}) \right \| \|\omega_{x,y} \| \, dr \, dy 
\\ & =  
\int_W \int_{0}^{\|y-x\|}\left \| (\nabla \Phi)(x+r\omega_{x,y}) \right \|  dr \, dy  
.
\end{split}
\end{equation}
The fact that $ W \subseteq B_{\rho} $
and Fubini's theorem therefore verify that 
\begin{equation}
\label{eq:poinc-pf2}
\begin{split}
  \left | \lambda(W) \Phi(x) - \int_{W} \Phi(y) \, dy \right | 
& \leq 
  \int_{0}^{\infty} \int_{B_{\rho}}  \left \| E(x+r\omega_{x,y}) \right \|  dy \, dr 
  .
\end{split}
\end{equation}
Next observe that 
the integral transformation theorem 
with the diffeomorphism 
$
\{ v \in \R^d \colon \| v \| < \rho \} \ni y \mapsto y+x \in B_{\rho}
$,
items \eqref{it:polar1}--\eqref{it:polar3} in Lemma~\ref{lem:polar},
the fact that for all 
$ r \in (0,\rho) $,
$ \phi \in S $ 
we have that
$
\| T_\rho(r,\phi)  \| = r
$,
and
\eqref{eq:poincare-tn}
imply that for all 
$ r \in (0,\rho) $ we have that
\begin{equation}
\begin{split}
\label{eq:poinc-pf3}
& \int_{ B_{\rho}  } 
\left \| E(x+r\omega_{x,y})\right\|   
dy
 = 
\int_{ \{ v \in \R^d \colon \|  v \| < \rho \} } 
\left \| E\!\left(x+r\omega_{0,y}\right)\right\|   
dy
\\ & = 
\int_0^\rho \int_S 
\left\| 
E \Big(x+r\tfrac{ T_\rho(s,\phi) }{\| T_\rho(s,\phi)\| }\Big)
\right\|
\left | \det \big((T_\rho)^{\prime}(s,\phi)\big) \right|  \, d\phi \, ds 
\\ & = 
\int_0^\rho \int_S 
\left\| E(x+T_\infty(r,\phi))\right\|
\left | \det \big((T_\infty)^{\prime}(s,\phi)\big) \right| \, d\phi\, ds\,
.
\end{split}
\end{equation}
Items \eqref{it:polar1}--\eqref{it:polar2} in Lemma~\ref{lem:polar},
Fubini's theorem,
and
\eqref{eq:poinc-pf2}
therefore prove that
\begin{equation}
\begin{split}
&  \left | \lambda(W) \Phi(x) - \int_{W} \Phi(y) \, dy \right | 
\\ &  \leq 
  \int_{0}^{\infty}  
  \int_0^\rho \int_S 
  \left\| E(x+T_\infty(r,\phi))\right\|
  \left| \det \big((T_\infty)^{\prime}(s,\phi)\big) \right|  d\phi \, ds \, dr 
\\ & = 
  \int_{0}^{\infty}  \int_S \int_0 ^{\rho}
  \left\| E(x+T_\infty(r,\phi))\right\|
  \left| \det \big((T_\infty)^{\prime}(s,\phi)\big) \right| ds \, d\phi \, dr
\\ & = 
  \int_{0}^{\infty}  \int_S \int_0 ^{\rho}
  \left\| E(x+T_\infty(r,\phi))\right\|
  \left| \det \big((T_\infty)^{\prime}(1,\phi)\big) \right|
  s^{d-1} \, ds \, d\phi \, dr
\\ & =
  \frac{ \rho^d }{d}
  \int_{0}^{\infty}  \int_S 
  \left\| E(x+T_\infty(r,\phi))\right\|
  \left| \det\big((T_\infty)^{\prime}(1,\phi)\big) \right| d\phi \, dr
\,.
\end{split}
\end{equation}
Combining this,
\eqref{eq:poincare-tn},
items \eqref{it:polar1}-\eqref{it:polar3} in Lemma~\ref{lem:polar},
the fact that for all 
$ r \in (0,\infty) $,
$ \phi \in S $ 
we have that
$
\| T_\infty(r,\phi) \| = r
$,
and
\eqref{eq:poincare-defe}
yields that
\begin{equation}
\begin{split}
&  \left | \lambda(W) \Phi(x) - \int_{W} \Phi(y) \, dy \right | 
\\ &  \leq
  \frac{ \rho^d }{d}
  \int_{0}^{\infty}  \int_S 
  \left\| E(x+T_\infty(r,\phi))\right\|
  \left| \det \big((T_\infty)^{\prime}(1,\phi)\big) \right|
  r^{1-d}r^{d-1}
  \,d\phi \, dr
\\ &  =
  \frac{ \rho^d }{d}
  \int_{0}^{\infty}  \int_S 
  \left\| E(x + T_\infty(r,\phi))\right\|
  \|  T_\infty(r,\phi)  \|^{1-d}
  \left| \det\big((T_\infty)^{\prime}(r,\phi)\big) \right|
  \, d\phi \, dr
\\ & = 
\frac{ \rho^d }{d}
\int_{\R^d} \left\|E(x + y) \right\| \|y\|^{1-d} \, dy
\\ & = 
\frac{ \rho^d }{d}
\int_{\R^d} \left\|E(y) \right\| \| x - y\|^{1-d} \, dy
 =  \frac{ \rho^d }{d}
\int_{W} \left\|(\nabla \Phi)(y) \right\| \|x-y\|^{1-d} \, dy
.
\end{split}
\end{equation}
This completes the proof of Lemma~\ref{lem:poincare}.
\end{proof}

\begin{lemma}
\label{lem:sob-const0}
Let $d \in \{2, 3, \ldots\}$,
$p \in (d,\infty)$, 
let $\lambda:\mathcal{B}(\R^d)\to [0,\infty]$ be the Lebesgue--Borel measure on $\R^d$,
let $W \subseteq \R^d$ be an open, bounded, and convex set 
with $\lambda(W) > 0$,
let $ \Phi \in C^1(W,\R) $,
let $\left\|  \cdot \right \| \colon \R^d \to [0,\infty)$ be the standard norm on $\R^d$, 
and 
assume that
$
\int_{ W }(|
\Phi(x)
|^p
+
\|
( \nabla \Phi )( x )
\|^{p})
\, 
dx
<
\infty
$.
Then
\begin{equation}
\begin{split}
& 
  \sup_{x \in W} \left| \Phi(x) - \frac{1}{\lambda(W)} \int_{W} \Phi(y) \, dy \right | 
\\ & \leq
  \frac{\left[\sup_{v,w \in W} \| v - w \|\right]^d}{\lambda(W)d}
  \left[
  \int_{\cup_{x \in W} \{x-y \colon y \in W \} } \|z\|^{(1-d)\frac{p}{p-1}}
  \, dz 
  \right]^{\nicefrac{(p-1)}{p}}
\\ & \cdot 
  \left[
  \int_{W} \| (\nabla \Phi)(y) \|^p \, dy
  \right]^{\nicefrac1p}
<
\infty
.
\end{split}
\end{equation}
\begin{proof}[Proof of Lemma~\ref{lem:sob-const0}]
Throughout this proof 
let $\rho \in [0,\infty) $ satisfy that $ \rho = \sup_{v,w \in W} \| v-w\|$,
let 
$W_x \subseteq \R^d$,
$x \in W$, satisfy for all $x \in W$ that
$W_x = \{x-y \colon y \in W\}$,
let $V \subseteq \R^d$ be the set given by $V = \cup_{x \in W} W_x$,
let $\psi \colon \R^d \to \R$ satisfy for all $x \in \R^d$ that
\begin{equation}
\psi(x) = 
\begin{cases}
\|x\|^{1-d} & \colon x \in V \backslash \{0\} \\
0 & \colon x \in (\R^d \backslash V) \cup \{0\},
\end{cases}
\end{equation}
and let $ E \colon \R^d \to \R^d $  satisfy for all $x \in \R^d$ that
\begin{equation}
E(x)=
\begin{cases}
(\nabla \Phi)(x) & \colon x \in W \\
0  & \colon x \in \R^d \backslash W.
\end{cases}
\end{equation}
Observe that 
the hypothesis that 
$
\int_{ W }(|
\Phi(x)
|^p
+
\|
( \nabla \Phi )( x )
\|^{p})
\, 
dx
<
\infty
$
ensures that  
\begin{equation}
\label{eq:sob-const0-lp}
\begin{split}
\int_{\R^d} \left \| E(x) \right \|^p \, dx 
= 
  \int_{W} \left\|  (\nabla \Phi)(x) \right\|^p \, dx 
<
\infty.
\end{split}
\end{equation}
Moreover, the assumption that $\lambda(W) > 0$, 
Lemma~\ref{lem:poincare},
and 
the integral transformation theorem
prove that
for all $x \in W$ we have that
\begin{equation}
\begin{split}
\label{eq:sob-const0-pf2}
& 
  \left| \Phi(x) -  \frac{1}{\lambda(W)} \int_{W} \Phi(y)\, dy \right|
\leq 
  \frac{\rho^d}{\lambda(W)d}  \int_{W}  
  \left \| (\nabla \Phi)(y) \right \| \left \|x-y \right\|^{1-d} dy 
\\ & =
  \frac{\rho^d}{\lambda(W)d}  \int_{W_x} 
  \left \| (\nabla \Phi)(x-y) \right \| \left \|y \right\|^{1-d} dy 
\\ & \leq 
  \frac{\rho^d}{\lambda(W)d}  \int_{V}  \left \| E(x-y) \right \| \left \|y \right\|^{1-d} dy
 = 
  \frac{\rho^d}{\lambda(W)d}  \int_{\R^d}  \left \| E(x-y) \right \| \psi(y) \, dy \,
.
\end{split}
\end{equation}
Lemma~\ref{lem:int-bound},
\eqref{eq:sob-const0-lp},
and
H\"older's inequality 
therefore yield that  
for all $x \in W$ we have that
\begin{equation}
\begin{split}
& \left| \Phi(x) -  \frac{1}{\lambda(W)} \int_{W} \Phi(y)\, dy \right|
\\ & \leq 
  \frac{\rho^d}{\lambda(W)d}
  \left[
  \int_{\R^d}  \left \| E(x-y) \right \|^p \, dy
  \right]^{\nicefrac1p}
  \left[\int_{\R^d}  
  \psi(y)^{\nicefrac{p}{(p-1)}} \, dy \right]^{\nicefrac{(p-1)}{p}}
\\ & = 
  \frac{\rho^d}{\lambda(W)d}
  \left[
  \int_{W}  \left \| (\nabla \Phi)(y) \right \|^p \, dy
  \right]^{\nicefrac1p}
  \left[\int_{V} \|y\|^{(1-d)\frac{p}{p-1}} \, dy \right]^{\nicefrac{(p-1)}{p}} 
 <
\infty \, .
\end{split}
\end{equation}
This completes the proof of Lemma~\ref{lem:sob-const0}.
\end{proof}

\end{lemma}

\begin{prop}
\label{prop:sobolev-const1}
Let $d \in \{2,3,\ldots\}$,
$p \in (d,\infty)$, 
let $W \subseteq \R^d$ be an open, bounded, and convex set,
let $ \Phi \in C^1(W,\R) $,
let $\left\|  \cdot \right \| \colon \R^d \to [0,\infty)$ be the standard norm on $\R^d$, 
let $\lambda:\mathcal{B}(\R^d)\to[0,\infty]$ be the Lebesgue--Borel measure on 
$\R^d$, 
let $I$ be a finite  and non-empty set,
let 
$W_i \subseteq \R^d$,
$i \in I$,
be open and convex sets,
assume for all $i \in I$, $j \in I \backslash \{i\}$ that
$\lambda(W_i) > 0$,
$  W_i \cap W_j = \emptyset$,
and $ \overline{W} = \cup_{i \in I} \overline{W_i}$,
assume that
$
\int_{ W }
(
|
\Phi(x)
|^p
+
\|
( \nabla \Phi )( x )
\|^{p}
)
\, 
dx
<
\infty
$,
and let 
$(D_{i})_{i \in I} \subseteq [0,\infty)$ satisfy for all $i \in I$ that 
\begin{equation}
D_{i} 
= 
\frac{\left[\sup_{v,w \in W_i} \|v-w\|\right]^d }{\lambda(W_i)d} 
\left[
\int_{\cup_{x \in W_i} \{x-y\colon y \in W_i \} } \| z \|^{(1-d) \tfrac{p}{p-1}} \, dz
\right]^{\nicefrac{(p-1)}{p}}
.
\end{equation}
Then
\begin{equation}
\begin{split}
\sup_{ x \in W }
\left| \Phi(x) \right|
& \leq
2^{ \nicefrac{(p-1)}{p} }
\max\left\{\max_{i \in I}\! \left([\lambda(W_i)]^{-\nicefrac1p}\right),
\max_{i \in I} ( D_{i} ) \right\}
\\
&  \cdot 
\left[ 
\int_{ W }
\left(
\left|
\Phi(x)
\right|^{p}
+
\left\|
( \nabla \Phi )( x )
\right\|^{p}
\right)
dx
\right]^{ \nicefrac{ 1 }{ p } }
.
\end{split}
\end{equation}
\end{prop}
\begin{proof}[Proof of Proposition \ref{prop:sobolev-const1}]
Note that the hypothesis that $W \subseteq \R^d$ is open and convex and~\cite[Theorem 6.3]{rockafellar:convex} imply that for all $i\in I$ we have that $W_i \subseteq W$.
Next observe that 
the hypothesis that
$
\int_{ W }
|
\Phi(x)
|^p
+
\|
( \nabla \Phi )( x )
\|^{p}
\, 
dx
<
\infty
$,
the hypothesis that for all $i \in I$ we have that $W_i$ is bounded,
and
H\"older's inequality 
yield that  
for all 
$i \in I$
we have that
\begin{equation}
\begin{split}
\frac{1}{\lambda(W_i) }  \int_{ W_i } | \Phi(y) | \, dy  
& = 
  \int_{W_i} \left| \frac{1}{\lambda(W_i)} \Phi(y) \right| dy  
\\  \leq &
  \left[\int_{W_i} [\lambda(W_i)]^{-\nicefrac{p}{(p-1)}} \, dy \right]^{\nicefrac{(p-1)}{p}}
  \left[ \int_{W_i} | \Phi(y) |^{p} \, dy \right]^{ \nicefrac{1}{p} }
\\  \leq & \,\,
  [\lambda(W_i)]^{\nicefrac{-1}{p}}
  \left[ \int_{W_i} | \Phi(y) |^{p} \, dy \right]^{ \nicefrac{1}{p} }
<
\infty
.
\end{split}
\end{equation}
The hypothesis that 
$ \overline{W} = \cup_{i \in I} \overline{W_i}$,
the triangle inequality,
and
Lemma~\ref{lem:sob-const0}
(applied with $W \leftarrow W_i$ for $i \in I$  in the notation of Lemma~\ref{lem:sob-const0})
hence verify that 
\begin{equation}
\label{eq:sob-const1-pf1}
\begin{split}
& 
\sup_{x \in W} | \Phi(x) | 
=
  \max_{i \in I} \!\left(\sup_{x \in W_i} | \Phi(x) | \right)
\\ & \leq 
  \max_{i \in I} \left( 
  \sup_{x \in W_i}  \left|\Phi(x) - 
  \frac{1}{\lambda(W_i) } \int_{ W_i }  \Phi(y)  \, dy\right| 
  +
  \left|
  \frac{1}{\lambda(W_i) } \int_{ W_i }  \Phi(y)  \, dy  \right|
  \right) 
\\ & \leq
  \max_{i \in I} 
  \left(
  D_i 
  \left[ \int_{W_i} \| (\nabla \Phi) (y) \|^{p} \, dy \right]^{\nicefrac{1}{p}}    
  + 
  [\lambda(W_i)]^{-\nicefrac{1}{p}}  \left[ \int_{W_i} | \Phi(y) |^{p} \, dy \right]^{\nicefrac{1}{p}}    
  \right)
\\ & \leq
  \max \left \{    
  \max_{i \in I} \!\left( [\lambda(W_i)]^{-\nicefrac{1}{p}}\right)
  ,
  \max_{i \in I} (D_i)
  \right \} 
\\ & \cdot
  \max_{i \in I} 
  \left(
  \left[ \int_{W_i} | \Phi(y) |^{p} \, dy \right]^{\nicefrac{1}{p}}    
+
  \left[ \int_{W_i} \| (\nabla \Phi) (y) \|^{p} \, dy \right]^{\nicefrac{1}{p}}    
  \right)
.
\end{split}
\end{equation}
Next note that for all $(x_i)_{i \in I} \subseteq \R$ 
we have that 
\begin{equation} 
\max_{i \in I} |x_i| \leq \left[\sum_{i \in I} |x_i|^p \right]^{\nicefrac{1}{p}}
.
\end{equation}
Combining 
this,
the hypothesis that 
for all $ i \in I $, $ j \in I \backslash \{i\} $
we have that
$  W_i \cap W_j = \emptyset$
and $ \overline{W} = \cup_{i \in I} \overline{W_i}$,
and
the fact that for all $a,b \in [0,\infty)$ we have that
$(a+b)^p \leq 2^{p-1}(a^p + b^p)$
with
\eqref{eq:sob-const1-pf1}
yields that
\begin{equation}
\begin{split}
  \sup_{x \in W} | \Phi(x) | 
& \leq 
  \max \left \{    
  \max_{i \in I} \!\left([\lambda(W_i)]^{-\nicefrac{1}{p}}\right)
  ,
  \max_{i \in I} (D_i)
  \right \} 
\\ & \quad \cdot 
  \left[
  \sum_{i \in I} \left| 
  \left[ \int_{W_i} | \Phi(y) |^{p} \, dy \right]^{\nicefrac{1}{p}}    
  +  \left[ \int_{W_i} \| (\nabla \Phi) (y) \|^{p} \, dy \right]^{\nicefrac{1}{p}}    
  \right| ^p
  \right]^{\!\nicefrac{1}{p}}
\\ & \leq 
  2^{\frac{p-1}{p}}
  \max \left \{    
  \max_{i \in I} \!\left( [\lambda(W_i)]^{-\nicefrac{1}{p}}\right)
  ,
  \max_{i \in I} (D_i)
  \right \} 
\\ & \quad \cdot 
  \left[
  \sum_{i \in I}  
  \int_{W_i} \left(| \Phi(y) |^{p} 
  +    \| (\nabla \Phi) (y) \|^{p} \right) dy     
  \right]^{\!\nicefrac{1}{p}}
\\ & = 
  2^{\frac{p-1}{p} } \max \left \{    
  \max_{i \in I} \!\left([\lambda(W_i)]^{-\nicefrac{1}{p}}\right)
  ,
  \max_{i \in I} (D_i)
  \right \} 
\\ & \quad \cdot 
  \left[
  \int_{W} \left(| \Phi(y) |^{p}     
  +    \| (\nabla \Phi) (y) \|^{p} \right) dy     
  \right]^{\!\nicefrac{1}{p}}
.
\end{split}
\end{equation}
This completes the proof of Proposition \ref{prop:sobolev-const1}.
\end{proof}
\begin{cor}
\label{cor:sob-const2} 
Let  
$ d \in \{2,3,\ldots\}$,
$ \Phi \in C^1((0,1)^d,\R) $,
let $\left\|  \cdot \right \| \colon \R^d \to [0,\infty)$ be the standard norm on $\R^d$,
and assume that 
$
\int_{ (0,1)^d }
(
|
\Phi(x)
|^{d^2}
+
\|
( \nabla \Phi )( x )
\|^{d^2}
)
\, 
dx
<
\infty
$. Then 
\begin{equation}
\sup_{ x \in (0,1)^d }
\left| \Phi(x) \right|
\leq
8\sqrt{e}
\left[ 
\int_{ (0,1)^d }
\left(
\left|
\Phi(x)
\right|^{d^2}
+
\left\|
( \nabla \Phi )( x )
\right\|^{d^2}
\right)
dx
\right]^{ \nicefrac{ 1 }{ d^2 } }
.
\end{equation}
\end{cor}
\begin{proof}[Proof of Corollary~\ref{cor:sob-const2}]
Throughout this proof 
let $\lambda:\mathcal{B}(\R^d)\to [0,\infty]$ be the Lebesgue--Borel 
measure on $\R^d$,
let $m \in \N$ satisfy that 
\begin{equation}
\label{eq:sob-defm}
m = \min \!\left( \N \cap\! \left[\frac{d}{2\ln(2)}(\ln(2) + \ln(\pi) + 1), \infty\right) \right) ,
\end{equation}
let $ B \subseteq \R^d$ be the set given by
$
  B = \{ x \in \R^d \colon \|x\| < 1 \}
$,
let $I$ be the set given by
$I = \{1,\ldots,2^m\}^d$,
let $W_i, V_i \subseteq \R^d$,
$i \in I$,
be the sets which satisfy for all $i = (i_1,\ldots, i_d) \in I$ 
that
\begin{equation}
\label{eq:sob-const2-wivi}
W_i = (\times_{j=1}^d ((i_j-1)2^{-m}, i_j 2^{-m}))
\quad
\text{and}
\quad
V_i = \cup_{x \in W_i} \{x-y \colon y \in W_i\},
\end{equation}
and let $ (D_i)_{ i \in I }, (\rho_i)_{ i \in I } \subseteq (0,\infty)$ 
satisfy for all $ i \in I$ that
\begin{equation}
\rho_i = \sup_{v,w \in W_i} \|v-w\|
\quad
\text{and}
\quad
D_{i} = \frac{\rho_i^d }{\lambda(W_i)d} 
\left[
\int_{V_i} \| x \|^{(1-d) \tfrac{d^2}{d^2-1}} \, dx 
\right]^{\tfrac{d^2-1}{d^2}}
.
\end{equation}
Observe that \eqref{eq:sob-defm} 
ensures that
\begin{equation}
\label{eq:sob-2m}
[2\pi e]^{\nicefrac{d}{2}}
=
e^{ \frac{d}{2}[\ln(2)+\ln(\pi) + 1]}
=
2^{\frac{d}{2\ln(2)}[\ln(2) + \ln(\pi) + 1]}
\leq
2^m
.
\end{equation}
Hence, we obtain that 
\begin{equation}
[2\pi e]^{\nicefrac{d}{2}} 2^{-m} \leq 1
.
\label{eq:sob-const2-pf0}
\end{equation}
Next note that 
\eqref{eq:sob-const2-wivi}
verifies that  
for all $i \in I$ we have that 
\begin{equation}
\label{eq:sob-const2-pf1}
\lambda(W_i) = 2^{-dm} \qquad \text{and}
\qquad \rho_i = \sqrt{d}2^{-m}
.
\end{equation}
Moreover, note that 
\eqref{eq:sob-defm}
implies that
\begin{equation}
\label{eq:sob-const2-m}
\frac{d}{2\ln(2)}(\ln(2) + \ln(\pi) + 1)
\leq
m
\leq
\frac{d}{2\ln(2)}(\ln(2) + \ln(\pi) + 1)
+
1
.
\end{equation}
This, \eqref{eq:sob-const2-pf1},
and
\eqref{eq:sob-2m} hence verify that 
\begin{equation}
\label{eq:sob-const2-pf2}
\max_{i \in I} \!\left([\lambda(W_i)]^{-\nicefrac{1}{d^2}}\right) 
=
2^{\nicefrac{m}{d}}
\leq
2^{ \frac{\ln(2) + \ln(\pi) + 1}{2\ln(2)} + \frac{1}{d}} 
=
2^{\frac{1}{d}}\sqrt{ 2 \pi e}
.
\end{equation}
Next note that Lemma~\ref{lem:hoelder} 
(applied with $p \leftarrow d^2$, $W \leftarrow W_i$ for $i \in I$  in the notation of Lemma~\ref{lem:hoelder})
assures that for all $i \in I$ we have that
\begin{equation}
\begin{split}
\left[\int_{V_i} \|x\|^{(1-d)\frac{d^2}{d^2-1}} \, dx \right]
^{\frac{d^2-1}{d^2}}
& \leq 
\rho_i^{\frac{d-1}{d}}
\left( \frac{d(d^2-1)}{d^2-d} \lambda(B)
\right)^{\frac{d^2-1}{d^2}}
\\ & =
\rho_i^{\frac{d-1}{d}}
\left( (d+1) \lambda(B)
\right)^{\frac{d^2-1}{d^2}}
.
\end{split}
\end{equation}
Corollary~\ref{cor:vol-ball},
the fact that $(d+1) \frac{1}{\sqrt{d\pi}} (2\pi e)^\frac{d}{2}\geq 1$,
the fact that 
$ (d+1) \leq 2d $,
the fact that $\frac{d^2-1}{d^2} \leq 1$,
and \eqref{eq:sob-const2-pf1}
hence prove that for all $i \in I$ we have that
\begin{equation}
\begin{split}
\bigg[\int_{V_i} &  \|x\|^{(1-d)\frac{d^2}{d^2-1}} \, dx \bigg]
  ^{\frac{d^2-1}{d^2}} \\
& \leq (\sqrt{d}2^{-m})^{\frac{d-1}{d}} \left((d+1) \frac{1}{\sqrt{d\pi}} (2\pi e)^\frac{d}{2}\right)^{\frac{d^2-1}{d^2}} d^{-\frac{d}{2} \frac{d^2-1}{d^2}}
 \\
& \leq (\sqrt{d}2^{-m})^{\frac{d-1}{d}} 2d \frac{1}{\sqrt{d\pi}} (2\pi e)^\frac{d}{2} d^{-\frac{d}{2} \frac{d^2-1}{d^2}}
 \\
& = 2^{ 1 + \frac{m}{d} - m } d^{ 1 + \frac{d-1}{2d} } \frac{1}{\sqrt{d\pi}} (2\pi e)^\frac{d}{2} d^{-\frac{d}{2} \frac{d^2-1}{d^2}} \\ 
& = 2^{ 1 + \frac{m}{d} - m } d^{ 1 + \frac12 - \frac{1}{2d} + \frac{1}{2d} - \frac{d}{2} } \frac{1}{\sqrt{d\pi}} (2\pi e)^\frac{d}{2}  \\
& = 2^{ 1 + \frac{m}{d} - m } d^{ 1  - \frac{d}{2} } \frac{1}{\sqrt{\pi}} (2\pi e)^\frac{d}{2} 
= \frac{2^{1 +\frac{m}{d} - m} d^{ 1 - \frac{d}{2} } (2 \pi e )^{\nicefrac{d}{2}} } { \sqrt{\pi} }\,.
\end{split}
\end{equation}
This, \eqref{eq:sob-const2-m}, \eqref{eq:sob-2m},
and
\eqref{eq:sob-const2-pf0}
therefore yield that  for all $i \in I$ we have that
\begin{equation}
\label{eq:sob-const2-pf3}
\begin{split}
  \left[\int_{V_i} \|x\|^{(1-d)\frac{d^2}{d^2-1}} \, dx \right]
  ^{\frac{d^2-1}{d^2}}
& \leq
  \frac{
  2^{1 +\frac{m}{d} }
  d^{ 1 - \frac{d}{2} }
  }
  { \sqrt{\pi} }
\\ & \leq
  \frac{
  2^{1 + \frac{\ln(2) + \ln(\pi) + 1}{2\ln(2)} + \frac{1}{d}}
  d^{ 1 - \frac{d}{2} }
  }
  { \sqrt{\pi} }
=
  { 2^{1+\frac{1}{d}} \sqrt{2e }  d^{ 1 - \frac{d}{2} }}
.
\end{split}
\end{equation}
Next note that \eqref{eq:sob-const2-pf1} ensures that for all $i \in I$ we have that
\begin{equation}
\frac{\rho_i^d}{ \lambda(W_i)d}
= d^{\frac{d}{2}-1} 2^{-dm} 2^{dm} = d^{\frac{d}{2}-1} 
.
\end{equation}
Combining this with \eqref{eq:sob-const2-pf3} yields that
\begin{equation}
\max_{i \in I} (D_i) 
=
\max_{i \in I} \!\left(\frac{\rho_i^d}{\lambda(W_i)d}
\left[\int_{V_i} \|x\|^{(1-d)\frac{d^2}{d^2-1}} \, dx \right]
^{\frac{d^2-1}{d^2}} \right)
\leq
2^{1+\frac{1}{d}} \sqrt{ 2  e }
.
\end{equation}
Combining this 
and \eqref{eq:sob-const2-pf2}
with
the hypothesis that $ d \in \{2,3,\ldots\}$
yields that
\begin{equation}
\label{eq:sob-const2-max}
\begin{split}
&\max \left\{
  \max_{i \in I} \!\left([\lambda(W_i)]^{-\frac{1}{d^2}}\right)
  ,
  \max_{i \in I} (D_{i})
  \right\}
\\ &   =
  \max \left\{ 
   2^{\frac{1}{d}}\sqrt{ 2 \pi e }
   ,
  2^{1+\frac{1}{d}} \sqrt{ 2 e }
  \right\} 
\\ & = 
  2^{1 + \frac{1}{d}}\sqrt{2e}
\leq
  2 \sqrt{2}\sqrt{2e}
=
4\sqrt{e}
.
\end{split}
\end{equation}
Next note that
\eqref{eq:sob-const2-wivi}
ensures that 
for all $ i \in I $, $ j \in I \backslash \{i\} $
we have that
\begin{equation}
W_i \cap W_j = \emptyset
\quad
\text{and}
\quad
[0,1]^d = \cup_{i \in I} \overline{W_i}
.
\end{equation}
This, 
\eqref{eq:sob-const2-max},
\eqref{eq:sob-const2-pf1},
Proposition \ref{prop:sobolev-const1}  
(applied with
$p \leftarrow d^2$,
$ I \leftarrow I $,
$ W \leftarrow (0,1)^d$,
$ W_i \leftarrow W_i $
for 
$ i \in I $
in the notation of Proposition \ref{prop:sobolev-const1}),
and
the hypothesis that $ d \in \{2,3,\ldots\}$
hence
imply that
\begin{equation}
\begin{split}
\sup_{ x \in (0,1)^d } \left| \Phi(x) \right|
&  \leq
  2^{1-\frac{1}{d^2}}
  \max\left\{
  \max_{i \in I} \!\left([\lambda(W_i)]^{-\frac{1}{d^2}}\right)
  ,
  \max_{i \in I} (D_{i})
  \right\}
\\ &  \cdot 
  \left[ 
  \int_{ (0,1)^d }
  \left(
  \left|
  \Phi(x)
  \right|^{d^2}
+
  \left\|
  ( \nabla \Phi )( x )
  \right\|^{d^2}
  \right)
  dx
  \right]^{ \nicefrac{ 1 }{ d^2} }
\\ & \leq
  8 \sqrt{e}
  \left[ 
  \int_{ (0,1)^d }
  \left(
  \left|
  \Phi(x)
  \right|^{d^2}
+
  \left\|
  ( \nabla \Phi )( x )
  \right\|^{d^2}
  \right)
  dx
  \right]^{ \nicefrac{ 1 }{ d^2} }
.
\end{split}
\end{equation}
This completes the proof of Corollary~\ref{cor:sob-const2}.
\end{proof}
\begin{cor}
\label{cor:sob-const3} 
Let  
$ d \in \N$,
$ \Phi \in C^1((0,1)^d,\R) $,
let $\left\|  \cdot \right \| \colon \R^d \to [0,\infty)$ be the standard norm on $\R^d$,
and assume that 
$
\int_{ (0,1)^d }
(
|
\Phi(x)
|^{\max\{2,d^2\}}
+
\|
( \nabla \Phi )( x )
\|^{\max\{2,d^2\}}
)
\, 
dx
<
\infty
$. Then 
\begin{equation}
\label{eq:sob-const3}
\sup_{ x \in (0,1)^d }
\left| \Phi(x) \right|
\leq
8\sqrt{e}
\left[ 
\int_{ (0,1)^d }
\left(
\left|
\Phi(x)
\right|^{\max\{2,d^2\}}
+
\left\|
( \nabla \Phi )( x )
\right\|^{\max\{2,d^2\}}
\right)
dx
\right]^{ \nicefrac{ 1 }{ \max\{2, d^2\} } }
.
\end{equation}
\end{cor}
\begin{proof}[Proof of Corollary~\ref{cor:sob-const3}]
To establish \eqref{eq:sob-const3} we distinguish between the case $d=1$ and the case $d \in \{2,3,\ldots\}$. 
First, we consider the case $d=1$. 
Note that Lemma~\ref{lem:sob-1d} ensures 
that
\begin{equation}
\begin{split}
\sup_{ x \in (0,1) }
\left| \Phi(x) \right|
& \leq
\int_0^1
\left(
\left|
\Phi(x)
\right|
+
\left|
\Phi^{ \prime } ( x )
\right|
\right)
\,dx
\\
& \leq
8\sqrt{e}
\left[ 
\int_0^1
\left(
\left|
\Phi(x)
\right|^2
+
\left|
\Phi^{ \prime } ( x )
\right|^2
\right)
\,dx
\right]^{\nicefrac12}
.
\end{split}
\end{equation}
This establishes \eqref{eq:sob-const3} in the case $d=1$. 
Next we consider the case $d \in \{2,3,\ldots\}$. 
Note that Corollary~\ref{cor:sob-const2} verifies that 
\begin{equation}
\sup_{ x \in (0,1)^d }
\left| \Phi(x) \right|
\leq
8\sqrt{e}
\left[ 
\int_{ (0,1)^d }
\left(
\left|
\Phi(x)
\right|^{d^2}
+
\left\|
( \nabla \Phi )( x )
\right\|^{d^2}
\right)
dx
\right]^{ \nicefrac{ 1 }{ d^2 } }
.
\end{equation}
This establishes \eqref{eq:sob-const3} in the case $d \in \{2,3,\ldots\}$. 
This completes the proof of Corollary~\ref{cor:sob-const3}.
\end{proof}

\subsection{Sobolev type estimates for Monte Carlo approximations}
\label{ssec:smc}
In this subsection we provide in Lemma~\ref{lem:Sobolev2} below a Sobolev type estimate for Monte Carlo approximations. 
Lemma~\ref{lem:Sobolev2}
is one of the main ingredients in our proof of 
Lemma~\ref{lem:mc-nnet-ap} in Subsection~\ref{ssec:qee} below.

\begin{lemma}
\label{lem:Sobolev2}
Let $ d, n \in \N $, $\zeta, a \in \R $, $ b \in (a,\infty) $, $ p \in [1,\infty) $, 
let $\left\|  \cdot \right \| \colon \R^d \to [0,\infty)$ be the standard norm on $\R^d$,
let 
$
\mathfrak{K} \in (0,\infty)
$
be the $\max\{2,p\}$-Kahane--Khintchine constant
(cf. Definition~\ref{def:p2-kahane-khintchine}
and Lemma~\ref{lem:kahane-const}), 
let $ ( \Omega, \mathcal{F}, \P ) $ be a probability space, 
let $\xi_{ i } \colon \R^d \times \Omega \to \R $, 
$ i \in \{ 1, \dots, n \} $,
be i.i.d.\ random fields satisfying for all $ i \in \{ 1, \dots, n \} $, $\omega \in \Omega$ that $\xi_i(\cdot,\omega) \in C^1(\mathbb R^d, \mathbb R)$,
let $\xi \colon \R^d \times \Omega \to \R $ be the random field satisfying
for all $x\in \R^d$, $\omega \in \Omega$ that $\xi(x,\omega) = \xi_1(x,\omega)$,
assume for all $\Phi \in C^1((0,1)^d, \R)$ 
with
$
\int_{ (0,1)^d }(|
\Phi(x)
|^{\max\{2,p\}}
+
\|
( \nabla \Phi )( x )
\|^{\max\{2,p\}})
\, 
dx
<
\infty
$
that
\begin{equation}
\label{eq:conc-sob}
  \sup_{ x \in (0,1)^d }
  \left| \Phi(x) \right|
  \leq
  \zeta
  \left[ 
  \int_{ (0,1)^d }
  \left(
  \left|
  \Phi(x)
  \right|^{\max\{2,p\}}
  +
  \left\|
  (\nabla \Phi)( x )
  \right\|^{\max\{2,p\}}
  \right)
  dx
  \right]^{ \nicefrac{ 1 }{ \max\{2,p\}} }
  ,
\end{equation}
and assume that for all $ x \in [a,b]^d $ we have that 
\begin{equation}
\label{eq:conc-ass1}
\begin{split}
\inf_{ \delta \in (0,\infty) }
& \sup_{ v \in [-\delta,\delta]^d}
\E\Biggl[ 
|
\xi( x + v )
|^{ 1 + \delta } + 
\left\|
(\nabla \xi)(x+v)\right\|^{1+\delta} \,
\Biggr] < \infty
.
\end{split}
\end{equation}
Then
\begin{enumerate}[(i)]
\item \label{it:conc1}
we have that
\begin{equation}
\sup_{ x \in [a,b]^d }
\left|
\E\big[ 
\xi( x )
\big]
-
\frac{ 1 }{ n }
\left(
\sum_{ 
  i = 1
}^{ n }
\xi_{ i }( x ) 
\right)
\right|^{ p }
\end{equation}
is a random variable and
\item \label{it:conc2}
we have that
\begin{equation}
\begin{split}
\Bigg(
  \E  \Bigg[ 
  \sup_{ x \in [a,b]^d } &
  \left|
  \E\big[ 
\xi( x )
  \big]
  -
  \frac{ 1 }{ n }
  \left(
  \sum_{ 
    i = 1
  }^{ n }
\xi_{ i }( x ) 
  \right)
  \right|^{ p }
  \Bigg]
  \Bigg)^{ \!\! \nicefrac{ 1 }{ p } }
\\  \leq 
  \frac{ 
  4 \mathfrak{K} \zeta
  }{
  \sqrt{ n }
  }
  \Bigg( &
  \sup_{ x \in [a,b]^d }
  \Bigg[
  \big(
  \E\big[
  |\xi( x )|^{ \max\{2,p\} }
  \big]
  \big)^{ \nicefrac{1}{\max\{2,p\} } } 
\\ & +
  ( b - a )
  \left|
  \E\Bigg[
  \left\|
(\nabla \xi)(x)  \right\|^{\max\{2,p\}}  
  \Bigg]  \right|^{ \nicefrac{1}{\max\{2,p\} } }
  \Bigg]
  \Bigg)
  .
\end{split}
\end{equation}
\end{enumerate}
\end{lemma}
\begin{proof}[Proof of Lemma~\ref{lem:Sobolev2}] 
Throughout this proof 
let $ \scp{\cdot}{\cdot}{} \colon \R^d \times \R^d \to \R$  be the $d$-dimensional Euclidean scalar product,
let $ q \in [2,\infty) $ satisfy that $ q = \max\{2,p\} $,
let
$
\rho \colon \mathbb R^d \to \mathbb R^d 
$
 satisfy for all $ x = ( x_1, \ldots, x_d ) \in \mathbb R^d $ that
\begin{equation}
\label{eq:conc-rho}
\rho( x ) 
=
\big(
( b - a ) x_1 + a ,
( b - a ) x_2 + a ,
\dots ,
( b - a ) x_d + a 
\big) 
,
\end{equation}
let $e_1,\ldots,e_d \in \R^d$ satisfy  that
$ e_1 = (1,0,\ldots,0), \ldots, e_d = (0,\ldots,0,1)$,
let 
$
Y \colon [0,1]^d \times \Omega \to \R
$
be the random field which satisfies for all $ x \in [0,1]^d $
that
\begin{equation}
\label{eq:conc-defy}
Y(x) 
=
\E\big[ 
\xi( \rho( x ) ) 
\big]
-
\frac{ 1 }{ n }
\left(
\sum_{ 
  i = 1
}^{ n }
\xi_{ i }( \rho( x ) ) 
\right)
,
\end{equation}
let 
$
Z \colon [0,1]^d \times \Omega \to \R^d
$
be the random field which satisfies for all $ x \in [0,1]^d $
that
\begin{equation}
\begin{split}
Z(x) 
=
\left( b - a \right)
\Bigg[
\E & \left[ 
(\nabla \xi)(\rho(x))
\right]  -
\frac{ 1 }{ n }
\left(
\sum_{i = 1}^n
(\nabla \xi_i)(\rho(x))
\right)
\Bigg]
,
\end{split}
\end{equation}
and let 
$ E \colon \Omega \to [0,\infty) $
be the random variable given by
\begin{equation}
\label{eq defn of E}
E =
\sup_{ x \in [a,b]^d \cap \Q^d }
\left|
\E\big[ 
\xi( x )
\big]
-
\frac{ 1 }{ n }
\left(
\sum_{ 
  i = 1
}^{ n }
\xi_{ i }( x ) 
\right)
\right|
.
\end{equation}
Note that~\eqref{eq:conc-rho} ensures that $\rho([0,1]^d)=[a,b]^d$.
Furthermore, note that~\eqref{eq:conc-ass1} 
ensures that for all $x \in [a,b]^d$
there exists $\delta_x \in (0,\infty) $ such that
\begin{equation}
\label{eq unif int at x}
\sup_{v \in [-\delta_x,\delta_x]^d } 
\E\left[\left\|
(\nabla \xi)(x+v)
\right\|^{ 1 + \delta_x }
\right] < \infty.
\end{equation}
H\"older's inequality therefore verifies that  for all 
$ x \in [a,b]^d$ there exists $\delta_x \in (0,\infty)$ 
which satisfies that
\begin{equation}
\label{eq:conc-ubound0}
\begin{split}
\sup_{v \in [-\delta_x,\delta_x]^d } &
\E\left[\left\|
(\nabla \xi)(x+v)
\right\|
\right] \\
& \leq
\sup_{v \in [-\delta_x,\delta_x]^d } 
\left(\E\left[\left\|
(\nabla \xi)(x+v)
\right\|^{ 1 + \delta_x }
\right]\right)^{\nicefrac{1}{(1+\delta_x)}}
<
\infty.
\end{split}
\end{equation}
Next note that the collection $S_x = \{y\in \mathbb [a,b]^d : y - x \in (-\delta_x,\delta_x)^d\}$, $x\in [a,b]^d$, is an open
cover of $[a,b]^d$.
The fact that $[a,b]^d$ is compact hence ensures that there exists 
$N\in \N$ and
$x_k \in [a,b]^d$, $k \in \{1,\ldots,N\}$ which satisfies that the collection $S_{x_k}$, 
$k\in\{1,\ldots,N\}$ is a finite
open cover of  $[a,b]^d$.
Combining this with~\eqref{eq:conc-ubound0} yields that 
\begin{equation}
\begin{split}
\label{eq:conc-ubound}
\sup_{x \in [a,b]^d } &
\E\left[\left\|
(\nabla \xi)(x)
\right\|
\right] \\
& \leq \max_{k\in\{1,\ldots,N\}} 
\sup_{v \in (-\delta_{x_k},\delta_{x_k})^d } 
\E\left[\left\|
(\nabla \xi)(x_k+v)
\right\|
\right] 
 < \infty\,.
\end{split}
\end{equation}
Moreover, note that 
the fact that for all $ \omega \in \Omega $ 
we have that the functions
$ [a,b]^d \ni x \mapsto \xi(x,\omega) \in \R $
and 
$ [a,b]^d \ni x \mapsto (\nabla \xi)(x,\omega) \in \R^d$
are continuous
ensures that 
$
[a,b]^d \times \Omega \ni (x,\omega) \mapsto \xi(x,\omega) \in \R 
$
and
$
[a,b]^d \times \Omega \ni (x,\omega) \mapsto (\nabla \xi)(x,\omega) \in \R^d 
$
are Carath\'eodory functions. 
This implies that 
$
[a,b]^d \times \Omega \ni (x,\omega) \mapsto \xi(x,\omega) \in \R 
$
is
$(\mathcal{B}([a,b]^d) \otimes \mathcal{F}) / \mathcal{B}(\R)$-measurable
and
$
[a,b]^d \times \Omega \ni (x,\omega) \mapsto (\nabla \xi)(x,\omega) \in \R^d 
$
is 
$(\mathcal{B}([a,b]^d) \otimes \mathcal{F}) / \mathcal{B}(\R^d)$-measurable, see, e.g., Aliprantis and Border~\cite[Lemma 4.51]{aliprantis:border:infinite}).
Next note that 
the fundamental theorem of calculus
ensures that 
for all   
$ x,y \in \mathbb R^d $ we have that
\begin{equation}
\label{eq:conc-ftc}
\begin{split}
\xi(x) - \xi(y) = \int_0^1 \left\langle(\nabla \xi)\left(y+t(x-y)\right),x-y\right\rangle\, dt  .
\end{split}
\end{equation} 
This reveals that for all
$ x,y \in [a,b]^d $ 
we have that
\begin{equation}
\begin{split}
| \xi(x) - \xi(y) |
& \leq
\left\|x-y\right\|  \int_0^1 \left\|(\nabla \xi)\left(y+t(x-y)\right)\right\|\, dt.
\end{split}
\end{equation} 
Combining this with 
Fubini's theorem 
verifies that  for all $ x,y \in [a,b]^d $ we have that
\begin{equation}
\label{eq:conc-lip2}
\begin{split}
\big| \E[\xi(x)]  -  \E[\xi(y)] \big|
& \leq  \E\big[|\xi(x) - \xi(y)|\big] \\
& \leq \left\|x-y\right\| \E \left[
  \int_0^1 \left\|(\nabla \xi)\left(y+t(x-y)\right)\right\|\,dt\, \right] \\
& = \left\|x-y\right\| 
  \int_0^1 \E \left[ \left\|(\nabla \xi)\left(y+t(x-y)\right)\right\|\, \right] \,dt\\
& \leq \left\|x-y\right\|  
  \sup_{v \in [a,b]^d} \E \left[ \left\|(\nabla \xi)\left(v\right)\right\|\right]\,.
\end{split}
\end{equation}
This and \eqref{eq:conc-ubound} prove that
$ [a,b]^d \ni x \mapsto \E[\xi(x)] \in \R $ 
is a Lipschitz continuous function.
Hence, we obtain
for all $ \omega \in \Omega $ that
$ [0,1]^d \ni x \mapsto Y(x,\omega) \in \R $ is 
a continuous function.
Combining this with~\eqref{eq defn of E} implies that
\begin{equation}
\label{eq:conc-E}
E =
\sup_{ x \in [0,1]^d }
\left|
Y( x )
\right|
.
\end{equation}
This establishes item \eqref{it:conc1}.
Next note that~\eqref{eq:conc-rho} implies that
for all $ j \in \{1,\ldots,d\} $,
$ x \in [0,1]^d $, 
$h\in \R$ 
we have that
\begin{equation}
\rho(x+he_j) - \rho(x) = (b-a)he_j\,.	
\end{equation}
This and~\eqref{eq:conc-ftc} verify that   
for all $ j \in \{1,\ldots,d\} $,
$ x \in [0,1]^d $,
$h\in \R\setminus\{0\}$ 
we have that 
\begin{equation}
\label{eq dir der with rho}
\frac{\xi(\rho(x+he_j))-\xi(\rho(x))}{h}
= (b-a)\int_0^1 \left\langle(\nabla \xi)\big(\rho(x) + t(b-a)he_j\big), e_j \right\rangle \,dt\,.	
\end{equation} 
Moreover, note that~\eqref{eq:conc-ass1} implies
that for all $x\in [0,1]^d$ there exists $\delta_x \in  (0,\infty)$ 
such that
\begin{equation}
\label{eq unif int at rho x}
\sup_{v \in [-\delta_x,\delta_x]^d } 
\E\left[\left\|
(\nabla \xi)(\rho(x)+v)
\right\|^{ 1 + \delta_x }
\right] < \infty.
\end{equation}
This, H\"older's inequality, and Fubini's theorem verify that  for all
$ j \in \{1,\ldots,d\} $, 
$ x \in [0,1]^d $
there exists $\delta_x \in (0, \infty)$ such that for 
all $h \in \{h' \in \mathbb R : |(b-a)h'| < \delta_x \}$
we have that
\begin{equation}
\begin{split}
& \E\left[\left\|\int_0^1 (\nabla \xi)\big(\rho(x) + t(b-a)h e_j \big)\,dt \right\|^{1+\delta_x}\right] \\
& \leq	\E\left[\int_0^1 \left\|(\nabla \xi)\big(\rho(x) + t(b-a)h e_j \big)\right\|^{1+\delta_x}\,dt \right] \\
& = \int_0^1 \E \left[ \left\|(\nabla \xi)\big(\rho(x) + t(b-a)h e_j \big)\right\|^{1+\delta_x}\right]\,dt  \\
& \leq \sup_{v\in [-\delta_x,\delta_x]^d} \E \left[ \left\|(\nabla \xi)\big(\rho(x)+v\big)\right\|^{1+\delta_x}\right] < \infty\,.
\end{split}
\end{equation}
This and~\eqref{eq dir der with rho}
verify that  for all
$ j \in \{1,\ldots,d\} $, 
$ x \in [0,1]^d $
there exists $\delta_x \in (0, \infty)$ such that for 
all $h \in \{h' \in \mathbb R\setminus\{0\} : |(b-a)h'| < \delta_x \}$
we have that
\begin{equation}
\begin{split}
& \E\left[\left|\frac{\xi(\rho(x+he_j))-\xi(\rho(x))}{h}\right|^{1+\delta_x}\right] \\
& \leq (b-a)^{1+\delta_x} \E\left[\left\|\int_0^1 (\nabla \xi)\big(\rho(x) + t(b-a)h e_j \big)\,dt \right\|^{1+\delta_x}\right]\\
& \leq (b-a)^{1+\delta_x}  
\sup_{v\in [-\delta_x,\delta_x]^d} \E \left[ \left\|(\nabla \xi)\big(\rho(x)+v\big)\right\|^{1+\delta_x}\right]
<\infty.	
\end{split}
\end{equation}
This, the theorem of de la Vall\'ee--Poussin (see, e.g.,~\cite[Theorem 6.19]{klenke:probability}), and Vitali's convergence theorem (see, e.g.,~\cite[Theorem 6.25]{klenke:probability})
verify that  for all
$ x \in [0,1]^d $,
$ j \in \{1,\ldots,d\} $
there exists $\delta_x \in (0, \infty)$ such that for all
$ (h_m)_{ m \in \N } \subseteq \{h' \in \mathbb R\setminus\{0\} : |(b-a)h'| < \delta_x \} $ 
with $ \lim_{ m \to \infty} h_m = 0 $
we have that
\begin{equation}
\begin{split}
& \lim_{ m \to \infty } 
\E\left[
\frac{\xi(\rho(x+h_m e_j)) - \xi(\rho(x))}{h_m}
\right]
\\ & =
\E\!\left[
\lim_{ m \to \infty } 
\frac{\xi(\rho(x+h_m e_j)) - \xi(\rho(x))}{h_m}
\right]
.
\end{split}
\end{equation}
Therefore, we obtain that
for all 
$ x \in [0,1]^d $,
$ j \in \{1,\ldots,d\} $
there exists $\delta_x \in (0, \infty)$ such that for all
$ (h_m)_{ m \in \N } \subseteq \{h' \in \mathbb R\setminus\{0\} : |(b-a)h'| < \delta_x \} $ 
with $ \lim_{ m \to \infty} h_m = 0 $
we have that
\begin{equation}
\label{eq:xi rho gradient}
\begin{split}
\lim_{m\to \infty}
 & \E \left[
\frac{\xi(\rho(x+h_m e_j)) - \xi(\rho(x))}{h_m}
\right] =
(b-a)\E \Big[
\langle (\nabla \xi) (\rho(x) ),e_j\rangle \Big].
\end{split}
\end{equation}
Furthermore, the theorem of de la Valle\'e--Poussin, Vitali's convergence theorem, and~\eqref{eq unif int at rho x} prove that for all $x \in [0,1]^d$, $j\in\{1,\ldots,d\}$ we have that 
\begin{equation}
\limsup_{\R^d \setminus \{0\}\ni h \to 0} | \E[\langle (\nabla \xi)(\rho(x)+h),e_j\rangle ] - \E[\langle (\nabla \xi)(\rho(x)),e_j\rangle] | = 0.	
\end{equation}
This and~\eqref{eq:xi rho gradient} imply that for all $\omega \in \Omega$, $x \in (0,1)^d$ we have that $((0,1)^d \ni y \mapsto Y(y,\omega) \in \R) \in C^1((0,1)^d,\R)$ and $(\nabla Y)(x,\omega) = Z(x,\omega)$.
Combining this, 
\eqref{eq:conc-sob},
and
\eqref{eq:conc-E}
yields that
\begin{equation}
\label{eq:conc-sobE}
\begin{split}
  E
  =
  \sup_{ x \in (0,1)^d }
  \left|
  Y( x )
  \right|
& \leq
  \zeta
  \left[ 
  \int_{ (0,1)^d }
  \left(
  \left|
  Y(x)
  \right|^{ q }
  +
  \left\|
  Z( x )
  \right\|^{ q }
  \right)
  dx
  \right]^{ \nicefrac{ 1 }{ q } }
  .
\end{split}
\end{equation}
Next observe that
\eqref{eq:conc-sob}
ensures that $ \zeta \in [0,\infty) $.
H\"older's inequality,
\eqref{eq:conc-sobE},
and
Fubini's theorem
hence verify that 
\begin{equation}
\label{eq:estimate_E}
\begin{split}
\big(
  \E\big[ 
  | E |^p
  \big]
  \big)^{ \nicefrac{1}{p} }
\leq 
  \big(
  \E\big[ 
  | E |^q
  \big]
  \big)^{ \nicefrac{1}{q} }
& \leq
\zeta
\left(
\E\!\left[
\int_{ (0,1)^d }
\left|
Y(x)
\right|^q
+
\left\|
Z( x )
\right\|^q
dx
\right]
\right)^{\! \nicefrac{ 1 }{ q } }
\\ & =
\zeta
\left[ 
\int_{ (0,1)^d }
\E\!\left[
\left|
 Y(x)
\right|^q
+
\left\|
Z( x )
\right\|^q
\right]
dx
\right]^{ \nicefrac{ 1 }{ q } }
\\ & \leq
\zeta
\left[
\sup_{ x \in [0,1]^d }
\E\!\left[
\left|
Y(x)
\right|^q
+
\left\|
Z( x )
\right\|^q
\right]
\right]^{ \nicefrac{ 1 }{ q } }
.
\end{split}
\end{equation}
Next note that 
\eqref{eq:conc-ass1} and 
Proposition~\ref{prop:kahane} 
prove that 
for all $ x \in [0,1]^d $ we have that
\begin{equation}
\begin{split}
\left(
\E\big[ 
|
Y(x)
|^q
\big]
\right)^{ \nicefrac{1}{q} }
\leq
\frac{ 
2 \, \mathfrak{K}}{ \sqrt{ n } }
\big(
\E\big[ 
|
\xi( \rho( x ) )
  -
  \E[ 
\xi( \rho( x ) ) 
  ]
|^q
\big]
\big)^{ \nicefrac{1}{q} }
\end{split}
\end{equation}
and
\begin{equation}
\begin{split}
&
\left(
\E\big[ 
\|
Z(x)
\|^q
\big]
\right)^{ \nicefrac{1}{q} }  \leq
\frac{ 
2 
\,
\mathfrak{K}
\left( b - a \right)
}{ \sqrt{ n } }
\left(
\E\left[ 
\left\|
(\nabla \xi) (\rho(x))
  -
  \E\Big[ 
(\nabla \xi) (\rho(x))
  \Big]
\right\|^q
\right]
\right)^{ \nicefrac{1}{q} }
.
\end{split}
\end{equation}
This and \eqref{eq:estimate_E} imply that
\begin{equation}
\begin{split}
\big(
\E\big[ 
 | E |^p
\big]
\big)^{ \nicefrac{1}{p} }
 \leq &
\frac{ 
2 \mathfrak{K}   \zeta
}{
\sqrt{ n }
}
\Bigg[
\sup_{ x \in [0,1]^d }
\Big(
\E\Big[
\big|
\xi( \rho( x ) )
  -
  \E\big[ 
\xi( \rho( x ) )
  \big]
\big|^q
\\ & +
( b - a )^q
\,
\big\|
(\nabla \xi)(\rho(x) )
  -
  \E\big[ 
(\nabla \xi) (\rho(x) )
  \big]
\big\|^q
\Big]
\Big)
\Bigg]^{ \! \nicefrac{ 1 }{ q } }
.
\end{split}
\end{equation}
Hence, we obtain that
\begin{equation}
\begin{split}
\big(
\E\big[ 
| E |^p
\big]
\big)^{ \nicefrac{1}{p} }
\leq  &
\frac{ 
2 \mathfrak{K}  \zeta
}{
\sqrt{ n }
}
\Bigg[  
\sup_{ x \in [a,b]^d }
\bigg(
\E\bigg[
| \xi(x) - \E[\xi(x)] |^q \\ 
& +( b - a )^q
\,
\left\|
(\nabla \xi) (x )
  -
  \E[ 
(\nabla \xi) (x )
  ]
\right\|^q
\bigg]
\bigg)
\Bigg]^{ \! \nicefrac{ 1 }{ q } }
.
\end{split}
\end{equation}
The fact that 
for all $ r,s \in [0,\infty)$ we have that
$
  (r+s)^{\nicefrac{1}{q}} \leq r^{ \nicefrac{1}{q} } + s^{ \nicefrac{1}{q} }
$
and the triangle inequality
therefore
yield that 
\begin{equation}
\begin{split}
\big(
  \E\big[ 
  | E |^p
  \big]
  \big)^{ \nicefrac{1}{p} }
\leq &
  \frac{ 
  2\mathfrak{K}  \zeta
  }{
  \sqrt{ n }
  }
  \Bigg(
  \sup_{ x \in [a,b]^d }
  \bigg[
  \Big(
  \E\Big[
  |
    \xi( x )
    -
    \E[\xi( x )]
  |^q
  \Big]
  \Big)^{ \nicefrac{1}{q} }
  \\ &   +
  ( b - a )
  \,
  \big(
  \E\big[
  \|
    (\nabla \xi) (x )
 - \E\big[(\nabla \xi) (x )\big]
  \|^q
  \big]
  \big)^{ \nicefrac{1}{q} }
  \bigg]
  \Bigg) \\ 
  \leq &
  \frac{ 
  4 \mathfrak{K}  \zeta
  }{
  \sqrt{ n }
  }
  \Bigg(
  \sup_{ x \in [a,b]^d }
  \bigg[
  \big(
  \E\big[
  |
    \xi( x )
  |^q
  \big]
  \big)^{ \nicefrac{1}{q} }
  +
  ( b - a )
  \,
  \big(
  \E\big[
  \|
    (\nabla \xi)(x ) 
  \|^q
  \big]
  \big)^{ \nicefrac{1}{q} }
  \bigg]
  \Bigg)
  .
\end{split}
\end{equation}
This establishes item \eqref{it:conc2}.
This completes the proof of Lemma~\ref{lem:Sobolev2}.
\end{proof}

\section[Stochastic differential equations with affine coefficients]{Stochastic differential equations with affine coefficient functions}
\label{sec:sde}

\subsection{A priori estimates for Brownian motions}
In this subsection we provide in Lemma~\ref{l:exp.Gauss} below essentially well-known a priori estimates for standard Brownian motions.
Lemma~\ref{l:exp.Gauss} 
will be employed in our proof of 
Corollary~\ref{cor:apriori1} in Subsection~\ref{ssec:apriori1}
below. Our proof of Lemma~\ref{l:exp.Gauss} is a slight adaption 
of the proof of  Lemma~2.5 in 
Hutzenthaler et al.~\cite{MR3766391}.
\begin{lemma}
\label{l:exp.Gauss}
Let
$ d, m \in \N$,
$ T \in [0, \infty ) $,
$ p \in (0,\infty)$,
$ A \in \R^{d \times m} $, 
let 
$
\left \| \cdot \right \| \colon \R^d \to [0,\infty)
$ 
be the standard norm on $\R^d$,
let
$
( \Omega, \mathcal{F}, \P ) 
$
be a probability space,
and let
$
W \colon [0,T] \times \Omega \to \R^m
$
be a standard Brownian motion. 
Then it holds for all $ t \in [0,T] $ that 
\begin{equation}
\label{eq:bm-lp}
\begin{split}
\big( \E\big[ \| A       W_t \|^p \big] \big)^{ 1 / p }
& \leq 
\sqrt{ \max\{1,p-1\} \operatorname{Trace}(A^{ \ast } A) \, t} 
.
\end{split}
\end{equation}
\end{lemma}

\begin{proof}[Proof of Lemma~\ref{l:exp.Gauss}]
Throughout this proof 
for every $n \in \N$ let
$
\left \| \cdot \right \|_{\R^n} \colon \R^n \linebreak \to [0,\infty)
$
be the standard norm on $\R^n$,
let $(q_r)_{r \in [0,\infty) } \subseteq \N_0$ satisfy for all $r \in [0,\infty)$ that
$
q_r = \max ( \N_0 \cap [0,\nicefrac{r}{2}] )
$,
let
$ f_r \colon \R^m \to \mathbb R$,
$ r \in [0,\infty) $,
 satisfy for all 
$
r \in [0,\infty)
$, 
$ 
x \in \R^m 
$  
that
\begin{equation}
\label{eq:bm-deff}
f_r(x) = \| A x \|^{ r } _{\R^d },
\end{equation}
and
let $ \beta^{ (i) } \colon [0,T]\times \Omega \to \R $, 
$ i \in \{ 1, \ldots, m \} $, 
be the stochastic processes which
satisfy for all 
$ t \in [0,T] $ that
\begin{equation}
W_t = 
\big( 
\beta^{ (1) }_t, \ldots, \beta^{ (m) }_t 
\big).
\end{equation}
Note that for all $ r \in  [2,\infty) $, $x \in \R^m$ we have that
\begin{equation}
\label{eq:exp.Gauss.grad}
(\nabla f_r)(x) = r \| A x \|_{\R^d } ^{ r-2 } A^{ \ast } Ax.
\end{equation}
This implies that for all $ r \in  [2,\infty)$, $x \in \R^m$ we have that
\begin{equation}
( \operatorname{Hess} f_r )( x ) 
= 
r
\,
\| A x \|^{ ( r - 2) }_{\R^d} A^* A
+
\mathbbm{1}_{ \{  A x \neq 0 \} }
r (r-2)
\| A x \|^{ (  r - 4 ) }_{ \R^d }
\left( A^* A x \right)
\left( A^* A x \right)^*.
\end{equation}
The fact that for all $B \in \R^{ m \times d }$, $x \in \R^d$ we have that
$
\|Bx\|_{\R^m}^2 \leq \operatorname{Trace}(B^{ \ast } B) \| x \|_{\R^d}^2
$
and 
$
\operatorname{Trace}(B^{ \ast } B) = \operatorname{Trace}(B B^{ \ast })
$
hence verifies that  for all  $ r \in  [2,\infty) $, $ x \in \R^m $ we have that
\begin{equation}
\label{eq:exp.Gauss.Hess}
\begin{split}
&
\operatorname{Trace}\!\left(
( \operatorname{Hess} f_r )( x )
\right)
\\ & =
\operatorname{Trace}\!\left(
r
\,
\| A x \|_{ \R^d } ^{ ( r - 2 ) } A^* A
+
\mathbbm{1}_{ \{ A x \neq 0 \} }
\,
r \left( r - 2 \right)
\| A x \|_{ \R^d }^{ ( r - 4 ) }
\left( A^* A x \right)
\left( A^* A x \right)^*
\right)
\\ & =
r
\left\| A x \right\|_{ \R^d }^{ ( r - 2 ) }
\operatorname{Trace}(A^{ \ast } A)
+
\mathbbm{1}_{ \{ A x \neq 0 \} } \, r
\left( r - 2 \right)
\left\|
A x
\right\|_{ \R^d }^{ ( r - 4 ) }
\left\|
A^* A x
\right\|_{\R^m }^2
\\ & \leq
r
\left\| A x \right\|_{ \R^d }^{ ( r - 2 ) }
\operatorname{Trace}(A^{ \ast } A)
+
\mathbbm{1}_{ \{ A x \neq 0 \} }r \left( r - 2 \right)
\left\| A x \right\|_{ \R^d }^{ ( r - 4 ) }
\operatorname{Trace}(A A^{ \ast })
\left\| A x \right\|_{ \R^d }^2
\\ & =
r
\left\| A x \right\|_{ \R^d }^{ ( r - 2 ) }
\operatorname{Trace}(A^{ \ast } A)
+
r \left( r - 2 \right)
\left\| A x \right\|_{ \R^d }^{ ( r - 2 ) }
\operatorname{Trace}(A^{ \ast } A)
\\ & =
r \left( r - 1 \right)
\operatorname{Trace}(A^{ \ast } A)
\,
f_{ r - 2 }( x )
\, .
\end{split}
\end{equation}
Moreover, note that the fact that 
$W\colon [0,T] \times \Omega \to \R^m$ 
is a stochastic process with continuous sample paths (w.c.s.p.)
ensures that
$W \colon [0,T] \times \Omega \to \R^m$ is a $(\mathcal{B}([0,T]) \otimes \mathcal{F}) / \mathcal{B}(\R^m)$-measurable function.
The fact for all $r \in [2,\infty)$ we have that $f_r \in C^2(\R^m, \R)$
hence
implies that for all
$r \in [2,\infty)$,
$ i \in \{1,\ldots,m\}$
we have that
\begin{equation}
\label{eq:bm-prodm}
[0,T] \times \Omega \ni (t,\omega) \mapsto  (\tfrac{ \partial }{ \partial x_i } f_r)(W_t(\omega)) \in \R
\end{equation}
is a $(\mathcal{B}([0,T]) \otimes \mathcal{F}) / \mathcal{B}(\R)$-measurable function.
Combining this
and
\eqref{eq:exp.Gauss.grad}
yields that for all
$ r \in [2,\infty) $,
$ i \in \{1,\ldots,m\}$
we have that
\begin{equation}
\label{eq:bm-l2}
\begin{split}
\int_{ 0 }^T
\E\!
\left[
\big| (\tfrac{ \partial }{ \partial x_i } f_r)(W_t) \big|^2
\right]
dt
& \leq
\int_{ 0 }^T
\E\!
\left[
\left\| ( \nabla f_r)(W_t) \right\|_{ \R^m }^2
\right]
dt
\\ & =
\int_{ 0 }^T
\E\!
\left[
r^2
\left\| A W_t \right \|_{ \R^d } ^{ 2r-4 } \left\|  A^{ \ast } AW_t \right\|_{ \R^m } ^2
\right]
dt
\\ & \leq
\int_{ 0 }^T
\E\!
\left[
r^2  \left\| A^{ \ast } \right \|_{ L(\R^d, \R^m) }^{ 2 }
\left \| A W_t \right \|_{ \R^d } ^{ 2r-2 }
\right]
dt
\\ & \leq
\int_{ 0 }^T
\E\!
\left[
r^2  \operatorname{Trace}(A^{ \ast } A )
\left \| A W_t \right \|_{ \R^d } ^{ 2r-2 }
\right]
dt
\\ & \leq
r^2  \operatorname{Trace}(A^{ \ast } A )
\,
T
\left(
\sup_{ t \in [0,T] }
\E\!
\left[
\left \| AW_t \right \|_{ \R^d } ^{ 2r-2 }
\right]
\right)
.
\end{split}
\end{equation}
Next note that the fact that for all $ r \in [2,\infty) $ we have that $ 2r-2 \in [2,\infty) $
ensures that for all
$ r \in [2,\infty) $
we have that
\begin{equation}
\sup_{ t \in [0,T] }
\E\!
\left[
\left \| AW_t \right \|_{ \R^d } ^{ 2r-2 }
\right]
=
\bigg(\sup_{ t \in [0,T] } t^{r-1} \bigg)
\E\!
\left[
\left \| AW_1 \right \|_{ \R^d } ^{ 2r-2 }
\right]
<
\infty
.
\end{equation}
Combining this with
\eqref{eq:bm-l2}
demonstrates
that for all
$ r \in [2,\infty) $,
$ i \in \{1,\ldots,m\}$
we have that
\begin{equation}
\int_{ 0 }^T
\E\!
\left[
\big| (\tfrac{ \partial }{ \partial x_i } f_r)(W_t) \big|^2
\right]
dt
<
\infty
.
\end{equation}
This
proves that for all $r \in  [2,\infty)$, $i \in \{1, \ldots, m \}$, $t \in [0,T]$ we have that
\begin{equation}
\E\Biggl[
\int_{ 0 } ^{ t } \big(\tfrac{ \partial }{ \partial x_i} f_r\big)(W_s) \, d\beta_s^{ (i) }
\Biggr]
= 0.
\end{equation}
It\^o's formula,
Fubini's theorem,
\eqref{eq:bm-deff},
and
\eqref{eq:exp.Gauss.Hess}
hence verify that  for all
$r \in [2,\infty)$,
$t \in [0,T]$
we have that
\begin{equation}
\label{eq:BM.Moments}
\begin{split}
& \E[f_r(W_t)]
\\ & =
\E\!
\left[
f_r(W_0)
+
\sum_{i=1}^m
\!\left(
\int_{ 0 }^{ t }
\big(\tfrac{ \partial }{ \partial x_i} f_r\big)(W_s) \,  d\beta_s^{ (i) }
+
\frac{ 1 }{ 2 }
\int_{ 0 }^{ t }
\big(\tfrac{ \partial^2 }{ \partial x_i^2 } f_r \big) ( W_s) \, ds
\right)
\right]
\\ & =
\frac{ 1 }{ 2 }
\int_0^{ t }
\E\big[
\operatorname{Trace}\!\left(
(\operatorname{Hess} f_r)( W_{ s } )
\right)
\big]
\,
ds
\\ & \leq
\frac{ r(r-1) \operatorname{Trace}(A^{ \ast } A)}{ 2 }
\int_0^{ t }
\E\big[
f_{ r - 2 }(
W_{ s }
)
\big]
\,
ds
.
\end{split}
\end{equation}
This
and
\eqref{eq:bm-deff}
yield that  for all $t \in [0,T]$ we have that
\begin{equation}
\label{eq:exp.Gauss.even}
\begin{split}
\E\!\left[
\left\|
A W_{ t }
\right\|_{ \R^d }^{ 2 }
\right]
& \leq
\frac{2(2-1)}{2}
\operatorname{Trace}(A^{ \ast } A)
\int_0^{ t }
\E\big[f_0( W_{ s } )
\big]
\, ds
\\ & =
\operatorname{Trace}(A^{ \ast } A) \, t
.
\end{split}
\end{equation}
H\"older's inequality therefore
proves that for  all  $r \in [0,2)$, $t \in [0,T]$ we have that
\begin{equation}
\label{eq:exp.Gauss.q}
\begin{split}
& \E\big[f_r(W_t)\big]
=
\E\big[\| AW_t \|_{ \R^d} ^r\big]
\leq
\big(
\E\big[\| AW_t\|_{ \R^d }^2\big]
\big)^{ \nicefrac{r}{2} }
\leq
\left(
\operatorname{Trace}(A^{ \ast } A)
\, t
\right)^{\nicefrac{r}{2}}
.
\end{split}
\end{equation}
This reveals that for all $r \in (0,2]$, $t \in [0,T]$ we have that
\begin{equation}
\label{eq:exp.Gauss.case2}
\big(
\E
\big[ \| AW_t \|_{ \R^d }^r \big]
\big)^{ \nicefrac{1}{r} }
\leq
\sqrt{
\operatorname{Trace}(A^{ \ast } A)
\, t
}
.
\end{equation}
Next note that \eqref{eq:BM.Moments},
the fact that for all $r \in (2,\infty)$ we have that  $r - 2q_r \in [0,2)$,
and \eqref{eq:exp.Gauss.q} imply that for all $r \in (2,\infty)$, $s_0 \in [0,T]$ we have that
\begin{equation}
\begin{split}
& \E\!\left[
\left\|
A W_{ s_0 }
\right\|_{\R^d} ^{ r }
\right]
\leq
\frac{
\left[ \prod_{ i = 0}^{ q_{ r } -1 } (r-2i)(r-1-2i) \right]
}{
2^{ q_{ r } }
}
\left[
\operatorname{Trace}(A^{ \ast } A)
\right]^{ q_{ r } }
\\ & \cdot
\int_0^{ s_0 }
\cdots
\int_0^{ s_{ q_{ r } -1 } }
\E\big[
f_{ r - 2q_{ r }}( W_{ s_{q_{ r } } } )
\big]
\, ds_{q_{ r } }  \cdots \, ds_1
\\ & \leq
\frac{
\left[\prod_{ i = 0}^{ q_{ r } -1 } (r-2i)(r-1-2i)\right]
}{
2^{ q_{ r } }
}
\left[
\operatorname{Trace}(A^{ \ast } A)
\right]^{ q_{ r } + \frac{r - 2q_{ r }}{2} }
\\ & \cdot
\int_0^{ s_0 }
\cdots
\int_0^{ s_{ q_{ r } -1 } }
(s_{q_{ r }})^{ \frac{ r - 2q_{ r }}{2} }
\, ds_{q_{ r } }  \cdots \, ds_1
\\ & =
\frac
{
\left[\prod_{ i = 0}^{ q_{ r } -1 } (r-2i)(r-1-2i)\right]
}{
2^{q_{ r }}
}
\frac{
2^{q_{ r }}
}{
\left[\prod_{ i = 0}^{ q_{ r } - 1 } (r-2i) \right]
}
\left[
\operatorname{Trace}(A^{ \ast } A)
\right]^{ \nicefrac{r}{2}  }
s_0 ^{ \nicefrac{ r }{2}}
\\ & =
\left[\prod_{ i = 0}^{q_{ r }-1} (r-1-2i)\right]
\left[
\operatorname{Trace}(A^{ \ast } A)
\right]^{  \nicefrac{r}{2} }
s_0 ^{\nicefrac{ r }{2}}
.
\end{split}
\end{equation}
The fact that
for all $r \in (2,\infty)$ we have that
$q_r \leq \frac{r}{2}$ hence yields that for all
$ r \in (2,\infty)$,
$ t \in [0,T]$ we have that
\begin{equation}
\begin{split}
\big(
\E\!\left[
\left\|
A W_{ t }
\right\|_{ \R^d }^{ r }
\right]
\big)^{ \nicefrac{1}{r} }
& \leq
\left[\prod_{ i = 0}^{q_{ r }-1} (r-1-2i)\right]^{\nicefrac{1}{r}}
\sqrt{
\operatorname{Trace}(A^{ \ast } A)
\, t
}
\\ & \leq
(r-1)^{ \frac{q_{ r }}{r} }
\sqrt{
\operatorname{Trace}(A^{ \ast } A)
\, t
}
\\ & \leq
(r-1)^{\frac{r}{2r}}
\sqrt{
\operatorname{Trace}(A^{ \ast } A)
\, t
}
\\ & =
\sqrt{(r-1)
\operatorname{Trace}(A^{ \ast } A)
\, t
}
.
\end{split}
\end{equation}
Combining this with
\eqref{eq:exp.Gauss.case2}
establishes \eqref{eq:bm-lp}.
This completes the proof of Lemma \ref{l:exp.Gauss}.
\end{proof}
\subsection{A priori estimates for solutions}
\label{ssec:apriori1}
In this subsection we present in Lemma~\ref{lem:sde-lp-bound}
and
Corollary~\ref{cor:apriori1}
below essentially well-known a priori estimates
for solutions of stochastic differential equations
with
at most linearly growing drift coefficient functions
and
constant diffusion coefficient functions.
Corollary~\ref{cor:apriori1}
is one of the main ingredients in our proof of
Lemma~\ref{lem:mc-nnet-ap}
in Subsection \ref{ssec:qee} below
and
is a straightforward consequence
of
Lemma~\ref{l:exp.Gauss} above
and
Lemma~\ref{lem:sde-lp-bound} below.
Our proof of Lemma~\ref{lem:sde-lp-bound} is a slight adaption of the proof of Lemma~2.6 in
Beck et al.~\cite{Becker2018}. In our formulation of the statements of Lemma~\ref{lem:sde-lp-bound}
and Corollary~\ref{cor:apriori1} below we employ the elementary result in
 Lemma~\ref{lem:int of cont proc meas} below. In our proof of Lemma~\ref{lem:int of cont proc meas} we employ the elementary result in Lemma~\ref{lem:int meas} below. Lemma~\ref{lem:int meas} and Lemma~\ref{lem:int of cont proc meas} study measurability properties for time-integrals of  suitable stochastic processes.

\begin{lemma}
\label{lem:int meas}
Let $T \in [0,\infty)$, let $(\Omega, \mathcal F, \mathbb P)$ be a probability space, let $Y:[0,T] \times \Omega \to \mathbb R$ be
$(\mathcal B([0,T]) \otimes \mathcal F)/\mathcal B(\mathbb R)$-measurable, and assume that for all $\omega \in \Omega$ we have that $\int_0^T | Y_t(\omega) | \,dt < \infty$.
Then the function $\Omega \ni \omega \mapsto \int_0^T Y_t(\omega) \,dt \in \mathbb R$ is $\mathcal F/\mathcal B(\mathbb R)$-measurable.
\end{lemma}
\begin{proof}
Throughout this proof let $Y^+: [0,T] \times \Omega \to \mathbb R$ and $Y^-: [0,T] \times \Omega \to \mathbb R$ satisfy for all $t\in [0,T]$,  $\omega \in \Omega$ that
$
Y^+_t(\omega) = \max\{Y_t(\omega),0\}
$
and
$Y^-_t(\omega) = -\min\{Y_t(\omega),0\}$.
Note that for all $\omega \in \Omega$ we have that
\begin{equation}
\label{eq int meas}
\int_0^T Y_t(\omega)\,dt = \int_0^T Y^+_t(\omega)\,dt - \int_0^T Y^-_t(\omega)\,dt\,.
\end{equation}
Moreover, observe that Tonelli's theorem implies that $\Omega \ni \omega \mapsto \int_0^T Y^+_t(\omega) \,dt \in \mathbb R$ is $\mathcal F/\mathcal B(\mathbb R)$-measurable and $\Omega \ni \omega \mapsto \int_0^T Y^-_t(\omega) \,dt \in \mathbb R$ is $\mathcal F/\mathcal B(\mathbb R)$-measurable.
This and~\eqref{eq int meas} prove that
$\Omega \ni \omega \mapsto \int_0^T Y_t(\omega) \,dt \in \mathbb R$ is $\mathcal F/\mathcal B(\mathbb R)$-measurable.
This completes the proof of Lemma~\ref{lem:int meas}.
\end{proof}

\begin{lemma}
\label{lem:int of cont proc meas}
Let $d \in \N$,
$c, C, T \in [0,\infty)$,
let
$
\left \| \cdot \right \| \colon \R^d \to [0,\infty)
$ be the standard norm on $\R^d$,
let $\mu \colon \R^d \to \R^d$
be a $\mathcal{B}(\R^d) / \mathcal{B}(\R^d)$-measurable function which
satisfies for all $x \in \R^d$ that
$\|\mu(x)\| \leq C+c\|x\|$,
let $ (\Omega, \mathcal{F},\P) $ be a probability space,
and let $X  \colon [0,T]\times \Omega \to \R^d$ be a stochastic
process w.c.s.p.
Then for all $t\in [0,T]$ the function $\Omega \ni \omega \mapsto \int_0^t \mu(X_s(\omega))\,ds \in \mathbb R^d$ is $\mathcal F/\mathcal B(\mathbb R^d)$-measurable. 	
\end{lemma}
\begin{proof}
The fact that $X  \colon [0,T]\times \Omega \to \R^d$ is a stochastic
process w.c.s.p.\ and Aliprantis and Border~\cite[Lemma 4.51]{aliprantis:border:infinite} ensure that for all $t\in [0,T]$
the function 
$[0,t]\times \Omega \ni (s,\omega) \mapsto X_s(\omega) \in \mathbb R^d$ 
is $(\mathcal B([0,t])\otimes \mathcal F)/\mathcal B(\mathbb R^d)$-measurable.
This and the hypothesis that $\mu \colon \R^d \to \R^d$
is $\mathcal{B}(\R^d) / \mathcal{B}(\R^d)$-measurable imply that for
all $i \in \{1,\ldots,d\}$, $t\in [0,T]$ we have that the function 
$[0,t]\times \Omega \ni (s,\omega) \mapsto \mu_i(X_s(\omega)) \in \mathbb R$
is $(\mathcal B([0,t])\otimes \mathcal F)/\mathcal B(\mathbb R)$-measurable.
Moreover, note that the hypothesis that for all $x \in \R^d$ we have that
$\|\mu(x)\| \leq C+c\|x\|$ and the fact that for all $\omega \in \Omega$
the function $[0,T] \ni s\mapsto X_s(\omega)\in\mathbb R^d$ is continuous imply that for all $t\in [0,T]$, $\omega \in \Omega$, $i \in \{1,\ldots,d\}$ we have that $\int_0^t |\mu_i(X_s(\omega))|\,ds < \infty$.
Lemma~\ref{lem:int meas} (applied with $T \leftarrow t$, $Y \leftarrow \mu_i(X)$ for $t\in[0,T]$, $i \in \{1,\ldots,d\}$ in the notation
of Lemma~\ref{lem:int meas}) hence proves that for
all $i \in \{1,\ldots,d\}$, $t\in [0,T]$ the function 
$\Omega \ni \omega \mapsto \int_0^t \mu_i(X_s(\omega))\,ds \in \mathbb R$ 
is $\mathcal F/\mathcal B(\mathbb R)$-measurable.
This implies that for all $t\in [0,T]$ the function $\Omega \ni \omega \mapsto \int_0^t \mu(X_s(\omega))\,ds \in \mathbb R^d$ is $\mathcal F/\mathcal B(\mathbb R^d)$-measurable. 
This completes the proof of Lemma~\ref{lem:int of cont proc meas}.
\end{proof}

\begin{lemma}
\label{lem:sde-lp-bound}
Let $ d, m \in \N $, $ x \in \R^d $, $ p \in [1,\infty) $, 
$ c, C, T \in [0,\infty) $, 
$ A \in \R^{ d \times m } $,  
let 
$
\left \| \cdot \right \| \colon \R^d \to [0,\infty)
$ be the standard norm on $\R^d$,
let $ (\Omega, \mathcal{F},\P) $ be a probability space, 
let $W\colon [0,T]\times \Omega \to \R^m$ be a standard Brownian motion,
let $\mu \colon \R^d \to \R^d$
be a $\mathcal{B}(\R^d) / \mathcal{B}(\R^d)$-measurable function which
satisfies for all $y \in \R^d$ that
$\|\mu(y)\| \leq C+c\|y\|$,
and let $X  \colon [0,T]\times \Omega \to \R^d$ be a stochastic
process w.c.s.p.\ which satisfies for all
$ t \in [0,T] $ that
\begin{equation}
\label{eq:apriori1-ass1}
\P\!\left(X_t = x + \int_0^t \mu\!\left(X_{s}\right) ds +
AW_t\right) = 1
\end{equation}
(cf. Lemma~\ref{lem:int of cont proc meas}).
Then
\begin{equation}
\begin{split}
\big(\E\! \left[\|X_T\|^p\right]\big)^{\nicefrac{1}{p}} 
& \leq
\Big(
\|x\| + C T 
+ 
\big(
\E\big[ 
  \| A W_T \|^p
\big]
\big)^{ \nicefrac{ 1 }{ p } }
\Big)
\,
e^{ c T } .
\end{split}
\end{equation}
\end{lemma}

\begin{proof}[Proof of Lemma~\ref{lem:sde-lp-bound}]
Throughout this proof 
for every $n \in \N$ let
$
\left \| \cdot \right \|_{\R^n} \colon \R^n \linebreak \to [0,\infty)
$
be the standard norm on  $\R^n$,
let $ \beta^{ (i) } \colon [0,T]\times \Omega \to \R $, 
$ i \in \{ 1, \ldots, m \} $, 
be the stochastic processes which
satisfy for all 
$ t \in [0,T] $ that
\begin{equation}
W_t = 
\big( 
\beta^{ (1) }_t, \ldots, \beta^{ (m) }_t 
\big)
\end{equation}
and let $ B \subseteq \Omega $ 
be the set given by 
\begin{equation}
\begin{split}
B 
& = 
\bigcap_{ t \in [0,T] }
\left\{ 
X_t
=
x
+
\int_0^t
\mu\big(
X_{ s }
\big) 
\, ds
+
A W_t
\right\}
\\ &
=
\left\{ 
\omega \in \Omega 
\colon 
\left( 
\forall \, t \in [0,T] \colon 
X_t(\omega) 
= 
x
+
\int_0^t
\mu\big(
  X_{ s }( \omega )
\big) 
\, ds
+
A W_t( \omega )
\right) 
\right\}
.
\end{split}
\end{equation}
Observe that the fact that 
$ X \colon [0,T] \times \Omega \to \R^d $ 
and 
$ W \colon [0,T] \times \Omega \to \R^m $
are stochastic processes w.c.s.p.\
yields that
\begin{equation}
B 
= 
\left(
\bigcap_{ t \in [0,T] \cap \Q }
\left\{ 
X_t
=
x
+
\int_0^t
\mu\big(
X_{ s }
\big) 
\, ds
+
A W_t
\right\}
\right)
\in \mathcal{F}
.
\end{equation}
Combining this and \eqref{eq:apriori1-ass1} proves that
\begin{equation}
\label{eq:apriori1-A}
\begin{split}
\P(B) 
&  
= 
\P\!\left( 
\bigcap_{ t \in [0,T] \cap \Q } 
\left\{
X_t
=
x
+
\int_0^t
\mu\big(
  X_{ s }
\big) 
\, ds
+
A W_t
\right\}
\right) 
\\ &
= 
1 - 
\P\!\left( 
\Omega \setminus \! 
\left[
\bigcap_{ t \in [0,T] \cap \Q } 
\left\{ 
  X_t
  =
  x
  +
  \int_0^t
  \mu\big(
    X_{ s }
  \big) 
  \, ds
  +
  A W_t
\right\}
\right]
\right) 
\\
& 
= 
1 - 
\P\!\left( 
\bigcup_{ t \in [0,T] \cap \Q } 
\left\{ 
  X_t
  \neq
  x
  +
  \int_0^t
  \mu\big(
    X_{ s }
  \big) 
  \, ds
  +
  A W_t
\right\}
\right) 
\\ &
\geq 
1 
-
\left[ 
\sum_{t \in [0,T] \cap \Q} 
\P\!\left( 
X_t
\neq
x
+
\int_0^t
\mu\big(
  X_{ s }
\big) 
\, ds
+
A W_t
\right) 
\right] = 1 .
\end{split}
\end{equation}
Next note that the triangle inequality and the hypothesis that 
for all $ y \in \R^d $ we have that
$ \|\mu(y)\| \leq C + c \|y\| $ 
ensure that for all 
$ \omega \in B $, $ t \in [0,T] $ 
we have that
\begin{equation}
\label{eq:apriori1-eq1}
\begin{split}
\|X_t(\omega) \|
& \leq
\|x\|
+ \| A  W_t(\omega)  \|
+ \int_0^t \left \|\mu\!\left(X_{s}(\omega)\right)\right\| ds 
\\
& 
\leq  
\|x\|
+ \| A W_t(\omega)  \|
+ Ct
+ c\int_0^t \left\|X_{s}(\omega)\right \|  ds \\
& 
\leq  
\|x\|
+ 
\left[
\sup_{ s \in [0,T] } \| A W_s(\omega)  \| 
\right]
+ C T
+ c \int_0^t \| X_s(\omega) \| \, ds .
\end{split}
\end{equation}
Moreover, note that the assumption that
$ X \colon [0,T]  \times \Omega \to \R^d$ is a stochastic process
w.c.s.p.\ assures that for all
$ \omega \in \Omega $ we have that
\begin{equation}
\int_0^T \| X_s( \omega ) \| \, ds < \infty .
\end{equation}
Grohs et al.~\cite[Lemma 2.11]{HornungJentzen2018}
(applied with
$ 
\alpha \leftarrow \| x \| + \sup_{ t \in [0,T] } \| A  W_t(\omega ) \| + C T
$,
$ \beta \leftarrow c $,
$ 
f \leftarrow ( [0,T] \ni t \mapsto \| X_t(\omega) \| \in [0,\infty) )
$
for $ \omega \in B $ in the notation of 
Grohs et al.~\cite[Lemma 2.11]{HornungJentzen2018}) and \eqref{eq:apriori1-eq1} 
hence prove that for all $ \omega \in B $, $ t \in [0,T] $ we have that
\begin{equation}
\label{eq:apriori1-eq12}
\| X_t( \omega ) \| 
\leq
\left( 
\|x\| 
+ 
\left[
\sup_{ s \in [0,T] }
\| A  W_s(\omega)  \|
\right]
+ C T 
\right)
e^{ c t } .
\end{equation}
Next note that 
\begin{equation}
\label{eq:apriori1-eq21-0}
\begin{split}
&   
\left(
\E\!\left[ 
\sup_{ s \in [0,T] }
\| A W_s \|^p 
\right]
\right)^{
\! \! \nicefrac{ 1 }{ p } 
}
\leq       
\| A \|_{ L( \R^m, \R^d ) }
\left(
\E\!\left[ 
\sup_{ s \in [0,T] }
\| W_s \|_{\mathbb R^m}^p 
\right]
\right)^{\! \! \nicefrac{1}{p}}
.
\end{split}
\end{equation}
Moreover, if $p\in [1,2]$, then we have that
\begin{equation}
\begin{split}
\sup_{s\in[0,T]} \|W_s\|_{\mathbb R^m}^p 
& = \sup_{s\in[0,T]} \left(\sum_{i=1}^m |\beta^{(i)}_s|^2\right)^{\nicefrac{p}2}	
\leq \sup_{s\in[0,T]} \left(1+\sum_{i=1}^m |\beta^{(i)}_s|^2\right)^{\nicefrac{p}2}	\\
& \leq 1 + \sum_{i=1}^m\left[ \sup_{s\in[0,T]} |\beta^{(i)}_s|^2\right]	\,.
\end{split}
\end{equation}
Hence, if $p\in [1,2]$, then the Burkholder--Davis--Gundy inequality (see, e.g., Karatzas and Shreve~\cite[Theorem 3.28]{karatzas:shreve:brownian}) implies that
\begin{equation}
\label{eq:apriori1-eq21-a}
\mathbb E \left[\sup_{s\in[0,T]} \|W_s\|_{\mathbb R^m}^p\right]
\leq 1 + m \,\mathbb E	\left[ \sup_{s\in[0,T]} |\beta^{(1)}_s|^2\right]
< \infty\,.
\end{equation}
Next note that if $p\in (2,\infty)$, then H\"older's inequality implies that
\begin{equation}
\begin{split}
\sup_{s\in[0,T]} \|W_s\|_{\mathbb R^m}^p & = \sup_{s\in[0,T]} \left(\sum_{i=1}^m |\beta^{(i)}_s|^2\right)^{\nicefrac{p}2}
\leq m^{\nicefrac{p}{2}-1}\sup_{s\in [0,T]} \left( \sum_{i=1}^m |\beta^{(i)}_s|^p\right) \\
& \leq m^{\nicefrac{p}{2}-1} \sum_{i=1}^m \left[ \sup_{s\in [0,T]}|\beta^{(i)}_s|^p \right]\,.			
\end{split}
\end{equation}
Hence, if $p\in (2,\infty)$, then the Burkholder--Davis--Gundy inequality (see, e.g., Karatzas and Shreve~\cite[Theorem 3.28]{karatzas:shreve:brownian}) implies that
\begin{equation}
\label{eq:apriori1-eq21-b}
\mathbb E \left[\sup_{s\in[0,T]} \|W_s\|_{\mathbb R^m}^p\right]
\leq m^{\nicefrac{p}{2}-1} \,m \mathbb E\left[ \sup_{s\in [0,T]}|\beta^{(1)}_s|^p \right]
< \infty\,.
\end{equation}
This, \eqref{eq:apriori1-eq21-0}, and~\eqref{eq:apriori1-eq21-a} hence imply that
\begin{equation}
\label{eq:apriori1-eq21}
\begin{split}
&   
\left(
\E\!\left[ 
\sup_{ s \in [0,T] }
\| A W_s \|^p 
\right]
\right)^{
\! \! \nicefrac{ 1 }{ p } 
}
< \infty      
.
\end{split}
\end{equation}
Combining this, \eqref{eq:apriori1-eq12}, and \eqref{eq:apriori1-A} yields that
\begin{equation}
\label{eq:apriori1-eq2}
\begin{split}
& 
  \int_0^T 
  \big(
  \E\big[
  \| X_t \|^p 
  \big]
  \big)^{ \nicefrac{ 1 }{ p } } 
  \,
  dt 
\\ & \leq
  T
  \left[
  \sup_{ t \in [0,T] }
  \big(
  \E\big[
  \| X_t \|^p 
  \big]
  \big)^{ \nicefrac{ 1 }{ p } } 
  \right] 
\\ & \leq 
  T
  \left|
  \E\Big[
  \big|
    \| x \| 
  + 
    \sup\nolimits_{ s \in [0,T] }
    \| A W_s \| 
  + 
  C T 
\big|^{ p } 
e^{ p c T } 
\Big]
\right|^{
\nicefrac{1}{p}
}
\\
& 
\leq T
  \left[ \|x\| +
  \left|\E\!\left[ \sup_{s \in [0,T]}\|AW_s\|^{p} \right]
  \right|^{ \nicefrac{1}{p}}
  + CT \right]
e^{ c T }  < \infty.
\end{split}
\end{equation}
Next observe that the hypothesis that 
$ X \colon [0,T] \times \Omega \to \R^d $ 
is a  stochastic process w.c.s.p.\
ensures that 
$ X \colon [0,T] \times \Omega \to \R^d $ is a 
$ ( \mathcal{B}([0,T]) \otimes \mathcal{F} ) / \mathcal{B}(\R^d) $-measurable function.
This reveals that for all $ t \in [0,T] $ we have that
\begin{equation}
\left(\E\!\left[
\left|\int_0^t \| X_s \| \, ds \right|^p
\right]\right)^{\! \nicefrac{1}{p} }
\leq
\int_0^t
(\E[\|X_s\|^p])^{\nicefrac{1}{p} }
\, ds
\end{equation}
(cf., for example, Garling~\cite[Corollary 5.4.2]{MR2343341}
or
Jentzen \& Kloeden~\cite[Proposition 8 in Appendix A]{MR2856611}).
The triangle inequality,
the fact 
that for all $ t \in [0,T] $ we have that 
$W_t$ has the same distribution as
$\frac{ \sqrt{t} }{ \sqrt{T} } W_T$,
\eqref{eq:apriori1-A}, 
and
\eqref{eq:apriori1-eq1}
therefore verify that  for all $ t \in [0,T] $ we have that
\begin{equation}
\begin{split}
& 
\big(
\E\big[
\| X_t \|^p
\big]
\big)^{ \nicefrac{ 1 }{ p } } 
\\ & \leq  
\| x \|
+ 
\big(
\E\big[
\| A W_t \|^p 
\big]
\big)^{ \nicefrac{ 1 }{ p } }
+ C t
+ c \int_0^t 
\big(
\E\big[
\| X_{ s } \|^p
\big]
\big)^{ \nicefrac{ 1 }{ p } } 
\, ds 
\\
& =
\|x\|
+ 
\tfrac{ \sqrt{t} }{ \sqrt{T} }
\big(
\E\big[
\| A W_T \|^p
\big]
\big)^{ \nicefrac{ 1 }{ p } }
+ C t
+ c \int_0^t 
\big(
\E\big[
\| X_{ s } \|^p
\big]
\big)^{ \nicefrac{ 1 }{ p } } 
\, ds  
\\
& 
\leq  
\| x \|
+ 
\big( 
\E\big[
\| A W_T \|^p 
\big]
\big)^{ \nicefrac{ 1 }{ p } }
+ C T
+ c \int_0^t 
\big(
\E\big[
\| X_{ s } \|^p
\big]
\big)^{ \nicefrac{ 1 }{ p } } 
\, ds  
.
\end{split}
\end{equation}
Combining Grohs et al.~\cite[Lemma 2.11]{HornungJentzen2018}
(applied with $ \alpha \leftarrow \| x \| + ( \E[ \| A W_T \|^p ] )^{ 1 / p } + C T $,
$ \beta \leftarrow c $,
$ 
f \leftarrow ( [0,T] \ni t \mapsto
( \E[ \| X_t \|^p ] )^{ 1 / p } \in [0,\infty) )
$ 
in the notation of Grohs et al.~\cite[Lemma 2.11]{HornungJentzen2018})
and \eqref{eq:apriori1-eq2} hence establishes 
that for all $ t \in [0,T] $ we have that
\begin{equation}
\label{eq:sde-lp-bound-eq1}
\big(
\E\big[
\| X_t \|^p
\big]
\big)^{ \nicefrac{ 1 }{ p } } 
\leq 
\Big(
\| x \|
+ 
\big(
\E\big[ \| A W_T \|^p \big] 
\big)^{ \nicefrac{ 1 }{ p } }
+ C T 
\Big)
\,
e^{ c t } .
\end{equation}
This completes the proof of Lemma~\ref{lem:sde-lp-bound}.
\end{proof}

\begin{cor}
\label{cor:apriori1}
Let $ d, m \in \N $, $ x \in \R^d $, $ p \in [1,\infty) $, 
$ c, C, T \in [0,\infty) $, 
$ A \in \R^{ d \times m } $,  
let 
$
\left \| \cdot \right \| \colon \R^d \to [0,\infty)
$ be the standard norm on $\R^d$,
let $ (\Omega, \mathcal{F},\P) $ be a probability space, 
let $W\colon [0,T]\times \Omega \to \R^m$ be a standard Brownian motion,
let $\mu \colon \R^d \to \R^d$
be a $\mathcal{B}(\R^d) / \mathcal{B}(\R^d)$-measurable function which
satisfies for all $y \in \R^d$ that
$\|\mu(y)\| \leq C+c\|y\|$,
and let $X  \colon [0,T]\times \Omega \to \R^d$ be a stochastic
process w.c.s.p.\ which satisfies for all
$ t \in [0,T] $ that
\begin{equation}
\label{eq:apriori1-ass}
\P\!\left(X_t = x + \int_0^t \mu\!\left(X_{s}\right) ds +
AW_t\right) = 1
\end{equation}
(cf. Lemma~\ref{lem:int of cont proc meas}).
Then
\begin{equation}
\label{eq:apriori1}
\begin{split}
  \big(\E\big[\|X_T\|^p \big] \big)^{\nicefrac{1}{p}} 
& \leq
  \Big(
  \|x\| + C T 
  + 
  \sqrt{ \max\{1,p-1\} \operatorname{Trace}(A^{ \ast } A) \, T} 
  \Big)
  \,
  e^{ c T } .
\end{split}
\end{equation}
\end{cor}
\begin{proof}[Proof of Corollary~\ref{cor:apriori1}]
Observe that Lemma~\ref{l:exp.Gauss}
and
Lemma~\ref{lem:sde-lp-bound}
establish \eqref{eq:apriori1}.
This completes the proof of Corollary~\ref{cor:apriori1}.
\end{proof}

\subsection{A priori estimates for differences of solutions}
In this subsection we provide in Lemma~\ref{lem:sde-lp-bound2} below well-known a priori estimates 
for differences of solutions of  stochastic differential equations
with
Lipschitz continuous drift coefficient functions
and
constant diffusion coefficient functions. 
Lemma~\ref{lem:sde-lp-bound2} 
is one of the main ingredients in our proof of
Lemma~\ref{lem:mc-nnet-ap} 
in Subsection \ref{ssec:qee} below. Our proof of Lemma~\ref{lem:sde-lp-bound2} is a slight adaption of the proof of Lemma~2.6 in
Beck et al.\ \cite{Becker2018}.
\begin{lemma}
\label{lem:sde-lp-bound2}
Let 
$d,m \in \N$, 
$ p \in [1,\infty)$,
$l \in [0,\infty)$,
$T \in [0,\infty)$, 
$A \in \R^ {d \times m}$,  
let $\left \| \cdot \right \| \colon \R^d \to [0,\infty)$ be the standard norm on $\R^d$,
let $(\Omega, \mathcal{F},\P)$ be a probability space, 
let $W\colon [0,T]\times \Omega \to \R^m$ be a standard Brownian motion, 
let $\mu \colon \R^d \to \R^d$  
be a $\mathcal{B}(\R^d) / \mathcal{B}(\R^d)$-measurable function which satisfies for all $x,y \in \R^d$ that
$\|\mu(x)-\mu(y)\| \leq l\|x-y\|$,
and let 
$X^x \colon [0,T]\times \Omega \to \R^d$,
$x \in \R^d$,
be stochastic processes w.c.s.p.\ which satisfy for all
$ x \in \R^d$,
$t \in [0,T]$  that
\begin{equation}
\label{eq:apriori2-ass}
\P\!\left(X_t^x = x + \int_0^t \mu(X_s^x)\, ds + AW_t\right) = 1.
\end{equation}
Then  it holds for all $x,y \in \R^d$, $t \in [0,T]$ that 
\begin{equation}
  \left(\E\! \left[\|X_t^x-X_t^y\|^p\right]\right)^{ \nicefrac{1}{p} } 
\leq 
  e^{lt} \|x-y\|.
\end{equation}
\end{lemma}
\begin{proof}[Proof of Lemma~\ref{lem:sde-lp-bound2}]
Throughout this proof
let 
$ B_x \subseteq \Omega $,
$ x \in \R^d $,
be the sets which satisfy for all 
$ x \in \R^d $
that
\begin{equation}
\begin{split}
B_x
& = 
\bigcap_{ t \in [0,T] }
\left\{ 
X_t^x
=
x
+
\int_0^t
\mu\big(
X_{ s }^x
\big) 
\, ds
+
A W_t
\right\}
\\ &
=
\left\{ 
\omega \in \Omega 
\colon 
\left( 
\forall \, t \in [0,T] \colon 
X_t^x(\omega) 
= 
x
+
\int_0^t
\mu\big(
  X_{ s }^x( \omega )
\big) 
\, ds
+
A W_t( \omega )
\right) 
\right\}
.
\end{split}
\end{equation}
Observe that the fact that 
for all $x \in \R^d$ we have that
$ X^x \colon [0,T] \times \Omega \to \R^d $ 
and 
$ W \colon [0,T] \times \Omega \to \R^m $
are stochastic processes w.c.s.p.\
yields that for all $ x \in \R^d$ we have that
\begin{equation}
B_x
= 
\left(
\bigcap_{ t \in [0,T] \cap \Q }
\left\{ 
X_t^x
=
x
+
\int_0^t
\mu\big(
X_{ s }^x
\big) 
\, ds
+
A W_t
\right\}
\right)
\in \mathcal{F}
.
\end{equation}
Combining this and \eqref{eq:apriori2-ass} proves that
for all $x \in \R^d$ we have that
\begin{equation}
\begin{split}
\P(B_x ) 
&  
= 
\P\!\left( 
\bigcap_{ t \in [0,T] \cap \Q } 
\left\{
X_t^x
=
x
+
\int_0^t
\mu\big(
  X_{ s }^{ x }
\big) 
\, ds
+
A W_t
\right\}
\right) 
\\ &
= 
1 - 
\P\!\left( 
\Omega \setminus \! 
\left[
\bigcap_{ t \in [0,T] \cap \Q } 
\left\{ 
  X_t^x
  =
  x
  +
  \int_0^t
  \mu\big(
    X_{ s }^{ x }
  \big) 
  \, ds
  +
  A W_t
\right\}
\right]
\right) 
\\
& 
= 
1 - 
\P\!\left( 
\bigcup_{ t \in [0,T] \cap \Q } 
\left\{ 
  X_t^x
  \neq
  x
  +
  \int_0^t
  \mu\big(
    X_{ s }^{ x }
  \big) 
  \, ds
  +
  A W_t
\right\}
\right) 
\\ &
\geq 
1 
-
\left[ 
\sum_{t \in [0,T] \cap \Q} 
\P\!\left( 
X_t^x
\neq
x
+
\int_0^t
\mu\big(
  X_{ s }^{ x }
\big) 
\, ds
+
A W_t
\right) 
\right] = 1 .
\end{split}
\end{equation}
This reveals that for all $x,y \in \R^d$ we have that
\begin{equation}
\label{eq:apriori2-A}
\begin{split}
& \P\!\left(B_x \cap B_y\right) 
= 
  1- \P\!\left(\Omega \backslash [B_x \cap B_y] \right) 
\\ & = 
  1 - \P\!\left(B_x^c \cup B_y^c\right) 
\geq 
  1 - \left[\P\!\left(B_x^c\right) + P\!\left(B_y^c\right) \right] = 1.
\end{split}
\end{equation}
Next note that the triangle inequality and the assumption that  for all $x,y \in \R^d$ we have that
$ \|\mu(x)-\mu(y)\| \leq l\|x-y\|$ ensure that for all 
$ x,y \in \R^d$,
$ \omega \in B_x \cap B_y $,
$t \in [0,T]$ we have that
\begin{equation}
  \label{eq:apriori2-eq1}
        \begin{split}
        \|X_t^x(\omega) -X_t^y(\omega)\| 
        & \leq 
        \|x-y\| 
        + \int_0^t \|\mu(X_s^x(\omega))-\mu(X_s^y(\omega))\| \, ds \\
        & \leq  \|x-y\| 
        + l\int_0^t \|X_s^x(\omega)-X_s^y(\omega)\| \, ds. 
        \end{split}
\end{equation}
Moreover, note that the assumption that for all $x \in \R^d$ we have that
$X^x \colon [0,T] \linebreak \times \Omega \to \R^d $ is a stochastic process with  continuous sample paths assures that for all $x,y \in \R^d$, $ \omega \in \Omega $ it holds  that
\begin{equation}
  \int_0^T \|X_s^x(\omega)-X_s^y(\omega)\| \, ds < \infty.
\end{equation}
Grohs et al.~\cite[Lemma 2.11]{HornungJentzen2018}
(applied with $\alpha \leftarrow \|x-y\| $,
$\beta \leftarrow l$,  $f \leftarrow ([0,T] \ni t \mapsto \|X_t^x(\omega)-X_t^y(\omega) \| \in [0,\infty))$
for $x,y \in \R^d$, $\omega \in B_x \cap B_y$
in the notation
of Grohs et al.~\cite[Lemma 2.11]{HornungJentzen2018}) 
and \eqref{eq:apriori2-eq1}
hence prove that for all 
$ x,y \in \R^d $, 
$ \omega \in B_x \cap B_y$,
$ t \in [0,T] $ we have that
\begin{equation}
  \|X_t^x(\omega) - X_t^y(\omega) \| \leq 
  \|x-y\|
  e^{lt}.
\end{equation}
This and~\eqref{eq:apriori2-A} prove that for all $x,y \in \mathbb R^d$, $t\in [0,T]$ we have that
\begin{equation}
\left(\E\! \left[\|X_t^x-X_t^y\|^p\right]\right)^{\nicefrac{1}{p}} 
\leq e^{lt}\|x-y\|.	
\end{equation}
This completes the proof of Lemma~\ref{lem:sde-lp-bound2}.
\end{proof}
\section{Error estimates}
\label{sec:error}

\subsection{Quantitative error estimates}
\label{ssec:qee}

In this subsection we establish in Corollary~\ref{cor:pde-nnet-ap5} below a quantitative approximation result for viscosity solutions
(cf., for example, Hairer et al.~\cite{hairer2015}) of Kolomogorov PDEs with constant coefficient
functions. Our proof of Corollary~\ref{cor:pde-nnet-ap5} employs the quantitative approximation results in Lemma~\ref{lem:mc-nnet-ap} and Corollary~\ref{cor:pde-nnet-ap4} below. 
Corollary~\ref{cor:pde-nnet-ap5} is one of the key ingredients which we use in our proof of Proposition~\ref{prop:ql-ap1} and Proposition~\ref{prop:nnet-ap1} below, respectively, in order to construct ANN
approximations for viscosity solutions of Kolmogorov PDEs
with constant coefficient functions.

\begin{lemma}
\label{lem:mc-nnet-ap}
Let 
$ d,n \in \N$, 
$ \varphi \in C(\R^d,\R)$, 
$ c,l, a \in \R$, $b \in (a,\infty)$,   
$ \zeta, \varepsilon,T \in (0,\infty)$, 
$ v,\mathbf{v},w,\mathbf{w}, z,\mathbf{z}  \in [0,\infty)$, 
$ p \in [1,\infty)$, 
let $\left \| \cdot \right \| \colon \R^d \to [0,\infty)$ be the standard norm on $\R^d$,
let $ \scp{\cdot}{\cdot}{} \colon \R^d \times \R^d \to \R$  be the $d$-dimensional Euclidean scalar product,
let $\mathfrak{K} \in (0,\infty)$ be the  $\max\{2,p\}$-Kahane--Khintchine constant
(cf. Definition~\ref{def:p2-kahane-khintchine} and Lemma~\ref{lem:kahane-const}),
assume for all $\Phi \in C^1((0,1)^d, \R)$ 
that

\begin{equation}
  \sup_{ x \in (0,1)^d }
  \left| \Phi(x) \right|
  \leq
  \zeta
  \left[ 
  \int_{ (0,1)^d }
  \left(
  \left|
  \Phi(x)
  \right|^{\max\{2,p\}}
  +
  \left\|
  (\nabla \Phi)( x )
  \right\|^{\max\{2,p\}}
  \right)
  dx
  \right]^{ \nicefrac{ 1 }{ \max\{2,p\}} }
  ,
\end{equation}
let $\phi \in C^1(\R^d,\R)$ satisfy for all $ x \in \R^d$ that
\begin{equation}
  \label{eq:mc-nnet-ass1}
  \left| \phi(x)\right| \leq c  d^z(1+\|x\|^{\mathbf{z}}), \qquad \| (\nabla \phi)(x)\|\leq  c  d^w(1+\|x\|^{\mathbf{w}}),
\end{equation}
\begin{equation}
  \label{eq:mc-nnet-ass2}
  \text{and} \qquad \left| \varphi(x)-\phi(x)\right| \leq \varepsilon  d^v(1+\|x\|^{\mathbf{v}}), 
\end{equation}
let $\mu \colon \R^d \to \R^d$ satisfy for all  $x,y \in \R^d$, $\lambda \in \R$ that $\mu(\lambda x+y) + \lambda \mu(0) = \lambda\mu(x)+\mu(y)$  and 
\begin{equation}
\label{eq:pde-nnet-muc2} \|\mu(x)-\mu(y)\| \leq  l\|x-y\|,
\end{equation}
let $A = (A_{i,j})_{(i,j) \in \{1,\ldots,d\}^2} \in \R^{d \times d}$ be a  symmetric and positive semi-definite matrix, 
let $u \in  C([0,T] \times \R^d, \R)$, assume for all $x \in \R^d$ that $u(0,x) = \varphi(x)$,
assume that 
$
\inf_{\gamma \in (0,\infty)} \sup_{(t,x) \in [0,T] \times \R^d} \big( \frac{|u(t,x)|}{1 +\|x\|^{\gamma}} \big) 
< \infty
$, 
and 
assume that  
$u|_{(0,T) \times \R^d}$ 
is a viscosity solution of
\begin{align}
(\tfrac{\partial }{\partial t}u)(t,x) 
&=
\scp{\mu(x)}{(\nabla_x u)(t,x)}{}
+
\tfrac{1}{2} \smallsum_{i,j=1}^d A_{i,j}\, (\tfrac{\partial^2 }{\partial x_i \partial x_j}u)(t,x) 
\end{align}
for $(t,x) \in (0,T) \times \R^d$.
Then there exist $\WE_1,\ldots,\WE_{n} \in \R^{d \times d}$, $\BI_1,\ldots,\BI_{n} \in \R^d$ such that
\begin{equation}
\begin{split}
& \sup_{x \in [a,b]^d} \left|
u(T,x) 
- \left[\frac{1}{n}\sum_{k=1}^{n}\, \phi( \WE_k x + \BI_k)\right]
\right| \\
& \leq
\varepsilon d^v  \Big(
  1+ 
  e^{\mathbf{v} lT}
  \big[
  T\|\mu(0)\|+ \sqrt{ \max\{1,\mathbf{v}-1\} \, T \operatorname{Trace}(A)} 
\\ & + 
  \sup_{x \in [a,b]^d} \|x\| 
  \big]^{\mathbf{v}}
  \Big) 
\\ & + 
\frac{4 \mathfrak{K} \zeta}{\sqrt{n}}  \Biggl(   
cd^z  \Big(
  1+ 
  e^{\mathbf{z} lT}
  \big[
  T\|\mu(0)\|+ \sqrt{ \max\{1,\mathbf{z} \max\{2,p\} -1\} \, T \operatorname{Trace}(A)} 
\\ & + 
  \sup_{x \in [a,b]^d} \|x\| 
  \big]^{\mathbf{z}}
  \Big)
+
  (b-a)
  cd^{w+\nicefrac{1}{2} } e^{lT}
  \Big(
  1+
  e^{\mathbf{w} lT}
  \big[
  T \|\mu(0)\|
\\ & + 
  \sqrt{ 
  \max\{1,\mathbf{w} \max\{2,p\}
  + 
  \mathbf{w} -1\} \, T \operatorname{Trace}(A)} 
  +
  \sup_{x \in [a,b]^d} \|x\| 
  \big]^{\mathbf{w}}
  \Big)
  \Biggr)
.
\end{split}
\end{equation}
\end{lemma}
\begin{proof}[Proof of Lemma~\ref{lem:mc-nnet-ap}]
Throughout this proof let $e_1,\ldots,e_d\in \R^d$ satisfy  that
$
e_1 = (1,0,\ldots,0), \ldots, e_d = (0,\ldots,0,1),
$
let 
$m_r \colon (0,\infty) \to [r,\infty)$, 
$ r \in \R$,  
 satisfy for all $ r \in \R$, $x \in (0,\infty)$ that 
$m_r(x) = \max\{r,x\}$,
let $(\Omega, \mathcal{F},\P)$ be a probability space with a normal filtration $(\mathbbm{F}_t)_{t \in [0,T]}$, 
let $W^{i}\colon[0,T]\times \Omega \to \R^d$, $i \in \N,$ be independent standard $(\mathbbm{F}_t)_{t \in [0,T]}$-Brownian motions, 
and let $X^{i,x} \colon [0,T]\times \Omega \to \R^d$, $i \in \N$, $x \in \R^d$, be $(\mathbbm{F}_t)_{t \in [0,T]}$-adapted stochastic processes w.c.s.p.\
which satisfy that 
\begin{enumerate}[(a)]
\item \label{it:pde-nnet-pf1} for all $i \in \N$, $x \in \R^d$, $t \in [0,T]$ it holds $\P$-a.s.\   that
\begin{equation}
\label{eq:pde-nnet-sde}
X_t^{i,x} = x + \int_0^t \mu(X_s^{i,x})\, ds + \sqrt{A} W^{i}_t
\end{equation}
and
\item \label{it:pde-nnet-pf2} 
for all $i \in \N$,  $x,y \in \R^d$, $\lambda \in \R$, $t \in [0,T]$,  $\omega \in \Omega$ we have that
$X_t^{i,\lambda x+y}(\omega) + \lambda X_t^{i,0}(\omega) = \lambda X_t^{i,x}(\omega) + X_t^{i,y}(\omega)$
\end{enumerate}
(cf.\ Grohs et al.~\cite[Proposition~2.20]{HornungJentzen2018}).
Note that item \eqref{it:pde-nnet-pf2} and Grohs et al.~\cite[Corollary 2.8]{HornungJentzen2018}
(applied with $d \leftarrow d$,
$m \leftarrow d$,
$\varphi \leftarrow (\R^d \ni x \mapsto X_T^{i,x}(\omega) \in \R^d)$ for
$i \in \N$, $ \omega \in \Omega $ 
in the notation of Grohs et al.~\cite[Corollary 2.8]{HornungJentzen2018})
ensure that for all 
$i \in \N$,
$ \omega \in \Omega $
there exist
$ 
\mathcal{W}_{i,\omega} \in \R^{d\times d} 
$ 
and 
$
\mathcal{B}_{i,\omega} \in \R^d
$ 
which satisfy that for all 
$ x \in \R^d $
we have that
\begin{equation}
\label{eq:pde-nnet-xaff}
X_T^{i,x}(\omega) 
=
\mathcal{W}_{i,\omega} x
+
\mathcal{B}_{i,\omega}.
\end{equation}
Combining this with the assumption that $\phi \in C^1(\mathbb R^d,\mathbb R)$ proves that for all
$ \omega \in \Omega $,
$ i \in \mathbb N$ 
we have that
\begin{equation} 
\label{eq:pde-nnet-Xfct}
\left(\mathbb R^d \ni x \mapsto \phi\left(X_T^{i,x}(\omega)\right) \in \R\right)
\in C^1(\mathbb R^d,\R)
.
\end{equation}
Next note that
\eqref{eq:mc-nnet-ass1}
and
\eqref{eq:mc-nnet-ass2}
ensure 
that $ \varphi \colon \R^d \to \R $ is an at most polynomially growing function.
This, 
\eqref{eq:pde-nnet-muc2},
and
the Feynman-Kac formula (cf., for example, Grohs et al.~\cite[Proposition 2.22]{HornungJentzen2018} or Hairer et al.~\cite[Corollary 4.17]{hairer2015}) imply that for all  $x \in \R^d$ we have that
\begin{equation}
\label{eq:pde-nnet-fcf}
u(T,x) = \E\big[\varphi(X_T^{1,x})\big].
\end{equation}
This verifies that  
\begin{equation}
\label{eq:pde-nnet-step1}
\begin{split}
& \sup_{x \in [a,b]^d} \left| u(T,x) - \E\big[\phi(X_T^{1,x})\big]\right|
\\ & =
  \sup_{x \in [a,b]^d  } \left| \E\big[\varphi(X_T^{1,x})-\phi(X_T^{1,x})\big]\right| 
\\ & \leq
  \sup_{x \in [a,b]^d  } \E\! \left[\left|\varphi(X_T^{1,x})
  -\phi(X_T^{1,x})\right| 
  \frac{ d^v \big(1+\big\|X_T^{1,x}\big\|^{\mathbf{v}} \big) }
  { d^v\big(1+\big\|X_T^{1,x}\big\|^{\mathbf{v}}\big) }
  \right]
\\ & \leq
  d^v
  \left[\sup_{y \in \mathbb R^d} \left(\frac{|\varphi(y)-\phi(y)|}
  {d^v(1+\|y\|^{\mathbf{v}})}
  \right)\right] \left[ \sup_{x \in[a,b]^d } \E\!\left[1+\big\|X_T^{1,x}\big\|^{\mathbf{v}}\right]\right] 
.
\end{split}
\end{equation}
Jensen's inequality
and
\eqref{eq:mc-nnet-ass2}
therefore verify that 
\begin{equation}
\begin{split}
& \sup_{x \in [a,b]^d} \left| u(T,x) - \E\big[\phi(X_T^{1,x})\big]\right|
\\ & \leq
  d^v
  \left[\sup_{y \in \mathbb R^d} \left(\frac{|\varphi(y)-\phi(y)|}
  {d^v(1+\|y\|^{\mathbf{v}})}
  \right)\right] 
  \left[ 1+\sup_{x \in[a,b]^d } \E\!\left[\big\|X_T^{1,x}\big\|
  ^{\mathbf{v} }
  \right]\right] 
\\ & \leq
  \varepsilon d^v \left[1+ \sup_{x \in [a,b]^d} 
  \left| \E\big[\big\|X_T^{1,x}\big \|^{m_1(\mathbf{v})}\big] \right|
  ^{ \nicefrac{\mathbf{v}} { m_1(\mathbf{v})} } \right]
.
\end{split}
\end{equation}
Item \eqref{it:pde-nnet-pf1},
Corollary~\ref{cor:apriori1} 
(applied with  $m \leftarrow d$,  $C \leftarrow \|\mu(0)\|$,  $ c \leftarrow l$, $W \leftarrow W^1$, $A \leftarrow \sqrt{A}$,  $X \leftarrow X^{1,x}$, $p \leftarrow m_1(\mathbf{v})$ for $x\in [a,b]^d$ in the notation of Corollary~\ref{cor:apriori1}),
and 
the fact that 
$
m_1(m_1( \mathbf{v}) - 1) = m_1( \mathbf{v} - 1)
$
hence yield that  \begin{equation}
\label{eq:pde-nnet-err1}
\begin{split}
& \sup_{x \in [a,b]^d} \left| u(T,x) - \E\big[\phi(X_T^{1,x})\big]\right| 
\\& \leq 
  \varepsilon d^v  \Big(
  1+ e^{\mathbf{v} lT}\big[
  \|\mu(0)\|T+\sqrt{m_1(\mathbf{v} - 1) \, T \operatorname{Trace}(A)} 
\\ & + 
  \sup_{x \in [a,b]^d} \|x\| 
  \big]^{\mathbf{v}}
  \Big) 
  .
\end{split}
\end{equation}
Next note that
\begin{equation}
\begin{split}
& \sup_{x \in [a,b]^d} 
  \left|\E\!\left[|\phi(X_T^{1,x})|^{ m_2(p) }\right]\right|^{ \nicefrac{1}{{ m_2(p) } }}
\\ & = 
  \sup_{x \in [a,b]^d} 
  \left|\E\!\left[\left|
  \phi(X_T^{1,x})
   \frac{ d^z \big(1+\big\|X_T^{1,x}\big\|^{\mathbf{z}} \big) }
  { d^z\big(1+\big\|X_T^{1,x}\big\|^{\mathbf{z}}\big) }
  \right|^{ m_2(p) }\right]\right|^{ \nicefrac{1}{{ m_2(p) } }} 
\\ & \leq
  d^z \left[ \sup_{y \in \mathbb R^d} \frac{|\phi(y)|}{d^{z}(1+ \|y\|^{\mathbf{z}})}\right]  
  \left[\sup_{x \in [a,b]^d }
  \left|\E\!\left[
  \big(1+\big\|X_T^{1,x}\big\|^{\mathbf{z}}\big)^{ { m_2(p) }}\right]\right|
  ^{ \nicefrac{1}{m_2(p)}}\right] 
.
\end{split}
\end{equation}
This, \eqref{eq:mc-nnet-ass1}, 
the triangle inequality,
and
H\"older's inequality
verify that  
\begin{equation}
\begin{split}
& \sup_{x \in [a,b]^d} 
  \left|\E\!\left[|\phi(X_T^{1,x})|^{ m_2(p) }\right]\right|^{ \nicefrac{1}{{ m_2(p) } }}
\\ & \leq
  c d^z\left[1+\sup_{x \in [a,b]^d }
  \left|\E\!\left[\|X_T^{1,x}\|^{\mathbf{z} { m_2(p) }}\right]\right|^{\nicefrac{1}{m_2(p)}}\right] 
\\ & \leq 
  c d^z\left[1+\sup_{x \in [a,b]^d }
  \left|\E\!\left[\|X_T^{1,x}\|^{m_1(\mathbf{z} { m_2(p) })}\right]\right|
  ^{ \nicefrac{\mathbf{z}}{ m_1(\mathbf{z}{ m_2(p) }) } }\right] 
  .
\end{split}
\end{equation}
Item \eqref{it:pde-nnet-pf1},
Corollary~\ref{cor:apriori1}
(applied with $m \leftarrow d$, $C \leftarrow \|\mu(0)\|$,  $c \leftarrow l$, $A \leftarrow \sqrt{A}$, $W \leftarrow W^1$,  $X \leftarrow X^{1,x}$, $p \leftarrow m_1(\mathbf{z} m_2(p) )$ for $x\in[a,b]^d$ in the notation of Corollary~\ref{cor:apriori1}),
and
the fact that 
$
m_1(m_1(\mathbf{z} m_2(p) ) - 1) = m_1(\mathbf{z} m_2(p) -1 )
$
hence yield that 
\begin{equation}
\label{eq:pde-nnet-lqb1}
\begin{split}
& \sup_{x \in [a,b]^d} \left|\E\!\left[|\phi(X_T^{1,x})|^{ m_2(p) } \right]\right|
  ^{ \nicefrac{1}{m_2(p) } } 
\\ & \leq 
  cd^z  \Big(
  1+ 
  e^{\mathbf{z} lT}
  \big[
  T\|\mu(0)\|+ \sqrt{ m_1(\mathbf{z} m_2(p)-1) \, T \operatorname{Trace}(A)} 
\\ & + 
  \sup_{x \in [a,b]^d} \|x\| 
  \big]^{\mathbf{z}}
  \Big)
  .
\end{split}
\end{equation}
Next note that 
\begin{equation}
\label{eq:pde-nnet-grad}
\begin{split}
& \sup_{x \in [a,b]^d } 
  \left|
  \E\!\left[
  \limsup_{ \R^d \backslash \{0\} \ni h \to 0 }
  \Bigg(
  \frac
  { | \phi(X_T^{1,x+h}) - \phi(X_T^{1,x})|^{ m_2(p) } }
  {\|h\|^{ m_2(p) }}
  \Bigg)
  \right]
  \right|^{\nicefrac{1}{m_2(p)}} 
\\ & = 
  \sup_{x \in [a,b]^d } 
  \left|
  \E\!\left[
  \Big\| \big(  \tfrac{\partial}{ \partial x } X_T^{1,x} \big)^* (\nabla \phi)(X_T^{1,x}) \Big\|^{ m_2(p) }
  \right]
  \right|^{\nicefrac{1}{m_2(p)}} 
\\ & \leq
  \sup_{x \in [a,b]^d } 
  \left|
  \E\!\left[
  \Big \|
  (\nabla \phi)(X_T^{1,x})
  \Big \|^{ m_2(p) }
  \Big\| \big(\tfrac{\partial}{\partial x} X_T^{1,x}\big) \Big\|_{L(\R^d,\R^d)}^{ m_2(p) }
  \right]
  \right|^{\nicefrac{1}{m_2(p)}} 
  .
\end{split}
\end{equation}
Moreover, observe that 
for all $x \in [a,b]^d$ we have that
\begin{equation}
\begin{split}
& \left|
  \E\!\left[
  \Big \|
  (\nabla \phi)(X_T^{1,x})
  \Big \| ^{ m_2(p) }
  \Big\| \big(\tfrac{\partial}{\partial x} X_T^{1,x}\big) \Big\|_{L(\R^d,\R^d)}^{ m_2(p) }
  \right]
  \right|^{ \nicefrac{1}{m_2(p)} }
\\ & = 
\left|
  \E\!\left[
  \Big \|
  (\nabla \phi)(X_T^{1,x})
  \Big \| ^{ m_2(p) }
  \frac{ \big[ d^{w} \big( 1 + \|X_T^{1,x}\|^{\mathbf{w}}\big)\big]
  ^{ m_2(p) }}
  {\big[d^{w} \big( 1 + \|X_T^{1,x}\|^{\mathbf{w}}\big)\big]
  ^{ m_2(p) }}
  \Big\| \big(\tfrac{\partial}{\partial x} X_T^{1,x}\big) \Big\|_{L(\R^d,\R^d)}^{ m_2(p) }
  \right]
  \right|^{ \nicefrac{1}{m_2(p)} }
\\ & \leq 
  d^{w} \left[
  \sup_{ y \in \R^d } 
  \frac{ \| (\nabla \phi)(y) \| }
  { d^{w} ( 1 + \|y\|^{\mathbf{w}}) } 
  \right]
  \left|
  \E\!\left[
  \big( 1 + \|X_T^{1,x}\|^{\mathbf{w}}\big)
  ^{ m_2(p) }
  \Big\| \big(\tfrac{\partial}{\partial x} X_T^{1,x}\big) \Big\|_{L(\R^d,\R^d)}^{ m_2(p) }
  \right]
  \right|^{ \nicefrac{1}{m_2(p)} }
  .
\end{split}
\end{equation}
This, \eqref{eq:pde-nnet-grad}, and~\eqref{eq:mc-nnet-ass1} prove that 
\begin{equation}
\begin{split}
& \sup_{x \in [a,b]^d } 
  \left|
  \E\!\left[
  \limsup_{ \R^d \backslash \{0\} \ni h \to 0 }
  \Bigg(
  \frac
  { | \phi(X_T^{1,x+h}) - \phi(X_T^{1,x})|^{ m_2(p) } }
  {\|h\|^{ m_2(p) }}
  \Bigg)
  \right]
  \right|^{\nicefrac{1}{m_2(p)}} 
\\ & \leq 
  c d^{ w } 
  \sup_{ x \in [a,b]^d }
  \left|
  \E\!\left[
  \big( 1 + \|X_T^{1,x}\|^{\mathbf{w}}\big)
  ^{ m_2(p) }
  \Big\| \big(\tfrac{\partial}{\partial x} X_T^{1,x}\big) \Big\|_{L(\R^d,\R^d)}^{ m_2(p) }
  \right]
  \right|^{ \nicefrac{1}{m_2(p)} }
  .
\end{split}
\end{equation}
This, H\"older's inequality,
and
the triangle inequality
verify that  
\begin{equation}
\begin{split}
&  \sup_{x \in [a,b]^d } 
  \left|
  \E\!\left[
  \limsup_{ \R^d \backslash \{0\} \ni h \to 0 }
  \Bigg(
  \frac
  { | \phi(X_T^{1,x+h}) - \phi(X_T^{1,x})|^{ m_2(p) } }
  {\|h\|^{ m_2(p) }}
  \Bigg)
  \right]
  \right|^{\nicefrac{1}{m_2(p)}} 
\\ &  \leq 
  c
  d^w  
  \Biggl[
  \sup_{x \in [a,b]^d} \Bigg(
  \left|
  \E\!
  \left[ 
  \big\|\big(\tfrac{\partial}{\partial x} X_T^{1,x}\big)\big\|_{L(\R^d,\R^d)}^{{ m_2(p) }({ m_2(p) }+1)}
  \right]
  \right|^{1/({ m_2(p) }({ m_2(p) }+1))} 
\\ & \cdot 
  \left|
  \E\!
  \left[ 
  \big(1+\big\|X_T^{1,x}\big\|^{\mathbf{w}}\big)^{{ m_2(p) }+1}
  \right]
  \right|^{1/({ m_2(p) }+1)} \Bigg)
  \Biggr]
\\ &  \leq 
  c
  d^w  
  \Biggl[
  \sup_{x \in [a,b]^d} \Bigg(
  \left|
  \E\!
  \left[ 
  \big\|\big(\tfrac{\partial}{\partial x} X_T^{1,x}\big)\big\|_{L(\R^d,\R^d)}^{{ m_2(p) }({ m_2(p) }+1)}
  \right]
  \right|^{1/({ m_2(p) }({ m_2(p) }+1))} 
\\ & \cdot 
  \left(
  1+
  \left|
  \E\!
  \left[ 
  \big\|X_T^{1,x}\big\|^{\mathbf{w} ({ m_2(p) }+1)}
  \right]
  \right|^{1/({ m_2(p) }+1)}
  \right) \Bigg)
  \Biggr]
.
\end{split}
\end{equation}
Jensen's inequality therefore yields that 
\begin{equation}
\label{eq:pde-nnet-pf-grad}
\begin{split}
&  \sup_{x \in [a,b]^d } 
  \left|
  \E\!\left[
  \limsup_{ \R^d \backslash \{0\} \ni h \to 0 }
  \Bigg(
  \frac
  { | \phi(X_T^{1,x+h}) - \phi(X_T^{1,x})|^{ m_2(p) } }
  {\|h\|^{ m_2(p) }}
  \Bigg)
  \right]
  \right|^{\nicefrac{1}{m_2(p)}} 
\\ &  \leq
  c
  d^w  
  \Biggl[ 
  \sup_{x \in [a,b]^d } \Bigg(
  \left|
  \E\!
  \left[ 
  \big\|\big(\tfrac{\partial}{\partial x} X_T^{1,x}\big)\big\|_{L(\R^d,\R^d)}^{{ m_2(p) }({ m_2(p) }+1)}
  \right]
  \right|^{1/({ m_2(p) }({ m_2(p) }+1))} 
\\  & \cdot 
  \left( 1+
  \left|
  \E\!
  \left[ 
  \big\|X_T^{1,x}\big\|^{m_1(\mathbf{w}({ m_2(p) }+1))}
  \right]
  \right|^{\mathbf{w}/m_1(\mathbf{w}({ m_2(p) }+1))}
  \right) \Bigg)
  \Biggr]
  .
\end{split}
\end{equation}
Next note that \eqref{eq:pde-nnet-xaff} implies that for all 
$ \omega \in \Omega$,
$x, y \in \R^d$
we have that
\begin{equation}
\big(\tfrac{\partial}{\partial x} X_T^{1,x}(\omega)\big) y = X_T^{1,y}(\omega)-X_T^{1,0}(\omega).
\end{equation}
This and H\"older's inequality verify that  for all  $x \in \R^d$ we have that
\begin{equation}
\begin{split}
& \left|
  \E\!
  \left[ 
  \big\|\big(\tfrac{\partial}{\partial x} X_T^{1,x}\big)\big\|_{L(\R^d,\R^d)}^{{ m_2(p) }({ m_2(p) }+1)}
  \right]
  \right|^{1/({ m_2(p) }({ m_2(p) }+1))} 
\\ & \leq 
  \left|
  \E\!
  \left[ 
  \left(
  \sum_{i=1}^d \big\|\big(\tfrac{\partial}{\partial x} X_T^{1,x}\big)e_i\big\|^2
  \right)^{\!({ m_2(p) }({ m_2(p) }+1))/2}
  \right]
  \right|^{1/({ m_2(p) }({ m_2(p) }+1))} 
\\ & \leq \left| \mathbb E \left[ d^{\frac{m_2(p) ( m_2(p) +1)-2}{2}} \left(\sum_{i=1}^d \big\|\big(\tfrac{\partial}{\partial x} X_T^{1,x}\big)e_i\big\|^{m_2(p) ( m_2(p) +1)} \right) \right] \right|^{1/( m_2(p) ( m_2(p) +1))} 
  \\
  & = d^{\frac{m_2(p) ( m_2(p) +1)-2}{2m_2(p) ( m_2(p) +1)}} \left|\sum_{i=1}^d \mathbb E\left[\big\|\big(\tfrac{\partial}{\partial x} X_T^{1,x}\big)e_i\big\|^{m_2(p) ( m_2(p) +1)}\right]\right|^{1/( m_2(p) ( m_2(p) +1))} 
   \\
  & \leq d^{\nicefrac12} \max_{i\in \{1,\ldots,d\}} \left(\left|\mathbb E\left[\big\|\big(\tfrac{\partial}{\partial x} X_T^{1,x}\big)e_i\big\|^{m_2(p) ( m_2(p) +1)}\right]\right|^{1/( m_2(p) ( m_2(p) +1))}\right)
  \\  
  & = d^{\nicefrac12} \max_{i\in \{1,\ldots,d\}} \left(\left|\mathbb E\left[\big\|X^{1,e_i}_T - X^{1,0}_T\big\|^{m_2(p) ( m_2(p) +1)}\right]\right|^{1/( m_2(p) ( m_2(p) +1))}\right)\,.
\end{split}
\end{equation}
Combining this 
with
\eqref{eq:pde-nnet-pf-grad}
verifies that 
\begin{equation}
\begin{split}
&  \sup_{x \in [a,b]^d} 
  \left|
  \E\!\left[
  \limsup_{ \R^d \backslash \{0\} \ni h \to 0 }
  \Bigg(
  \frac
  { | \phi(X_T^{1,x+h}) - \phi(X_T^{1,x})|^{ m_2(p) } }
  {\|h\|^{ m_2(p) }}
  \Bigg)
  \right]
  \right|^{\nicefrac{1}{m_2(p)}}
\\ & \leq 
  cd^{w+\nicefrac12} 
  \max_{i\in \{1,\ldots,d\}} \left(\left|\mathbb E\left[\big\|X^{1,e_i}_T - X^{1,0}_T\big\|^{m_2(p) ( m_2(p) +1)}\right]\right|^{1/( m_2(p) ( m_2(p) +1))}\right)\\ & \cdot 
  \left[ 1+
  \sup_{x \in [a,b]^d} 
  \left|
  \E\!
  \left[ 
  \big\|X_T^{1,x}\big\|^{m_1(\mathbf{w}({ m_2(p) }+1))}
  \right]
  \right|^{\mathbf{w}/m_1(\mathbf{w}({ m_2(p) }+1))}
  \right] 
.
\end{split}
\end{equation}
Item \eqref{it:pde-nnet-pf1},
Corollary~\ref{cor:apriori1} 
(applied with $m \leftarrow d$, $C \leftarrow \|\mu(0)\|$,  $c \leftarrow l$, $A \leftarrow \sqrt{A}$, $W \leftarrow W^1$,  $X \leftarrow X^{1,x}$, $p \leftarrow m_1(\mathbf{w}({ m_2(p) }+1))$ for $x\in[a,b]^d$ in the notation of Corollary~\ref{cor:apriori1}),
Lemma~\ref{lem:sde-lp-bound2}
(applied with
$m \leftarrow d$,
$l \leftarrow l$,
$A \leftarrow \sqrt{A}$,
$W \leftarrow W^1$,
$X \leftarrow X^1$,
$ x \leftarrow e_i$,
$ y \leftarrow 0$,
$p \leftarrow { m_2(p) }({ m_2(p) }+1)$
for $ i \in \{1,\ldots,d\} $ 
in the notation of Lemma~\ref{lem:sde-lp-bound2}),
and
the fact that 
$
m_1(m_1(\mathbf{w}({ m_2(p) }+1)) - 1) = m_1(\mathbf{w}({ m_2(p) }+1) -1 )
$
hence yield that 
\begin{equation}
\label{eq:pde-nnet-diffb}
\begin{split}
&  \sup_{x \in [a,b]^d} 
  \left|
  \E\!\left[
  \limsup_{ \R^d \backslash \{0\} \ni h \to 0 }
  \Bigg(
  \frac
  { | \phi(X_T^{1,x+h}) - \phi(X_T^{1,x})|^{ m_2(p) } }
  {\|h\|^{ m_2(p) }}
  \Bigg)
  \right]
  \right|^{\nicefrac{1}{m_2(p)}}
\\ & \leq 
  cd^{w+\nicefrac12} \max_{i\in\{1,\ldots,d\}} \left(e^{lT} \|e_i\|\right)
  \Big(
  1+ 
  e^{\mathbf{w} lT} 
  \big[
  T\|\mu(0)\|
\\ & +
  \sqrt{ m_1(\mathbf{w}({ m_2(p) }+1)-1) \, T \operatorname{Trace}(A)}
  +
  \sup_{x \in [a,b]^d} \|x\| 
  \big]^{\mathbf{w}}
  \Big) 
\\ & =
  cd^{w+\nicefrac{1}{2} } e^{lT}
  \Big(
  1+ 
  e^{\mathbf{w} lT}
  \big[
  T \|\mu(0)\|
\\ & +
  \sqrt{ m_1(\mathbf{w}{ m_2(p) } 
  +
  \mathbf{w} - 1) \, T \operatorname{Trace}(A)}
  +
  \sup_{x \in [a,b]^d} \|x\| 
  \big]^{\mathbf{w}}
  \Big)
  .
\end{split}
\end{equation}
Combining this,
\eqref{eq:pde-nnet-lqb1},
and
\eqref{eq:pde-nnet-Xfct}
with
item \eqref{it:conc2} in Lemma~\ref{lem:Sobolev2}
(applied with $n \leftarrow n$,
$ 
\xi_{ i } \leftarrow ((\R^d \times \Omega) \ni (x,\omega) \mapsto \phi(X_T^{i,x }(\omega)) \in \R)
$
for $ i \in \{1,\ldots, n \}$
in the notation of Lemma~\ref{lem:Sobolev2})
implies that
\begin{equation}
\begin{split}
& \left|
  \E\!
  \left[
  \sup_{x \in [a,b]^d} 
  \left| 
  \E\big[\phi(X_T^{1,x})\big] 
  - \left[
  \frac{1}{n} 
  \sum_{k=1}^{n} 
  \phi(X_T^{k,x})
  \right]
  \right|^p
  \right]
  \right|^{1/p} 
\\ & \leq
\frac{4 \mathfrak{K} \zeta}{\sqrt{n}}  \Biggl(   
cd^z  \Big(
  1+
  e^{\mathbf{z} lT}
  \big[
  T\|\mu(0)\|+ \sqrt{ m_1(\mathbf{z}m_2(p)-1) \, T \operatorname{Trace}(A)} 
\\ & + 
  \sup_{x \in [a,b]^d} \|x\| 
  \big]^{\mathbf{z}}
  \Big)
+
  (b-a)
  cd^{w+\nicefrac{1}{2} } e^{lT}
  \Big(
  1+ 
  e^{\mathbf{w} lT}
  \big[
  T \|\mu(0)\|
\\ & +
  \sqrt{ m_1(\mathbf{w}{ m_2(p) } 
  +
  \mathbf{w} -1) \, T \operatorname{Trace}(A)} 
  +
  \sup_{x \in [a,b]^d} \|x\| 
  \big]^{\mathbf{w}}
  \Big)
  \Biggr)
.
\end{split}
\end{equation}
Grohs et al.~\cite[Proposition 3.3]{HornungJentzen2018} hence verifies that  there exists $\omega_n \in \Omega$ which satisfies that
\begin{equation}
\begin{split}
& \sup_{x \in [a,b]^d} \left| \E\big[\phi(X_T^{1,x})\big] 
- \left[\frac{1}{n} \sum_{k=1}^{n} \phi(X_T^{k,x}(\omega_n))\right] \right|
\\ & \leq
\frac{4 \mathfrak{K} \zeta}{\sqrt{n}}  \Biggl(   
cd^z  \Big(
  1+ 
  e^{\mathbf{z} lT}
  \big[
  T\|\mu(0)\|+ \sqrt{ m_1(\mathbf{z}m_2(p)-1) \, T \operatorname{Trace}(A)} 
\\ & + 
  \sup_{x \in [a,b]^d} \|x\| 
  \big]^{\mathbf{z}}
  \Big)
+
  (b-a)
  cd^{w+\nicefrac{1}{2} } e^{lT}
  \Big(
  1+
  e^{\mathbf{w} lT}
  \big[
  T \|\mu(0)\|
\\ & + 
  \sqrt{ m_1(\mathbf{w}{ m_2(p) } 
  + 
  \mathbf{w} -1) \, T \operatorname{Trace}(A)} 
  +
  \sup_{x \in [a,b]^d} \|x\| 
  \big]^{\mathbf{w}}
  \Big)
  \Biggr)
.
\end{split}
\end{equation}
Combining this,  
\eqref{eq:pde-nnet-xaff},
\eqref{eq:pde-nnet-fcf},
and
\eqref{eq:pde-nnet-err1}
with
the triangle inequality yields that
\begin{equation}
\begin{split}
&  \sup_{x \in [a,b]^d} \left| u(T,x)
  - 
  \left[\frac{1}{n} 
  \sum_{k=1}^{n}
  \phi(\mathcal{W}_{k,\omega_n}x+\mathcal{B}_{k,\omega_n})
  \right] \right|
\\ & =  
  \sup_{x \in [a,b]^d} \left| \E[\varphi(X_T^{1,x})] 
  - \left[\frac{1}{n} \sum_{k=1}^{n} \phi(X_T^{k,x}(\omega_n))\right] \right|
\\ & \leq  
  \sup_{x \in [a,b]^d} \left| \E[\varphi(X_T^{1,x})]
  - \E[\phi(X_T^{1,x})]\right| 
\\ & +
  \sup_{x \in [a,b]^d} \left| \E[\phi(X_T^{1,x})] 
  - \left[\frac{1}{n} \sum_{k=1}^{n} \phi(X_T^{k,x}(\omega_n))\right] \right| 
\\ & \leq
  \varepsilon d^v  \Big(
  1+ 
  e^{\mathbf{v} lT}
  \big[
  T\|\mu(0)\|+ \sqrt{ m_1(\mathbf{v}-1) \, T \operatorname{Trace}(A)} 
\\ & + 
  \sup_{x \in [a,b]^d} \|x\| 
  \big]^{\mathbf{v}}
  \Big) 
\\ & + 
\frac{4 \mathfrak{K} \zeta}{\sqrt{n}}  \Biggl(   
cd^z  \Big(
  1+ 
  e^{\mathbf{z} lT}
  \big[
  T\|\mu(0)\|+ \sqrt{ m_1(\mathbf{z} m_2(p) -1) \, T \operatorname{Trace}(A)} 
\\ & + 
  \sup_{x \in [a,b]^d} \|x\| 
  \big]^{\mathbf{z}}
  \Big)
+
  (b-a)
  cd^{w+\nicefrac{1}{2} } e^{lT}
  \Big(
  1+
  e^{\mathbf{w} lT}
  \big[
  T \|\mu(0)\|
\\ & + 
  \sqrt{ m_1(\mathbf{w}{ m_2(p) } 
  + 
  \mathbf{w} -1) \, T \operatorname{Trace}(A)} 
  +
  \sup_{x \in [a,b]^d} \|x\| 
  \big]^{\mathbf{w}}
  \Big)
  \Biggr)
.
\end{split}
\end{equation}
This completes the proof of Lemma~\ref{lem:mc-nnet-ap}.
\end{proof}


\begin{cor}
\label{cor:pde-nnet-ap4}
Let $d, n \in \N$,
$\varphi \in C(\R^d,\R)$, 
$\alpha, a \in \R$, $\beta \in [0,\infty)$,
$b \in (a,\infty)$,
$\varepsilon,T, c \in (0,\infty)$, $v,\mathbf{v},w,\mathbf{w}, z,\mathbf{z}  \in [0,\infty)$, $\mu \in \R^d$,
let $\left \| \cdot \right \| \colon \R^d \to [0,\infty)$ be the standard norm on $\R^d$,
let $\phi \in C^1(\R^d,\R)$, let $A = (A_{i,j})_{(i,j) \in \{1,\ldots,d\}^2} \in \R^{d \times d}$ be a symmetric and positive semi-definite matrix,
assume for all $ x \in \R^d$ that
\begin{equation}
  \left| \phi(x)\right| \leq c  d^z(1+\|x\|^{\mathbf{z}}), \qquad \| (\nabla \phi)(x)\|\leq  c  d^w(1+\|x\|^{\mathbf{w}}),
  \label{eq:pde-nnet-ap4-ass1}
\end{equation}
\begin{equation}
  \label{eq:pde-nnet-ap4-ass2}
\left| \varphi(x)-\phi(x)\right| \leq \varepsilon  d^v(1+\|x\|^{\mathbf{v}}),
\qquad
\sqrt{\operatorname{Trace}(A)} \leq c d^{\beta}, 
\end{equation}
and $\|\mu\| \leq c  d^{\alpha}$, let $u \in  C([0,T] \times \R^d, \R)$, assume for all $x \in \R^d$ that $u(0,x) = \varphi(x)$,
assume that 
$
\inf_{\gamma \in (0,\infty)} \sup_{(t,x) \in [0,T] \times \R^d} \big( \frac{|u(t,x)|}{1 +\|x\|^{\gamma}} \big) 
< \infty
$,
and assume that $u|_{(0,T) \times \R^d}$ is a viscosity solution of
\begin{align}
(\tfrac{\partial }{\partial t}u)(t,x) &= \smallsum_{i,j=1}^d A_{i,j}\, (\tfrac{\partial^2 }{\partial x_i \partial x_j}u)(t,x) + 
\smallsum_{i=1}^d \mu_i\, (\tfrac{\partial}{\partial x_i} u)(t,x)
\end{align}
for $(t,x) \in (0,T) \times \R^d$.
Then there exist $\WE_1,\ldots,\WE_{n} \in \R^{d \times d}$, $\BI_1,\ldots,\BI_{n} \in \R^d$ such that
\begin{equation}
\begin{split}
& \sup_{x \in [a,b]^d} \left|
  u(T,x) 
  - \left[\frac{1}{n}\sum_{k=1}^{n}\, \phi( \WE_k x + \BI_k)\right]
  \right| 
\\ & \leq
  \varepsilon
  d^{ v +  \mathbf{v} \max\{\alpha, \beta, \nicefrac{1}{2}\} } 
\\ & \cdot
  \Big(
  1+ 
  \big[ 
  \sqrt{2}
  \max\{1,T\}
  \max\{1, \sqrt{  \mathbf{v} } \}
  ( 2c + \max\{|a|,|b|\} )
  \big]^{\mathbf{v}}
  \Big) 
\\ & + 
d^{ 1 +   
  \max\{  
  z + 
  \mathbf{z} \max\{\alpha,\beta + 1 \}
  ,
  w + \nicefrac{1}{2}  
  +
  \mathbf{w} \max\{\alpha,\beta + 1 \}
  \} 
  }
\\ & \cdot
  \frac{32 \sqrt{e} c
  }
  {\sqrt{n}}  
  \Big[
  (
  1+ 
  (b-a)
  )
\\ & \cdot 
  \Big(
  1+ 
  \big[
  \sqrt{6}
  \max\{1,T\} \max\{1,\sqrt{ \mathbf{z} }, \sqrt{ \mathbf{w} } \}
  (2c+ 
  \max\{|a|,|b|\})
  \big]^{ \max\{\mathbf{z}, \mathbf{w} \} }
  \Big)
  \Big]
.
\end{split}
\end{equation}
\end{cor}
\begin{proof}[Proof of Corollary~\ref{cor:pde-nnet-ap4}]
Throughout this proof let 
$ \mathfrak{K}  \in (0,\infty) $ be the $\max\{2,d^2\} $-Kahane--Khintchine constant
(cf. Definition~\ref{def:p2-kahane-khintchine} and Lemma~\ref{lem:kahane-const}).
Observe that for all $ x \in [a,b]^d$ we have that
\begin{equation}
\| x \|
=
\left[\sum_{i=1}^d |x_i|^2 \right]^{ \nicefrac{1}{2} }
\leq
\left[ \sum_{i=1}^d [\max\{|a|,|b|\}]^2 \right]^{ \nicefrac{1}{2} }
=
d^{\nicefrac{1}{2} }
\max\{|a|,|b|\}
.
\end{equation}
This proves that
\begin{equation}
\label{eq:pde-nnet-ap4-x}
\sup_{ x \in [a,b]^d } \| x \| 
\leq
d^{\nicefrac{1}{2} }
\max\{|a|,|b|\}
.
\end{equation}
Next note that 
Lemma~\ref{lem:kahane-const}
(applied with $p \leftarrow \max\{2,d^2\}$ in the notation of Lemma~\ref{lem:kahane-const})
ensures that
\begin{equation}
\mathfrak{K} \leq \sqrt{\max\big\{1,\max\{2,d^2\}-1\big\}} \leq d.
\end{equation}
Combining this,
\eqref{eq:pde-nnet-ap4-x},
Corollary~\ref{cor:sob-const3},
\eqref{eq:pde-nnet-ap4-ass2},
and 
the hypothesis that 
$\|\mu\| \leq c  d^{\alpha}$
with
Lemma~\ref{lem:mc-nnet-ap} 
(applied with
$l \leftarrow 0$, $n \leftarrow n$,
$ \zeta \leftarrow 8 \sqrt{e} $,
$ p \leftarrow d^2$,
$\mu \leftarrow (\R^d \ni x \mapsto \mu \in \R^d)$,
$A \leftarrow 2A$
in the notation of Lemma~\ref{lem:mc-nnet-ap})
yields that 
there exist $\WE_1,\ldots,\WE_n \in \R^{d \times d}$, $\BI_1,\ldots,\BI_n \in \R^d$ which satisfy that

\begin{equation}
\label{eq:pde-nnet-ap4-1}
\begin{split}
& \sup_{x \in [a,b]^d} 
  \left|
  u(T,x) 
  - \left[\frac{1}{n}\sum_{k=1}^{n}\, \phi( \WE_k x + \BI_k)\right]
  \right| 
\\ & \leq
  \varepsilon d^v  \Big(
  1+ 
  \big[
  Tcd^{\alpha}+ \sqrt{ 2 \max\{1,\mathbf{v}-1\} \, T} 
  cd^{\beta}
\\ & +
  d^{\nicefrac{1}{2}}
  \max\{|a|,|b|\} 
  \big]^{\mathbf{v}}
  \Big) 
\\ & + 
\frac{32 \sqrt{e} d}{\sqrt{n}} 
\Biggl(   
cd^z  \Big(
  1+ 
  \big[
  Tcd^{\alpha}+ \sqrt{ 2 \max\{1,\mathbf{z}\max\{2,d^2\}-1\} \, T}
  cd^{\beta}
\\ & + 
  d^{\nicefrac{1}{2} }
  \max\{|a|,|b|\} 
  \big]^{\mathbf{z}}
  \Big)
+
  (b-a)
  cd^{w+\nicefrac{1}{2} }
  \Big(
  1+ 
  \big[
  T cd^{\alpha}
\\ & +
  \sqrt{ 2 \max\{1,\mathbf{w}{ \max\{2,d^2\} } + \mathbf{w} - 1\} \, T}
  cd^{\beta}
  + 
  d^{\nicefrac{1}{2} }
  \max\{|a|,|b|\} 
  \big]^{\mathbf{w}}
  \Big)
  \Biggr)
.
\end{split}
\end{equation}
Next note that
\begin{equation}
\label{eq:pde-nnet-ap4-2}
\begin{split}
  \max\{1, \mathbf{z} \max\{2,d^2\} - 1 \} 
& \leq 
  \max\{1,\mathbf{z}\} \max\{2,d^2\}
\leq
  \max\{1,\mathbf{z}\} (d^2 + 1) \\
& \leq
 2 \max\{1,\mathbf{z}\} d^2
.
\end{split}
\end{equation}
Moreover, note that 
\begin{equation}
\begin{split}
&  \max\{1,\mathbf{w}{ \max\{2,d^2\} } + \mathbf{w} - 1 \}
\\ &  \leq
  \max\{1,\mathbf{w}\}{ \max\{2,d^2\} } + \max\{1,\mathbf{w}\}
\leq
  \max\{1,\mathbf{w}\}{ (d^2 + 1 )} + \max\{1,\mathbf{w}\}\\
& \leq 
  2 \max\{1,\mathbf{w}\}d^2 + \max\{1,\mathbf{w}\}
\leq
  3 \max\{1,\mathbf{w}\} d^2
.
\end{split}
\end{equation}
Combining this,
the fact that $ \max\{1, \mathbf{v} - 1\} \leq \max\{1,\mathbf{v}\}$,
\eqref{eq:pde-nnet-ap4-1},
and
\eqref{eq:pde-nnet-ap4-2}
yields that

\begin{equation}
\begin{split}
& \sup_{x \in [a,b]^d} \left|
  u(T,x) 
  - \left[\frac{1}{n}\sum_{k=1}^{n}\, \phi( \WE_k x + \BI_k)\right]
  \right| 
\\ & \leq
  \varepsilon 
  d^{ v +  \mathbf{v} \max\{\alpha, \beta, \nicefrac{1}{2}\} } 
  \Big(
  1+ 
  \big[
  Tc+ \sqrt{ 2 \max\{1,\mathbf{v}\}  T } c
+ 
  \max\{|a|,|b|\} 
  \big]^{\mathbf{v}}
  \Big) 
\\ & + 
  d^{ 1 +   
  \max\{  
  z + 
  \mathbf{z} \max\{\alpha,\beta + 1 \}
  ,
  w + \nicefrac{1}{2}  
  +
  \mathbf{w} \max\{\alpha,\beta + 1 \}
  \} 
  }
\\ & \cdot 
  \frac{32 \sqrt{e} c
  }
  {\sqrt{n}}  
  \Biggl(   
  \Big(
  1+ 
  \big[
  Tc+ \sqrt{ 4 \max\{1,\mathbf{z}\}  T } c
+ 
  \max\{|a|,|b|\} 
  \big]^{\mathbf{z}}
  \Big)
\\ & +
  (b-a)
  \Big(
  1+ 
  \big[
  T c
+
  \sqrt{ 6 \max\{1,\mathbf{w}\}  T } c + \max\{|a|,|b|\} 
  \big]^{\mathbf{w}}
  \Big)
  \Biggr)
.
\end{split}
\end{equation}
Hence, we obtain that
\begin{equation}
\begin{split}
& \sup_{x \in [a,b]^d} \left|
  u(T,x) 
  - \left[\frac{1}{n}\sum_{k=1}^{n}\, \phi( \WE_k x + \BI_k)\right]
  \right| 
\\ & \leq
  \varepsilon 
  d^{ v +  \mathbf{v} \max\{\alpha, \beta, \nicefrac{1}{2}\} } 
\\ & \cdot
  \Big(
  1+ 
  \big[ 
  \sqrt{2}
  \max\{1,T\}
  \max\{1, \sqrt{ \mathbf{v} } \}
  ( 2c + \max\{|a|,|b|\} )
  \big]^{\mathbf{v}}
  \Big) 
\\ & + 
d^{ 1 +   
  \max\{  
  z + 
  \mathbf{z} \max\{\alpha,\beta + 1 \}
  ,
  w + \nicefrac{1}{2}  
  +
  \mathbf{w} \max\{\alpha,\beta + 1 \}
  \} 
  }
\\ & \cdot
  \frac{32 \sqrt{e} c
  }
  {\sqrt{n}}  
  \Big[
  (
  1+ 
  (b-a)
  )
\\ & \cdot 
  \Big(
  1+ 
  \big[\sqrt{6}
  \max\{1,T\} \max\{1,\sqrt{ \mathbf{z} }, \sqrt{ \mathbf{w} } \}
  (2c+ 
  \max\{|a|,|b|\})
  \big]^{ \max\{\mathbf{z}, \mathbf{w} \} }
  \Big)
  \Big]
.
\end{split}
\end{equation}
This completes the proof of Corollary~\ref{cor:pde-nnet-ap4}.
\end{proof}
\begin{cor}
\label{cor:pde-nnet-ap5}
Let $d, n \in \N$, $\varphi \in C(\R^d,\R)$,  $\alpha, a \in \R$, 
$ \beta \in [0,\infty)$,
$b \in (a,\infty)$,   $\varepsilon,T \in (0,\infty)$, 
$ c \in [\nicefrac{1}{2},\infty)$,
$v,\mathbf{v},w,\mathbf{w}, z,\mathbf{z}  \in [0,\infty)$, $\mu \in \R^d$, 
let $\left \| \cdot \right \| \colon \R^d \to [0,\infty)$ be the standard norm on $\R^d$,
let $A = (A_{i,j})_{(i,j) \in \{1,\ldots,d\}^2} \in \R^{d \times d}$ be a symmetric and positive semi-definite matrix,
let $\phi \in C^1(\R^d, \R)$, 
assume for all $ x \in \R^d$ that
\begin{equation}
  \left| \phi(x)\right| \leq c  d^z(1+\|x\|^{\mathbf{z}}), \qquad \| (\nabla \phi)(x)\|\leq  c  d^w(1+\|x\|^{\mathbf{w}}),
  \label{eq:pde-nnet-ap5-ass1}
\end{equation}
\begin{equation}
  \label{eq:pde-nnet-ap5-ass2}
\left| \varphi(x)-\phi(x)\right| \leq \varepsilon  d^v(1+\|x\|^{\mathbf{v}}),
\qquad
\sqrt{\operatorname{Trace}(A)} \leq c d^{\beta}, 
\end{equation}
and $\|\mu\| \leq c  d^{\alpha}$, let $u \in  C([0,T] \times \R^d, \R)$, assume for all $x \in \R^d$ that $u(0,x) = \varphi(x)$,
assume that 
$
\inf_{\gamma \in (0,\infty)} \sup_{(t,x) \in [0,T] \times \R^d} \big( \frac{|u(t,x)|}{1 +\|x\|^{\gamma}} \big) 
< \infty
$,
and assume that $u|_{(0,T) \times \R^d}$ is a viscosity solution of
\begin{align}
(\tfrac{\partial }{\partial t}u)(t,x) &= \smallsum_{i,j=1}^d A_{i,j}\, (\tfrac{\partial^2 }{\partial x_i \partial x_j}u)(t,x) + \smallsum_{i=1}^d \mu_i\, (\tfrac{\partial}{\partial x_i} u)(t,x)
\end{align}
for $(t,x) \in (0,T) \times \R^d$.
Then there exist $\WE_1,\ldots,\WE_{n} \in \R^{d \times d}$, $\BI_1,\ldots,\BI_{n} \in \R^d$ such that
\begin{equation}
\label{eq:pde-nnet-ap5}
\begin{split}
& \sup_{x \in [a,b]^d} \left|
  u(T,x) 
  - \left[\frac{1}{n}\sum_{k=1}^{n}\, \phi( \WE_k x + \BI_k)\right]
  \right| 
\\ & \leq 
  \varepsilon  
  d^{ v +  \mathbf{v} \max\{\alpha, \beta, \nicefrac{1}{2}\} } 
  \,\big[ 5c + T + \sqrt{ \mathbf{v} } + |a| + |b| \big]^{ 4 \mathbf{v} +1 }
\\ & +
 d^{ 1 +   
  \max\{  
  z + 
  \mathbf{z} \max\{\alpha,\beta + 1 \}
  ,
  w + \nicefrac{1}{2}  
  +
  \mathbf{w} \max\{\alpha,\beta + 1 \}
  \} 
  }
\\ & \cdot
  \frac{1}{\sqrt{ n }}
 \big[
  7c + T + \sqrt{ \mathbf{z} } + \sqrt{  \mathbf{w} } + |a| + |b| 
  \big]^{ 5 + 4 ( \mathbf{z} + \mathbf{w} ) }
.
\end{split}
\end{equation}   
\end{cor}
\begin{proof}[Proof of Corollary~\ref{cor:pde-nnet-ap5}]
Observe that Corollary~\ref{cor:pde-nnet-ap4} ensures that 
there exist $\WE_1,\ldots,\WE_{n} \in \R^{d\times d}$, $\BI_1,\ldots, \BI_{n} \in \R^d$ which satisfy 	that
\begin{equation}
\label{eq:ap4-pf1}
\begin{split}
& \sup_{x \in [a,b]^d} \left|
  u(T,x) 
  - \left[\frac{1}{n}\sum_{k=1}^{n}\, \phi( \WE_k x + \BI_k)\right]
  \right| 
\\ & \leq
  \varepsilon 
   d^{ v +  \mathbf{v} \max\{\alpha, \beta, \nicefrac{1}{2}\} } 
\\ & \cdot
  \Big(
  1+ 
  \big[ 
  \sqrt{2}
  \max\{1,T\}
  \max\{1, \sqrt{  \mathbf{v} } \}
  ( 2c + \max\{|a|,|b|\} )
  \big]^{\mathbf{v}}
  \Big) 
\\ & + 
  d^{ 1 +   
  \max\{  
  z + 
  \mathbf{z} \max\{\alpha,\beta + 1 \}
  ,
  w + \nicefrac{1}{2}  
  +
  \mathbf{w} \max\{\alpha,\beta + 1 \}
  \} 
  }
  \frac{32 \sqrt{e} c
  }
  {\sqrt{n}}  
  \Big[
  (
  1+ 
  (b-a)
  )
\\ & \cdot 
  \Big(
  1+ 
  \big[
  \sqrt{6}
  \max\{1,T\} \max\{1,\sqrt{ \mathbf{z} }, \sqrt{ \mathbf{w} } \}
  (2c+ 
  \max\{|a|,|b|\})
  \big]^{ \max\{\mathbf{z}, \mathbf{w} \} }
  \Big)
  \Big]
.
\end{split}
\end{equation}
Next note that the assumption that $c \in [\nicefrac{1}{2},\infty )$ implies that 
\begin{equation}
\label{eq:ap4-pf2}
\begin{split}
& 1+ 
  \big[ 
  \sqrt{2}
  \max\{1,T\}
  \max\{1, \sqrt{ \mathbf{v} } \}
  ( 2c + \max\{|a|,|b|\} )
  \big]^{\mathbf{v}}
\\ & \leq
  1 + 
  \max\{\sqrt{2}, T, \sqrt{ \mathbf{v}},2c + \max\{|a|,|b|\}  \}^{4 \mathbf{v} }
\\ & \leq 
  1 + \big[ 2\sqrt{2}c +2c + T + \sqrt{  \mathbf{v} } + |a| + |b| \big]^{ 4 \mathbf{v} }
\\ & \leq
  2 \big[ 5c + T + \sqrt{  \mathbf{v} } + |a| + |b| \big]^{ 4 \mathbf{v} }
\\ & \leq
  \big[ 5c + T + \sqrt{  \mathbf{v} } + |a| + |b| \big]^{ 4 \mathbf{v} +1}
.
\end{split}
\end{equation}
Moreover, note that
the assumption that $c \in [\nicefrac{1}{2},\infty )$
verifies that 
\begin{equation}
\label{eq:ap4-pf3}
\begin{split}
& 32 \sqrt{e} c
  \Big[
  (
  1+ 
  (b-a)
  )
\\ & \cdot 
  \Big(
  1+ 
  \big[
  \sqrt{6}
  \max\{1,T\} \max\{1, \sqrt{ \mathbf{z} }, \sqrt{\mathbf{w} } \}
  (2c+ 
  \max\{|a|,|b|\})
  \big]^{ \max\{\mathbf{z}, \mathbf{w} \} }
  \Big)
  \Big]
\\ & \leq 
  53 c ( 1 + |a| + |b|)
  \Big[
  1
  +
  \big(
  2\sqrt{6}c + 2c + T + \sqrt{ \mathbf{z} } + \sqrt{ \mathbf{w} } + |a| + |b| 
  \big)^{ 4\max\{ \mathbf{z}, \mathbf{w} \} }
  \Big]
\\ & \leq 
  53 c ( 1 + |a| + |b|)
  \Big[
  1
  +
  \big(
  7c + T + \sqrt{ \mathbf{z} } + \sqrt{ \mathbf{w} } + |a| + |b| 
  \big)^{ 4\max\{ \mathbf{z}, \mathbf{w} \} }
  \Big]
.
\end{split}
\end{equation}
Hence, we obtain that
\begin{equation}
\begin{split}
& 32 \sqrt{e} c
  \Big[
  (
  1+ 
  (b-a)
  )
\\ & \cdot 
  \Big(
  1+ 
  \big[
  \sqrt{6}
  \max\{1,T\} \max\{1, \sqrt{ \mathbf{z} }, \sqrt{\mathbf{w} } \}
  (2c+ 
  \max\{|a|,|b|\})
  \big]^{ \max\{\mathbf{z}, \mathbf{w} \} }
  \Big)
  \Big]
\\ & \leq
  \big[
  7c + T + \sqrt{ \mathbf{z} } + \sqrt{  \mathbf{w} } + |a| + |b| 
  \big]^{ 5 + 4 ( \mathbf{z} + \mathbf{w} ) }
.
\end{split}
\end{equation}
Combining  this 
with \eqref{eq:ap4-pf1} and \eqref{eq:ap4-pf2} 
establishes \eqref{eq:pde-nnet-ap5}.
This completes the proof of Corollary~\ref{cor:pde-nnet-ap5}.
\end{proof}
\subsection{Qualitative error estimates}
\label{ssec:qlee}
In this subsection we provide in Proposition~\ref{prop:ql-ap1} below
a qualitative approximation result for viscosity solutions
(cf., for example, Hairer et al.\  \cite{hairer2015})
of Kolomogorov PDEs
with constant coefficient functions.
Informally speaking, we can think of the approximations in 
Proposition~\ref{prop:ql-ap1} as  
linear combinations of realizations of ANNs with a suitable
continuously differentiable activation function.
Proposition~\ref{prop:ql-ap1} will be employed in our proof of 
Proposition~\ref{prop:nnet-ap1} in Subsection~\ref{ssec:qee-nn} below.
\begin{prop}
\label{prop:ql-ap1}
Let 
$d \in \N$, 
$\varphi \in C(\R^d,\R)$, 
$\alpha, a \in \R$, 
$\beta \in [0,\infty)$,
$b \in (a,\infty)$,   
$r,T \in (0,\infty)$, 
$c \in \big[\nicefrac{1}{2},\infty\big)$,
$v,\mathbf{v},w,\mathbf{w}, z,\mathbf{z}  \in [0,\infty)$, 
$C = \frac{1}{2} [ 5c + T + \sqrt{  \mathbf{v} } + |a| + |b| ]^{ -4 \mathbf{v} -1 }$, 
$
\mathbf{C} 
= 
4
[
  7c + T + \sqrt{ \mathbf{z} } + \sqrt{  \mathbf{w} } + |a| + |b| 
 ]^{ 10 + 8 ( \mathbf{z} + \mathbf{w} ) }
$,
$p =  v +  \mathbf{v} \max\{\alpha, \beta, \nicefrac{1}{2}\} $,
$
\mathbf{p}
=
2 +   
  \max\{  
  2z + 
  2\mathbf{z} \max\{\alpha,\beta + 1 \}
  ,
  2w + 1
  +
  2\mathbf{w} \max\{\alpha,\beta + 1 \}
  \} 
$,
$\mu \in \R^d$, 
let $\left \| \cdot \right \| \colon \R^d \to [0,\infty)$ be the standard norm on $\R^d$,
let $\phi_{\varepsilon} \in C^1(\R^d,\R)$, $\varepsilon \in  (0,r]$,
let $A = (A_{i,j})_{(i,j) \in \{1,\ldots,d\}^2} \in \R^{d \times d}$ be a symmetric and positive semi-definite matrix,
and
assume for all $\varepsilon \in (0,r]$, $ x \in \R^d$ that
\begin{equation}
  \left| \phi_{\varepsilon}(x)\right| \leq c  d^z(1+\|x\|^{\mathbf{z}}), \qquad \| (\nabla \phi_{\varepsilon})(x)\|\leq  c  d^w(1+\|x\|^{\mathbf{w}}),
  \label{eq:ql-ap1-ass1}
\end{equation}
\begin{equation}
  \label{eq:ql-ap1-ass2}
\left| \varphi(x)-\phi_{\varepsilon}(x)\right| \leq \varepsilon c d^v(1+\|x\|^{\mathbf{v}}),
\qquad
\sqrt{\operatorname{Trace}(A)} \leq c d^{\beta}, 
\end{equation}
and $\|\mu\| \leq c  d^{\alpha}$.
Then
\begin{enumerate}[(i)]
\item \label{it:ql-ap1-1} there exists a unique $u \in  C([0,T] \times \R^d, \R)$ which satisfies for all $x \in \R^d$ that $u(0,x) = \varphi(x)$, 
which satisfies that 
$
\inf_{\gamma \in (0,\infty)} \sup_{(t,x) \in [0,T] \times \R^d} \big( \frac{|u(t,x)|}{1 +\|x\|^{\gamma}} \big) 
< \infty
$,
and which satisfies that $u|_{(0,T) \times \R^d}$ is a viscosity solution of
\begin{align}
(\tfrac{\partial }{\partial t}u)(t,x) &= \smallsum_{i,j=1}^d A_{i,j}\, (\tfrac{\partial^2 }{\partial x_i \partial x_j}u)(t,x) + \smallsum_{i=1}^d \mu_i\, (\tfrac{\partial}{\partial x_i} u)(t,x)
\end{align}
for $(t,x) \in (0,T) \times \R^d$
and
\item \label{it:ql-ap1-2}  it holds for all 
$\varepsilon \in (0,r]$, 
$n \in \N \cap [  \varepsilon^{-2}d^{\mathbf{p}}\mathbf{C}, \infty)$
that there exist $\WE_1,\ldots, \linebreak \WE_{n} \in \R^{d \times d}$, $\BI_1,\ldots,\BI_{n} \in \R^d$ such that
\begin{equation}
  \sup_{x \in [a,b]^d} \left|
u(T,x) 
- \left[\frac{1}{n}\sum_{k=1}^{n}\, \phi_{ \varepsilon c^{-1}d^{-p} C  }( \WE_k x + \BI_k)\right]
\right| 
\leq \varepsilon.
\end{equation}
\end{enumerate}
\end{prop}
\begin{proof}[Proof of Proposition~\ref{prop:ql-ap1}]
First, note that
\eqref{eq:ql-ap1-ass1}
and
\eqref{eq:ql-ap1-ass2}
ensure that
$ \varphi \colon \R^d $ \linebreak $\to \R $ is an at most polynomially growing function.
Grohs et al.~\cite[Corollary 2.23]{HornungJentzen2018} 
(see also Hairer et al.\  \cite[Corollary 4.17]{hairer2015}) 
hence
implies 
that there exists a unique continuous function $u \colon [0,T] \times \R^d \to \R$ 
which satisfies for all $x \in \R^d$ that $u(0,x) = \varphi(x)$, 
which satisfies that 
$
\inf_{\gamma \in (0,\infty)}  \sup_{(t,x) \in [0,T] \times \R^d} 
\big( \frac{|u(t,x)|}{1 +\|x\|^\gamma} \big) 
< \infty
$, 
and which satisfies that $u|_{(0,T) \times \R^d}$ is a viscosity solution of 
\begin{align}
\label{eq:ql-ap2-pf}
(\tfrac{\partial }{\partial t}u)(t,x) 
&= 
\smallsum_{i,j=1}^d A_{i,j}\, (\tfrac{\partial^2 }{\partial x_i \partial x_j}u)(t,x) + \smallsum_{i=1}^d \mu_i\, (\tfrac{\partial}{\partial x_i} u)(t,x)
\end{align}
for $(t,x) \in (0,T) \times \R^d$.
This establishes item \eqref{it:ql-ap1-1}. 
Next note that the hypothesis that $ c \in [\nicefrac{1}{2},\infty)$ verifies that 
\begin{equation}
\frac{C}{c}
=
\frac{1}{2c} \frac{1}{[ 5c + T + \sqrt{  \mathbf{v} } + |a| + |b| ]^{ 4 \mathbf{v} +1 }}
\leq
1
.
\end{equation}
This reveals that for all
$
\varepsilon \in (0,r] 
$
we have that 
$ \varepsilon c^{-1} d^{-p} C \leq \varepsilon $.
This and
\eqref{eq:ql-ap1-ass2}
prove that 
for all $ \varepsilon \in (0,r]$, $x\in \mathbb R^d$ 
we have that
\begin{equation}
\left| \varphi(x)-\phi_{ \varepsilon c^{-1} d^{-p} C }(x)\right| 
\leq
\varepsilon c^{-1} d^{-p} C  c d^v(1+\|x\|^{\mathbf{v}})
=
\varepsilon d^{-p} C d^v(1+\|x\|^{\mathbf{v}})
.
\end{equation}
Corollary~\ref{cor:pde-nnet-ap5}
(applied with $\varepsilon \leftarrow \varepsilon d^{-p}C$,
$ \phi \leftarrow \phi_{ \varepsilon c^{-1} d^{-p} C }$
for $ \varepsilon \in (0,r] $
in the notation of Corollary~\ref{cor:pde-nnet-ap5})
hence
yields that 
for all 
$ \varepsilon \in (0,r]$, 
$n \in \N \cap [  \varepsilon^{-2}d^{\mathbf{p}}\mathbf{C}, \infty)$
there exist 
$ \WE_1,\ldots,\WE_{n} \in \R^{d\times d} $, 
$\BI_1,\ldots, \BI_{n} \in \R^d$ 
such that
\begin{equation}
\begin{split}
& \sup_{x \in [a,b]^d} \left|
  u(T,x) 
  - \left[\frac{1}{n}\sum_{k=1}^{n}\, \phi_{\varepsilon c^{-1} d^{-p} C}( \WE_k x + \BI_k)\right]
  \right| 
\\ & \leq
  \varepsilon d^{-p}C \frac{d^p}{2C}  + 
  \frac{ \sqrt{\mathbf{C}d^{\mathbf{p}} }}{2 \sqrt{ n } }
\leq 
\frac{\varepsilon}{2} + \frac{\varepsilon}{2} = \varepsilon.
\end{split}
\end{equation}
This establishes item \eqref{it:ql-ap1-2}.
This completes the proof of Proposition~\ref{prop:ql-ap1}.
\end{proof}
\subsection{Qualitative error estimates for artificial neural networks (ANNs)}
\label{ssec:qee-nn}
In this subsection we prove in Corollary~\ref{cor:nnet-ap2} below 
that ANNs 
with continuously differentiable activation functions
can overcome the curse of dimensionality
in the uniform approximation of 
viscosity solutions 
of Kolmogorov PDEs 
with constant coefficient functions. 
Corollary~\ref{cor:nnet-ap2} is an immediate consequence of 
Proposition~\ref{prop:nnet-ap1} below. 
Proposition~\ref{prop:nnet-ap1}, in turn,
follows from  Proposition~\ref{prop:ql-ap1} above 
and the well-known fact that 
linear combinations of realizations of ANNs are again realizations of ANNs.
To formulate Corollary~\ref{cor:nnet-ap2} we introduce in Setting~\ref{setting:ANNs} below a common framework from the scientific literature (cf., e.g.,  Grohs et al.~\cite[Section~2.1]{GrohsHornungJentzen2019} and Petersen \& Voigtlaender~\cite[Section~2]{petersen2017optimal}) to mathematically describe ANNs.

\begin{setting}
\label{setting:ANNs}
Let $\mathbf{N}$  be the set given by 
\begin{equation}
\mathbf{N} = \cup_{L \in \N \cap [2,\infty)} \cup_{(l_0,\ldots,l_L) \in \\ ((\N^{L} )
\times \{1\})} 
\big(\times_{k=1}^L (\R^{l_k\times l_{k-1}}\times\R^{l_k})
\big),
\end{equation}
let $\af\in C^1(\R,\R)$,
let $\Af_n \colon \R^n \to \R^n$, $ n \in \N$, satisfy for all  $n \in \N$, $x=(x_1,\ldots,x_n) \in \R^n$ that 
$\Af_n(x) = (\af(x_1),\ldots,\af(x_n))$, 
and let $\mathcal{N}, \mathcal{L}, \mathcal{P}, \mathfrak{P} \colon \mathbf{N} \to \N$ and $\mathcal{R} \colon \mathbf{N} \to \cup_{d \in \mathbb{N}} C(\R^d,\R)$ satisfy for all $L \in \N \cap [2,\infty)$, 
$(l_0,\ldots,l_L) \in ((\N^{L}) \times \{1\})$,  
$\Phi = ((W_1,B_1),\ldots,(W_L,B_L))$
$ = (((W_1^{(i,j)})_{i \in \{1,\ldots,l_1\}, j \in \{1,\ldots,l_{0}\}},$ $(B_1^i)_{i\in \{1,\ldots,l_1\}}),$ $\ldots,((W_L^{(i,j)})_{i \in \{1,\ldots,l_L\}, j \in \{1,\ldots,l_{L-1}\}},
(B_k^i)_{i\in \{1,\ldots,l_L\}} ))$
$\in (\times_{k=1}^{L}$ $\linebreak (\R^{l_k \times l_{k-1}} \times \R^{l_k}))$, 	
$x_0 \in \R^{l_0}$, $\ldots$, $x_{L-1}  \in \R^{l_{L-1}}$ with 
$\forall \, k \in \N \cap (0,L) \colon x_k = \Af_{l_k}(\WE_k x_{k-1} + \BI_k)$  that
$\mathcal{N}(\Phi) = \sum_{k=0}^L l_k$,
$\mathcal{L}(\Phi) = L + 1$, 
$\mathcal{P}(\Phi) =  \smallsum_{k=1}^{L} l_k(l_{k-1} + 1)$,
$(\mathcal{R} \Phi) \in C(\mathbb{R}^{l_0},\mathbb{R})$,
$(\mathcal{R}\Phi)(x_0) =  \WE_L x_{L-1} + \BI_L$, and 
\begin{equation}
\mathfrak{P}(\Phi)  = 
\sum_{k=1}^L \sum_{i=1}^{l_k}
\left[
\mathbbm{1}_{\R \backslash \{0\}}(\BI_k^i) + \sum_{j=1}^{l_{k-1}} \mathbbm{1}_{\R \backslash \{0\}}(\WE_k^{(i,j)}) 
\right]
\,.
\end{equation}	
\end{setting}

\begin{prop}
\label{prop:nnet-ap1}
Assume Setting~\ref{setting:ANNs}, let 
$d \in \N$, 
$\mu \in  \R^d$,
$\varphi \in C(\R^d,\R)$, 
$\alpha, a \in \R$, 
$\beta \in [0,\infty)$,
$b \in (a,\infty)$,   
$r,T \in (0,\infty)$, 
$c \in \big[\nicefrac{1}{2},\infty\big)$,
$v,\mathbf{v},w,\mathbf{w}, z,\mathbf{z}  \in [0,\infty)$, 
let
\begin{equation}
\begin{split}
C & = 
(
4
[
  7c + T + \sqrt{ \mathbf{z} } + \sqrt{  \mathbf{w} } + |a| + |b| 
 ]^{ 10 + 8 ( \mathbf{z} + \mathbf{w} ) }
+1)(1+r^2)\,,\\
p & =
2 +   
  \max\{  
  2z + 
  2\mathbf{z} \max\{\alpha,\beta + 1 \}
  ,
  2w + 1
  +
  2\mathbf{w} \max\{\alpha,\beta + 1 \}
  \} 
\,,	\\
\mathcal{C} & = \tfrac{1}{2}[ 5c + T + \sqrt{  \mathbf{v} } + |a| + |b| ]^{ -4 \mathbf{v} -1}\,,\\ 
\mathbf{p} & = v +  \mathbf{v} \max\{\alpha, \beta, \nicefrac{1}{2}\}\,,\\
\end{split}
\end{equation} 
let $\left \| \cdot \right \| \colon \R^d \to [0,\infty)$ be the standard norm on $\R^d$,
let $(\phi_{\varepsilon})_{\varepsilon \in (0,r]} \subseteq \mathbf{N}$, 
let $A = (A_{i,j})_{(i,j) \in \{1,\ldots,d\}^2} \in \R^{d \times d}$  be a  symmetric and positive semi-definite matrix,
and assume for all $\varepsilon \in (0,r]$, $ x \in \R^d$ that $(\mathcal R\phi_\varepsilon) \in C(\mathbb R^d,\mathbb R)$,
\begin{equation}
\left| (\mathcal{R}\phi_{\varepsilon})(x)\right| \leq c  d^z(1+\|x\|^{\mathbf{z}}), \quad \| (\nabla (\mathcal{R}\phi_{\varepsilon}))(x)\|\leq  c  d^w(1+\|x\|^{\mathbf{w}}),
  \label{eq:nnet-ap1-ass1}
\end{equation}
\begin{equation}
  \label{eq:nnet-ap1-ass2}
  \left| \varphi(x)-(\mathcal{R}\phi_{\varepsilon})(x)\right| \leq \varepsilon c  d^v(1+\|x\|^{\mathbf{v}}),
\qquad
\sqrt{\operatorname{Trace}(A)} \leq c d^{\beta}, 
\end{equation}
and $\|\mu\| \leq c  d^{\alpha}$.
Then
\begin{enumerate}[(i)]
\item \label{it:nnet-ap1-1} there exists a unique $u \in  C([0,T] \times \R^d, \R)$ which satisfies for all $x \in \R^d$ that $u(0,x) = \varphi(x)$,
which satisfies that 
$
\inf_{\gamma \in (0,\infty)} \sup_{(t,x) \in [0,T] \times \R^d} \big( \frac{|u(t,x)|}{1 +\|x\|^{\gamma}} \big) 
< \infty
$,
and which satisfies that $u|_{(0,T) \times \R^d}$ is a viscosity solution of
\begin{align}
(\tfrac{\partial }{\partial t}u)(t,x) &= \smallsum_{i,j=1}^d A_{i,j}\, (\tfrac{\partial^2 }{\partial x_i \partial x_j}u)(t,x) + \smallsum_{i=1}^d \mu_i\, (\tfrac{\partial}{\partial x_i} u)(t,x)
\end{align}
for $(t,x) \in (0,T) \times \R^d$
and
\item \label{it:nnet-ap1-2}  there exists $(\psi_{\varepsilon})_{\varepsilon \in (0,r]} \subseteq \mathbf{N}$ such that 
for all $\varepsilon \in (0,r]$ we have that
$
\mathcal{N}(\psi_{\varepsilon}) \leq C d^{p} \varepsilon^{-2} \mathcal{N}(\phi_{ \varepsilon c^{-1} d^{-\mathbf{p}}\mathcal{C}  })
$,
$
\mathcal{L}(\psi_{\varepsilon}) = \mathcal{L}(\phi_{ \varepsilon c^{-1} d^{-\mathbf{p}}\mathcal{C}  })
$,
$
\mathcal{P}( \psi_{ \varepsilon  } )
\leq C^2 d^{2p}\varepsilon^{-4}\mathcal{P} (\phi_{ \varepsilon c^{-1} d^{-\mathbf{p}}\mathcal{C}  })
$,
$
\mathfrak{P}(\psi_{\varepsilon}) \leq 
C d^{p} \varepsilon^{-2} \mathcal{P}(\phi_{ \varepsilon c^{-1} d^{-\mathbf{p}}\mathcal{C}  })
$,
$(\mathcal R\psi_\varepsilon) \in C(\mathbb R^d,\mathbb R)$,
and
\begin{equation}
\sup_{x \in [a,b]^d} \left| u(T,x) -  ( \mathcal{R} \psi_{ \varepsilon  } )( x ) \right| \leq \varepsilon.
\end{equation}
\end{enumerate}
\end{prop}
\begin{proof}[Proof of Proposition~\ref{prop:nnet-ap1}]
Throughout this proof let 
$ \varepsilon \in (0,r]$, 
let
\begin{equation}
\begin{split}
\mathbf{C} & = 4[7c + T + \sqrt{ \mathbf{z} } + \sqrt{  \mathbf{w} } + |a| + |b| 
]^{ 10 + 8 ( \mathbf{z} + \mathbf{w} ) }\,,\\
\end{split}	
\end{equation}
let $n = \min(\mathbb N \cap [\mathbf{C}d^{p}\varepsilon^{-2},\infty))$, 
let 
$\gamma_1,\ldots,\gamma_{n} \in \R^{d \times d}$,
$\delta_1,\ldots, \delta_{n} \in \R^d$ satisfy
\begin{equation}
\label{eq:nnet-ap1-pf1}
  \sup_{x \in [a,b]^d} \left|
u(T,x) 
- \left[\frac{1}{n}\sum_{k=1}^{n}\, (\mathcal{R} \phi_{ \varepsilon c^{-1} d^{-\mathbf{p}}\mathcal{C}  })( \gamma_k x + \delta_k)\right]
\right| 
\leq \varepsilon
\end{equation}
(cf.\ item \eqref{it:ql-ap1-2} in  Proposition~\ref{prop:ql-ap1}),
let  $L \in \N\cap [2,\infty)$, $(l_0, \ldots, l_L) \in (\{d\} \times (\N^{L-1}) \times \{1\})$, 
$( (\WE_1,\BI_1),\ldots,(\WE_L,\BI_L)) \in (\times_{k=1}^{L} (\R^{l_k \times l_{k-1}} \times \R^{l_k}))$ 
satisfy
\begin{equation}
\label{eq:nnet-ap1-phi}
\phi_{\varepsilon c^{-1} d^{-\mathbf{p}}\mathcal{C}} = ((\WE_1,\BI_1),\ldots,(\WE_L,\BI_L)),
\end{equation}
and let
\begin{equation}
\label{eq:nnet-ap1-psi}
\begin{split}
\psi & = ((\mathcal{W}_1\mathcal{W}_0,\mathcal{W}_1\mathcal{B}_0 + \mathcal{B}_1),(\mathcal{W}_2,\mathcal{B}_2),\ldots,
( \mathcal{W}_{L-1},\mathcal{B}_{L-1}),
( \mathcal{W}_L,B_L))
\\ & \in
(\R^{nl_1 \times l_0} \times \R^{n l_1})
\times
( \times_{k=2}^{L-1} (\R^{n l_k \times nl_{k-1}} \times \R^{nl_k}))
\times 
(\R^{l_L \times nl_{L-1}} \times \R^{l_L})
\end{split}
\end{equation}
satisfy for all 
$k \in \{1,\ldots,L-1\}$
that
\begin{equation}
\label{eq:nnet-ap-defw}
\mathcal{W}_0 = 
\begin{pmatrix}
\gamma_1 \\
\gamma_2    \\
\vdots \\
\gamma_n  
\end{pmatrix}
,
\quad
\mathcal{W}_{k} = \operatorname{diag}(\WE_k,\ldots,\WE_k),
\quad
\mathcal{W}_{L} 
=
\frac{1}{n}
\begin{pmatrix}
\WE_L &   \ldots & \WE_L
\end{pmatrix}
,
\end{equation}
\begin{equation}
\label{eq:nnet-ap-defb}
\mathcal{B}_0 =
\begin{pmatrix}
\delta_1 \\
\vdots \\
\delta_n
\end{pmatrix}
,
\qquad
\text{and}
\qquad
\mathcal{B}_k = 
\begin{pmatrix}
B_k\\
\vdots \\
B_k
\end{pmatrix}
.
\end{equation}
Observe that item \eqref{it:ql-ap1-1} in 
Proposition~\ref{prop:ql-ap1} 
(applied with $ r \leftarrow r $, $ \phi_{\varepsilon} \leftarrow ( \mathcal{R} \phi_{\varepsilon} ) $
for $ \varepsilon \in (0,r] $
in the notation of Proposition~\ref{prop:ql-ap1})
establishes item \eqref{it:nnet-ap1-1}. 
Next note that 
the fact that 
$ n \in [\mathbf{C}d^p\varepsilon^{-2},\mathbf{C}d^p\varepsilon^{-2} + 1] $
and
the fact that $r^2\varepsilon^{-2} \in [1,\infty)$  prove that
\begin{equation}
\label{eq:nnet-ap1-n}
\begin{split}
n \leq \mathbf{C}d^p\varepsilon^{-2} + 1 &  \leq (\mathbf{C} + 1) d^p \max\{1,\varepsilon^{-2}\} \\
& \leq (\mathbf{C} + 1) \max\{1,r^{-2}\}r^2 d^p \varepsilon^{-2} \\
& \leq C  d^p \varepsilon^{-2}.
\end{split}
\end{equation}
This
and
\eqref{eq:nnet-ap1-psi}
verify that 
\begin{equation}
\label{eq:nnet-ap-neurons}
\begin{split}
\mathcal{N}(\psi) & = l_0 + \sum_{k=1}^{L-1} nl_k + l_L 
\leq n \sum_{k=0}^L l_k \\
& = n \mathcal{N}(\phi_{ \varepsilon c^{-1} d^{-\mathbf{p}}\mathcal{C}  })
\leq C   d^p \varepsilon^{-2} \mathcal{N}(\phi_{ \varepsilon c^{-1} d^{-\mathbf{p}}\mathcal{C}  }). 
\end{split}
\end{equation}
Next note that 
\eqref{eq:nnet-ap1-psi}
and
\eqref{eq:nnet-ap1-n}
yield that 
\begin{equation}
\label{eq:nnet-ap-weight1}
\begin{split}
\mathcal{P}(\psi) & = nl_1 l_0  +nl_1 
+ 
\sum_{k=2}^{L-1} nl_k(n l_{k-1} +1) + nl_L l_{L-1}+ l_L
\\ & \leq
n^2 \left[
l_1(l_0 + 1) + \sum_{k=2}^{L-1} l_k(l_{k-1} + 1) + l_L(l_{L-1} + 1)
\right]
\\ & =
n^2 \mathcal{P}(\phi_{ \varepsilon c^{-1} d^{-\mathbf{p}}\mathcal{C}  })
\leq 
C^2  d^{2p} \varepsilon^{-4} \mathcal{P}(\phi_{ \varepsilon c^{-1} d^{-\mathbf{p}}\mathcal{C}  }).
\end{split}
\end{equation}
Moreover, note that 
\eqref{eq:nnet-ap1-psi}
and
\eqref{eq:nnet-ap1-n} 
ensure that
\begin{equation}
\label{eq:nnet-ap-weight2}
\begin{split}
\mathfrak{P}(\psi)   
& \leq nl_1(l_0 + 1) + n \sum_{k=2}^L l_k(l_{k-1} + 1) \\
&  = n \mathcal{P}(\phi_{ \varepsilon c^{-1} d^{-\mathbf{p}}\mathcal{C}  }) 
\leq C  d^p \varepsilon^{-2} \mathcal{P}(\phi_{ \varepsilon c^{-1} d^{-\mathbf{p}}\mathcal{C}  }).
\end{split}
\end{equation}
Furthermore, 
\eqref{eq:nnet-ap1-psi}, \eqref{eq:nnet-ap-defw}, and~\eqref{eq:nnet-ap-defb}
imply that
for all  $x \in  \R^d$ we have that
\begin{equation}
(\mathcal{R} \psi)(x) 
=
\frac{1}{n}
\sum_{k=1}^{n}\, (\mathcal{R}\phi_{ \varepsilon c^{-1} d^{-\mathbf{p}}\mathcal{C}  })
( \gamma_k x + \delta_k).
\end{equation}
Combining this and~\eqref{eq:nnet-ap1-pf1} yields that
\begin{equation}
\label{eq:nnet-ap-diff}
\sup_{x \in [a,b]^d} \left| u(T,x) -  ( \mathcal{R} \psi )( x ) \right| \leq \varepsilon.
\end{equation}
Next observe that
\eqref{eq:nnet-ap1-phi}
and
\eqref{eq:nnet-ap1-psi}
verify that 
$\mathcal{L}(\psi) = L + 1 = \mathcal{L}(\phi_{ \varepsilon c^{-1} d^{-\mathbf{p}}\mathcal{C}  })$.
Combining this with 
\eqref{eq:nnet-ap-neurons},
\eqref{eq:nnet-ap-weight1},
\eqref{eq:nnet-ap-weight2},
and
\eqref{eq:nnet-ap-diff}
establishes item \eqref{it:nnet-ap1-2}.
This completes the proof of Proposition~\ref{prop:nnet-ap1}.
\end{proof}
\begin{cor}
\label{cor:nnet-ap2}
Assume Setting~\ref{setting:ANNs}, 
let $\alpha,c, a \in \R$, 
$\beta \in [0,\infty)$,
$b \in (a,\infty)$,   $r,T \in (0,\infty)$, 
$p$, $q$, $v$, $\mathbf{v}$, $w$, $\mathbf{w}$, $z$, $\mathbf{z} \in [0,\infty)$, 
$
\mathbf{p}
=
2 +   
  \max\{  
  2z + 
  2\mathbf{z} \max\{\alpha,\beta + 1 \}
  ,
  2w + 1
  +
  2\mathbf{w} \max\{\alpha,\beta + 1 \}
  \} 
$,
$
\mathfrak p = v +  \mathbf{v} \max\{\alpha, \beta, \nicefrac{1}{2}\}
$,
for every $d \in \N$ let $\left \| \cdot \right \|_{\R^d} \colon \R^d \to [0,\infty)$ be the standard norm on $\R^d$,
let $(\phi_{\varepsilon,d})_{(\varepsilon,d) \in (0,r]\times \N}  \subseteq \mathbf{N}$, 
let $ \mu_d \in \R^d$, $d \in \N$, 
let $A_d = (A_d^{(i,j)})_{(i,j) \in \{1,\ldots,d\}^2} \in \R^{d \times d}$, $d \in \N$, be   symmetric and positive semi-definite matrices,
let $\varphi_d \in C(\mathbb R^d, \mathbb R)$, $d\in \mathbb N$,
and assume for all $\varepsilon \in (0,r]$, $d \in \N$, $ x \in \R^d$ that
$(\mathcal R\phi_{\varepsilon,d})\in C(\mathbb R^d,\mathbb R)$,
\begin{equation}
\left| (\mathcal{R}\phi_{\varepsilon,d})(x)\right| \leq c  d^z(1+\|x\|_{\R^d}^{\mathbf{z}}), \qquad \| (\nabla (\mathcal{R}\phi_{\varepsilon,d}))(x)\|_{\R^d} \leq  c  d^w(1+\|x\|_{\R^d}^{\mathbf{w}}),
  \label{eq:nnet-ap2-ass1}
\end{equation}
\begin{equation}
  \label{eq:nnet-ap2-ass2}
  \left| \varphi_d(x)-(\mathcal{R}\phi_{\varepsilon,d})(x)\right| \leq \varepsilon c  d^v(1+\|x\|_{\R^d}^{\mathbf{v}}),
\qquad
\sqrt{\operatorname{Trace}(A_d)} \leq c d^{\beta}, 
\end{equation}
\begin{equation}
\label{eq:nnet-ap2-ass3}
\|\mu_d\|_{ \R^d} \leq c  d^{\alpha},
\qquad
\text{and}
\qquad
\mathcal{P}(\phi_{\varepsilon,d}) \leq c  d^p\varepsilon^{-q}.
\end{equation}
Then
\begin{enumerate}[(i)]
\item \label{it:nnet-ap2-1} there exist unique $u_d \in  C([0,T] \times \R^d, \R)$, $d \in \N$,  which satisfy for all $d \in \N$, $x \in \R^d$ that $u_d(0,x) = \varphi_d(x)$,
which satisfy for all $d \in \N$ that 
$
\inf_{\gamma \in (0,\infty)}  
\sup_{(t,x) \in [0,T] \times \R^d}\! 
\big( \frac{|u_d(t,x)|}{1 + \|x\|_{ \R^d } ^{\gamma}} \big)
< \infty
$,
and which satisfy that for all $d \in \N$ we have that $u_d|_{(0,T) \times \R^d}$ is a viscosity solution of
\begin{align}
(\tfrac{\partial }{\partial t}u_d)(t,x) 
&= 
(\tfrac{\partial}{\partial x}u_d)(t,x) \, \mu_d +\smallsum_{i,j=1}^d A_d^{(i,j)}\, (\tfrac{\partial^2 }{\partial x_i \partial x_j}u_d)(t,x) 
\end{align}
for $(t,x) \in (0,T) \times \R^d$
and
\item \label{it:nnet-ap2-2}
there exist $(\psi_{\varepsilon,d})_{(\varepsilon,d) \in (0,r]\times \N} \subseteq \mathbf{N}$, 
$C \in \R$ such that 
for all $\varepsilon \in (0,r]$, $d \in \N$ we have that
$
\mathcal{N}(\psi_{\varepsilon,d}) \leq C d^{p + \mathbf{p} + \mathfrak p q}\varepsilon^{-(q+2)}
$,
$
\mathcal{P}( \psi_{ \varepsilon,d  } )
\leq C d^{p + 2 \mathbf{p} + \mathfrak p q} \varepsilon^{-(q+4)}
$,
$
\mathfrak{P}(\psi_{\varepsilon,d}) 
\leq C d^{p+ \mathbf{p}+ \mathfrak p q} \varepsilon^{-(q+2)}
$,
$
(\mathcal{R}\psi_{\varepsilon,d}) \in C(\R^d,\R)
$,
and
\begin{equation}
\sup_{x \in [a,b]^d} \left| u_d(T,x) -  ( \mathcal{R} \psi_{ \varepsilon, d  } )( x ) \right| \leq \varepsilon.
\end{equation}
\end{enumerate}
\end{cor}
\begin{proof}[Proof of Corollary~\ref{cor:nnet-ap2}]
Throughout this proof 
let 
$
\mathbf{C} = (
4
[
  7\max\{ \nicefrac{1}{2},c\} + T + \sqrt{ \mathbf{z} } + \sqrt{  \mathbf{w} } + |a| + |b| 
 ]^{ 10 + 8 ( \mathbf{z} + \mathbf{w} ) }
+1)(1+r^2)
$,
$
\mathbf c = \max\{\nicefrac12,c\}
$,
$
\mathcal C = \frac12[5\mathbf c + T + \sqrt{\mathbf v} + |a| + |b|]^{-4\mathbf v - 1}$.
Note that item \eqref{it:nnet-ap1-1} in Proposition~\ref{prop:nnet-ap1} 
(applied with
$c \leftarrow \max\{\nicefrac{1}{2}, c\}$,
$\mu \leftarrow \mu_d$,
$\varphi \leftarrow \varphi_d$, $(\phi_{\varepsilon})_{\varepsilon \in (0,r]} \leftarrow (\phi_{\varepsilon,d})_{\varepsilon \in (0,r]}$,
$A \leftarrow A_d$,
for
$d \in \N$ in the notation of Proposition~\ref{prop:nnet-ap1}) 
implies 
that for every $d \in \N$ 
there exists a unique continuous function $u_d \colon [0,T] \times \R^d \to \R$ which satisfies for all $x \in \R^d$ that $u_d(0,x) = \varphi_d(x)$,
which satisfies that 
$
\inf_{\gamma \in (0,\infty)}  
\sup_{(t,x) \in [0,T] \times \R^d}
\big( \frac{|u_d(t,x)|}{1 +\|x\|_{ \R^d }^{\gamma}} \big)
< \infty$,   
and which satisfies that $u_d|_{(0,T) \times \R^d}$ is a viscosity solution of
\begin{align}
(\tfrac{\partial }{\partial t}u_d)(t,x) 
&=
(\tfrac{\partial}{\partial x} u_d)(t,x) \, \mu_d
+
\smallsum_{i,j=1}^d A_d^{(i,j)}\, (\tfrac{\partial^2 }{\partial x_i \partial x_j}u_d)(t,x) 
\end{align}
for $(t,x) \in (0,T) \times \R^d$. This  establishes item \eqref{it:nnet-ap2-1}.
Next note that for all $L \in \N \cap [2,\infty)$, $(l_0,\ldots,l_L) \in \N^L \times \{1\}$, 
$\Phi  \in (\times_{k=1}^L (\R^{l_k \times l_{k-1}} \times \R^{l_k}))$ we have that 
\begin{equation}
\mathcal{N}(\Phi) = \sum_{k=0}^L l_k \leq l_1l_0 + \sum_{k=1}^L l_k +\sum_{k=2}^L l_kl_{k-1} = \mathcal{P}(\Phi).
\label{eq:nnet-ap2-pf1}
\end{equation}
Moreover, note that Proposition~\ref{prop:nnet-ap1} 
(applied with
$c \leftarrow \max\{\nicefrac{1}{2}, c\}$,
$\mu \leftarrow \mu_d$,
$\varphi \leftarrow \varphi_d$, $(\phi_{\varepsilon})_{\varepsilon \in (0,r]} \leftarrow (\phi_{\varepsilon,d})_{\varepsilon \in (0,r]}$,
$A \leftarrow A_d$, for $d \in \N$ in the notation of Proposition~\ref{prop:nnet-ap1})
assures that there exist 
$\psi_{\varepsilon, d} \in \mathbf{N}$, 
$ d \in \N$, 
$ \varepsilon \in (0,r] $,
which satisfy that for every 
$ d \in \N$, 
$ \varepsilon \in (0,r] $
we have that
$
\mathcal{N}(\psi_{\varepsilon, d}) \leq \mathbf{C}  d^{\mathbf{p}} \varepsilon^{-2} \mathcal{N}(\phi_{\varepsilon \mathbf c^{-1} d^{-\mathfrak{p}}\mathcal{C},d})
$,
$
\mathcal{P}( \psi_{ \varepsilon, d  } )
\leq \mathbf{C}^2 d^{2\mathbf{p}}\varepsilon^{-4}\mathcal{P} ( \phi_{\varepsilon \mathbf c^{-1} d^{-\mathfrak{p}}\mathcal{C},d})
$,
$
\mathfrak{P}(\psi_{\varepsilon,d }) \leq 
\mathbf{C}   d^{\mathbf{p}} \varepsilon^{-2} \mathcal{P}(\phi_{\varepsilon \mathbf c^{-1} d^{-\mathfrak{p}}\mathcal{C},d}) 
$,
$
(\mathcal{R} \psi_{\varepsilon,d}) \in C(\R^d, \R)
$,
and
\begin{equation}
\label{eq:nnet-ap2-sup}
\sup_{x \in [a,b]^d} \left| u_d(T,x) -  ( \mathcal{R} \psi_{ \varepsilon, d  } )( x ) \right| \leq \varepsilon.
\end{equation}
Next note that~\eqref{eq:nnet-ap2-ass3} implies that for all $\varepsilon\in (0,r], d\in \mathbb N$ we have that
\begin{equation}
\mathcal P(\phi_{\varepsilon \mathbf c^{-1} d^{-\mathfrak{p}}\mathcal{C},d})
\leq cd^p (\varepsilon \mathbf c^{-1} d^{-\mathfrak{p}}\mathcal{C})^{-q}
= c\mathbf c^{q}d^{p+\mathfrak p q} \mathcal{C}^{-q} \varepsilon^{-q}\,.	
\end{equation}
Combining this, \eqref{eq:nnet-ap2-pf1}, and~\eqref{eq:nnet-ap2-sup} hence yields that
for all $d \in \N$, $\varepsilon \in (0,r]$  we have that 
$
\mathcal{N}(\psi_{\varepsilon, d}) \leq \mathbf{C}c \mathbf c^{q} d^{\mathbf{p} + p + \mathfrak p q}  \mathcal C^{-q} \varepsilon^{-(q+2)}
$,
$
\mathcal{P}( \psi_{ \varepsilon, d   } )
\leq \mathbf{C}^2 c\mathbf c^{q}  d^{2\mathbf{p} + p + \mathfrak p q}\mathcal C^{-q}\varepsilon^{-(q+4)}
$,
and
$\mathfrak{P}(\psi_{\varepsilon,d }) \leq 
\mathbf{C} c\mathbf c^{q} d^{\mathbf{p}+p+\mathfrak p q} \mathcal C^{-q}\varepsilon^{-(q+2)} $.
This and \eqref{eq:nnet-ap2-sup} establish item \eqref{it:nnet-ap2-2}.
This completes the proof of Corollary~\ref{cor:nnet-ap2}.
\end{proof}
\section[ANN approximations for heat equations]{Artifical neural network approximations for heat equations}
\label{sec:heat-eq}
\subsection{Viscosity solutions for heat equations}
In this subsection we establish in Lemma~\ref{lem:ex-heat2} below a well-known connection 
between viscosity solutions and classical solutions of heat equations
with at most polynomially growing initial conditions. 
Lemma~\ref{lem:ex-heat2} 
will be employed in our proof of Theorem~\ref{thm:heat-eq} below, the main result of this article. 
Lemma~\ref{lem:ex-heat2} 
is a simple consequence
of Lemma~\ref{lem:ex-heat1} and 
the Feynman-Kac formula
for viscosity solutions of Kolmogorov PDEs
(cf., for example, Hairer et al.~\cite{hairer2015}).
Lemma~\ref{lem:ex-heat1}, in turn, 
is an elementary and well-known existence result 
 for solutions of heat equations (cf., for example, Evans~\cite[Theorem 1 in Subsection 2.3.1]{EvansPDEs}).
For completeness we also provide in this subsection a detailed proof for Lemma~\ref{lem:ex-heat1}.
Our proof of Lemma~\ref{lem:ex-heat1} employs the elementary and well-known result in  Lemma~\ref{lem:gauss-moments} below.
\begin{lemma}
\label{lem:gauss-moments}
Let $p \in [0,\infty)$,
$ d \in \N$, 
and
let $\left \| \cdot \right \| \colon \R^d \to [0,\infty)$ be the standard norm on $\R^d$.
Then
it holds for all 
$ t \in (0,\infty) $,
$x \in \R^d $ that
\begin{equation}
\label{eq:gauss-moments}
\int_{\R^d}
\| y\|^p
e^{
-
\frac{\| x - y \|^2}{4t}
}
\, 
dy
<
\infty.
\end{equation}
\end{lemma}
\begin{proof}[Proof of Lemma~\ref{lem:gauss-moments}]
Throughout this proof let 
$ S $ be the set given by
\begin{equation}
\label{eq:lem-gauss-S}
S 
=
\begin{cases}
(-1,1) & \colon d =1 \\
(0,2\pi) & \colon d = 2 \\
(0,2\pi) \times (0,\pi)^{d-2}  & \colon d \in \{3,4,\ldots\}
\end{cases}
\end{equation}
and for every $n\in \{2,3,\ldots\}$ let
$ T_n \colon (0,\infty) \times (0,2\pi) \times (0,\pi)^{n-2} \to \R $
satisfy for all
$n\in \{2,3,\ldots\}$,
$ r \in (0,\infty) $,
$ \varphi \in (0,2\pi) $, 
$ \vartheta_1,\ldots,\vartheta_{n-2} \in (0,\pi) $ 
that if $n=2$ then $T_2(r,\varphi) = r$
and if $n\geq 3$ then 
\begin{equation}
\label{eq:gauss-moments-Tn}
T_n(r,\varphi,\vartheta_1,\ldots,\vartheta_{n-2})
=
r^{n-1}
\left[\smallprod_{i=1}^{n-2} [\sin(\vartheta_{i})]^{i}\right]
.
\end{equation}
Observe that 
the integral transformation theorem 
with the diffeomorphism
$ (0,\infty) \ni r \mapsto \sqrt{r} \in (0,\infty) $ 
implies that
\begin{equation}
\begin{split}
& \int_{0}^{\infty}
r^{p+d-1} e^{-r^2}
\, dr
=
\int_{0}^{\infty}
r^{\nicefrac{(p+d-1)}{2}} e^{-r}
\frac{1}{2 r^{\nicefrac{1}{2}}}
\, dr
=
\frac{1}{2}
\int_{0}^{\infty}
r^{\nicefrac{(p+d)}{2} - 1} e^{-r}
\, dr
.
\end{split}
\end{equation}
Item \eqref{it:gamma4} in Lemma~\ref{lem:gamma}
(applied with $ x \leftarrow \frac{p+d}{2} $ in the notation of
Lemma~\ref{lem:gamma})
hence verifies that 
\begin{equation}
\label{eq:gauss-moments-bound}
\int_{0}^{\infty}
r^{p+d-1} e^{-r^2}
\, dr
\leq
\frac{1}{2}
\sqrt{\frac{4\pi}{p+d}} \left[ \frac{p+d}{2e} \right]
^{
\frac{p+d}{2}
}
e^{\frac{1}{6(p+d)}} 
<
\infty
.
\end{equation}
Next note that 
the integral transformation theorem 
with the diffeomorphism
$
\R^d \ni y \mapsto 2\sqrt{t} y \in \R^d
$
for 
$ t \in (0,\infty) $,
the triangle inequality,
and
the fact that for all 
$ a,b \in [0,\infty) $ 
we have that
$ (a+b)^p \leq \max\{1,2^{p-1}\} (a^p + b^p) $
ensure
that
for all 
$ t \in (0,\infty) $,
$ x \in \R^d $
we have that
\begin{equation}
\label{eq:gauss-moments-sum}
\begin{split}
\int_{\R^d}
\| y\|^p
e^{
-
\frac{\|x-y\|^2}{4t}
}
\, 
dy
& =
\int_{\R^d}
\| x-y\|^p
e^{
-
\frac{\| y \|^2 }{4t}
}
\, 
dy
\\ & =
\int_{\R^d}
\| x-2 \sqrt{t} y \|^p
e^{
-
\| y \|^2
}
(2 \sqrt{t})^d
\, 
dy
\\ &  \leq
\max\{1,2^{p-1}\}
(2\sqrt{t})^d \| x \|^p  \int_{\R^d}  e^{- \|y\|^2 } \, dy 
\\ & +
\max\{1,2^{p-1}\} (2\sqrt{t})^{p+d}\int_{\R^d} \| y \|^p e^{- \|y\|^2 } \, dy 
.
\end{split}
\end{equation}
To establish \eqref{eq:gauss-moments} we distinguish between 
the case $ d = 1 $ 
and
the case $ d \in \N \cap [2,\infty) $.
First, we consider the case $ d = 1 $.
Note that 
\begin{equation}
\int_{\R^d} \| y \|^p e^{- \|y\|^2 } \, dy 
=
2 \int_0^{\infty} y^p e^{-y^2} \, dy
.
\end{equation}
Combining this with 
\eqref{eq:gauss-moments-bound}
and
\eqref{eq:gauss-moments-sum}
establishes 
\eqref{eq:gauss-moments}
in the case $ d = 1 $.
Next we consider the case 
$ d \in \{2,3,\ldots \}$.
Note that
\eqref{eq:lem-gauss-S},
\eqref{eq:gauss-moments-Tn},
item \eqref{it:polar3} in Lemma~\ref{lem:polar},
and
Fubini's theorem
ensure that
\begin{equation}
\begin{split}
\int_{\R^d} \| y \|^p e^{- \|y\|^2 } \, dy 
& =
\int_0^\infty \int_S
r^{p+d-1} e^{-r^2} T_d(1,\phi) \, d\phi\,dr
\\ & =
\int_S \int_0^\infty  
r^{p+d-1} e^{-r^2} \, dr \, T_d(1,\phi) \, d\phi
\\ & \leq
2 \pi^{d-1}
\int_0^\infty  
r^{p+d-1} e^{-r^2} \, dr
.
\end{split}
\end{equation}
Combining this with 
\eqref{eq:gauss-moments-bound}
and
\eqref{eq:gauss-moments-sum}
establishes 
\eqref{eq:gauss-moments}
in the case
$ d \in \{2,3,\ldots \} $.
This completes the proof of Lemma~\ref{lem:gauss-moments}.
\end{proof}
\begin{lemma}
\label{lem:ex-heat1}
Let 
$ d \in \N$,
$ \varphi \in C(\R^d,\R)$,
let $\left \| \cdot \right \| \colon \R^d \to [0,\infty)$ be the standard norm on $\R^d$,
assume that
$
\inf_{ \gamma \in (0,\infty) } 
\sup_{ x \in \R^d }
\big(
\frac{|\varphi(x)|}{1+\| x \| ^{\gamma}}
\big)
< \infty
$,
and
let 
$ \Phi \colon (0,\infty) \times \R^d \to \R $ 
satisfy for all
$ t \in (0,\infty) $, 
$ x=(x_1,\ldots,x_d) \in \R^d $
that
\begin{equation}
\Phi(t,x) 
=
\int_{\R^d}
\frac{1}{(4\pi t)^{\frac{d}{2}} }
e^{
-
\frac{(x_1-y_1)^2+\ldots+(x_d-y_d)^2}{4t}
}
\varphi(y)
\, 
dy.
\end{equation}
Then it holds for all 
$ t \in (0,\infty) $,
$ x \in \R^d $
that
$ \Phi \in C^{1,2}((0,\infty)\times\R^d,\R) $
and
\begin{equation}
\label{eq:ex-heat1-heq}
(\tfrac{\partial}{\partial t }\Phi)(t,x) 
=
(\Delta_x \Phi)(t,x)
.
\end{equation}
\end{lemma}
\begin{proof}[Proof of Lemma~\ref{lem:ex-heat1}]
Throughout this proof 
let 
$ \rho \colon (0,\infty) \times \R^d \to \R $
satisfy for all
$ t \in (0,\infty) $,
$ x=(x_1,\ldots,x_d) \in \R^d $ 
that
\begin{equation}
\label{eq:ex-heat1-gamma}
\rho(t,x) 
=
\frac{1}{(4\pi t)^{\frac{d}{2}} }
e^{
-
\frac{x_1^2+\ldots+x_d^2}{4t}
}.
\end{equation}
Observe that for all 
$ t \in (0,\infty) $,
$ x=(x_1,\ldots,x_d) \in \R^d $ 
we have that
\begin{equation}
\label{eq:ex-heat1-dt}
(\tfrac{\partial}{\partial t } \rho) (t,x)
=
\left[
\frac{x_1^2+\ldots+x_d^2}{4t^2}
-
\frac{d}{2t}
\right]
\rho(t,x)
.
\end{equation}
Next note that for all 
$ i \in \{1,\ldots,d\} $,
$ t \in (0,\infty) $,
$ x=(x_1,\ldots,x_d) \in \R^d $ 
we have that
\begin{equation}
\label{eq:ex-heat1-dxi}
(\tfrac{\partial}{\partial x_i } \rho) (t,x)
=
-
\frac{x_i}{2t}
\rho(t,x)
.
\end{equation}
This implies that for all 
$ i,j \in \{1,\ldots,d\} $,
$ t \in (0,\infty) $,
$ x=(x_1,\ldots,x_d) \in \R^d $ 
we have that
\begin{equation}
\label{eq:ex-heat1-dxij}
(\tfrac{\partial^2}{\partial x_i \partial x_j } \rho) (t,x)
=
\begin{cases}
\Big[
\frac{x_i^2}{4t^2}
-
\frac{1}{2t}
\Big]
\rho(t,x)
&
\colon i = j
\\
\frac{x_i x_j}{4t^2}
\rho(t,x)
&
\colon 
i \neq j
.
\end{cases}
\end{equation}
This reveals that for all
$ t \in (0,\infty) $,
$ x=(x_1,\ldots,x_d) \in \R^d $ 
we have that
\begin{equation}
(\Delta_x \rho)(t,x)
=
\sum_{i=1}^d
(\tfrac{\partial^2}{\partial x_i^2 } \rho) (t,x)
=
\left[
\frac{x_1^2+\ldots+x_d^2}{4t^2}
-
\frac{d}{2t}
\right]
\rho(t,x)
.
\end{equation}
Combining this with \eqref{eq:ex-heat1-dt}
yields that for all 
$ t \in (0,\infty) $,
$ x \in \R^d $ 
we have that
\begin{equation}
\label{eq:ex-heat1-heq-phi}
(\tfrac{\partial}{\partial t } \rho) (t,x)
-
(\Delta_x \rho)(t,x)
=
0
.
\end{equation}
Next note that the hypothesis that
$
\inf_{ \gamma \in (0,\infty) } 
\sup_{ x \in \R^d }
\big(
\frac{|\varphi(x)|}{1+\| x \| ^{\gamma}}
\big)
< \infty
$
ensures that there exist $ \gamma \in (0,\infty) $, $ C  \in \R $ 
which satisfy that for all 
$ x \in \R^d $ we have that
\begin{equation}
\label{eq:ex-heat1-varphi}
|\varphi(x)| \leq C(1+\|x\|^\gamma)
.
\end{equation}
This 
and
Lemma~\ref{lem:gauss-moments}
verify that 
for all 
$ t \in (0,\infty) $, 
$ x \in \R^d $
we have that
\begin{equation}
\label{eq:ex-heat1-Phi}
|\Phi(t,x)|
\leq
\int_{\R^d}
|
\rho(t,x-y)\varphi(y)
|
\, dy
\leq
C
\int_{\R^d}
\rho(t,x-y)
(1+ \|y\|^\gamma)
\,
dy
< 
\infty
.
\end{equation}
Next note 
that
\eqref{eq:ex-heat1-dt},
\eqref{eq:ex-heat1-varphi},
the triangle inequality,
and
Lemma~\ref{lem:gauss-moments}
demonstrate
that for all
$ t \in (0,\infty) $,
$ x \in \R^d $
we have that
\begin{equation}
\begin{split}
& \int_{\R^d}
\big|
(\tfrac{\partial}{\partial t } \rho) (t,x-y) 
\varphi(y)
\big|
\, dy
\\ & 
\leq
C
\int_{\R^d}
\left[
\frac{\| x-y \|^2}{4t^2}
+
\frac{d}{2t}
\right]
\rho(t,x-y)
(1+\|y\|^\gamma)
\, 
dy
<
\infty
.
\end{split}
\end{equation}
Combining this,
\eqref{eq:ex-heat1-Phi},
and
\eqref{eq:ex-heat1-dt}
with
Amann \& Escher~\cite[Ch. X, Theorem 3.18]{MR2500068}
verifies that  
for all 
$ t \in (0,\infty) $,
$ x \in \R^d $
we have that
$ \Phi \in C^{1,0}((0,\infty) \times \R^d,\R)$ 
and
\begin{equation}
\label{eq:ex-heat1-dt2}
(\tfrac{\partial}{\partial t }\Phi )(t,x)
=
\int_{\R^d}
(\tfrac{\partial}{\partial t } \rho) (t,x-y) 
\varphi(y)
\, dy
.
\end{equation}
Next observe that
\eqref{eq:ex-heat1-dxi},
\eqref{eq:ex-heat1-varphi},
and
Lemma~\ref{lem:gauss-moments}
ensure that
for all 
$ i \in \{1,\ldots,d\} $,
$ t \in (0,\infty) $,
$ x \in \R^d $ 
we have that
\begin{equation}
\label{eq:ex-heat1-idxi}
\begin{split}
 \int_{\R^d} \big| (\tfrac{\partial}{\partial x_i } \rho) (t,x-y) \varphi(y) \big| \, dy
& \leq
\frac{C}{2t}
\int_{\R^d}
\|x-y\|
\rho(t,x-y)
(1+\|y\|^\gamma)
\, dy
< 
\infty.
\end{split}
\end{equation}
Combining this,
\eqref{eq:ex-heat1-Phi},
and
\eqref{eq:ex-heat1-dxi}
with
\eqref{eq:ex-heat1-dt2}
and
Amann \& Escher~\cite[Ch. X, Theorem 3.18]{MR2500068}
verifies that 
for all 
$ i \in \{1,\ldots,d\}$,
$ t \in (0,\infty)$,
$ x=(x_1,\ldots,x_d) \in \R^d $
we have that
$ \Phi \in C^{1,1} ((0,\infty)\times\R^d,\R) $ 
and
\begin{equation}
\label{eq:ex-heat1-dxi2}
(\tfrac{\partial}{\partial  x_i }\Phi )(t,x)
=
\int_{\R^d}
(\tfrac{\partial}{\partial x_i} \rho) (t,x-y) 
\varphi(y)
\, dy
.
\end{equation}
Next note that 
\eqref{eq:ex-heat1-dxij},
\eqref{eq:ex-heat1-varphi},
the fact that 
for all $ a,b \in \R $ we have that 
$
ab \leq a^2 + b^2
$,
and
Lemma~\ref{lem:gauss-moments}
ensure that
for all 
$ i,j \in \{1,\ldots,d\} $,
$ t \in (0,\infty) $,
$ x=(x_1,\ldots,x_d) \in \R^d $ 
we have that
\begin{equation}
\begin{split}
& \int_{\R^d}
\big| (\tfrac{\partial^2}{\partial  x_i \partial x_j } \rho) (t,x-y) \varphi(y) \big| 
\, dy
\\ & \leq
C
\int_{\R^d}
\left[
\frac{\|x-y\|^2}{4t^2} 
+
\frac{1}{2t}
\right]
\rho(t,x-y)
(1 + \| y \|^\gamma)
\,
dy
< 
\infty
.
\end{split}
\end{equation}
Combining this,
\eqref{eq:ex-heat1-idxi},
and
\eqref{eq:ex-heat1-dxij}
with
\eqref{eq:ex-heat1-dxi2}
and
Amann \& Escher~\cite[Ch. X, Theorem 3.18]{MR2500068}
verifies that  
for all 
$ i,j \in \{1,\ldots,d\}$,
$ t \in (0,\infty) $,
$ x=(x_1,\ldots,x_d) \in \R^d $
we have that
$ \Phi \in C^{1,2}((0,\infty)\times \R^d,\R)$ 
and
\begin{equation}
(\tfrac{\partial^2}{\partial x_i \partial x_j }\Phi )(t,x)
=
\int_{\R^d}
(\tfrac{\partial^2}{\partial x_i \partial x_j} \rho) (t,x-y) 
\varphi(y)
\, dy
.
\end{equation}
Hence, we obtain that
for all 
$ t \in (0,\infty) $,
$ x \in \R^d $
we have that
\begin{equation}
\begin{split}
(\Delta_x \Phi) (t,x)
& =
\sum_{i=1}^d
(\tfrac{\partial^2}{\partial x_i^2 }\Phi )(t,x)
\\ & =
\int_{\R^d}
\left[
\sum_{i=1}^d
(\tfrac{\partial^2}{\partial x_i^2} \rho) (t,x-y) 
\varphi(y)
\right]
dy
\\ & =
\int_{\R^d}
(\Delta_x \rho)(t,x-y)
\varphi(y)
\, dy
.
\end{split}
\end{equation}
Combining this and \eqref{eq:ex-heat1-dt2}
with
\eqref{eq:ex-heat1-heq-phi}
establishes \eqref{eq:ex-heat1-heq}.
This completes the proof of Lemma~\ref{lem:ex-heat1}.
\end{proof}
\begin{lemma}
\label{lem:ex-heat2}
Let 
$ d \in \N$,
$ T \in (0,\infty)$,
$ \varphi \in C(\R^d,\R)$, $ u \in  C([0,T] \times \R^d, \R) $,
let $\left \| \cdot \right \| \colon \R^d \to [0,\infty)$ be a norm, assume for all $x \in \R^d$ that $u(0,x) = \varphi(x)$, assume
that 
$
\inf_{\gamma \in (0,\infty)} \sup_{(t,x) \in [0,T] \times \R^d} \big( \frac{|u(t,x)|}{1 +\|x\|^{\gamma}} \big) 
< \infty
$,
and assume that $u|_{(0,T) \times \R^d}$ is a viscosity solution of
\begin{align}
(\tfrac{\partial }{\partial t}u)(t,x) 
&= 
(\Delta_x u)(t,x)
\end{align}
for $(t,x) \in (0,T) \times \R^d$.
Then it holds for all $ t \in (0,T]$, $ x \in \R^d $ that
$ u|_{(0,T] \times \R^d} \in C^{1,2}((0,T] \times \R^d,\R) $
and
\begin{align}
(\tfrac{\partial }{\partial t}u)(t,x) 
&= 
(\Delta_x u)(t,x)
.
\end{align}
\end{lemma}
\begin{proof}[Proof of Lemma~\ref{lem:ex-heat2}]
Throughout this proof  
let $\left \| \cdot \right \|_2 \colon \R^d \to [0,\infty)$ be the standard norm on $\R^d$,
let $(\Omega, \mathcal{F},\P)$ be a probability space with a normal filtration 
$(\mathbbm{F}_t)_{t \in [0,T]}$,
and
let $W\colon[0,T]\times \Omega \to \R^d$
be a standard $(\mathbbm{F}_t)_{t \in [0,T]}$-Brownian motion.
Observe that there exist 
$C \in \R $, 
$ c \in (0,\infty) $
such that for all 
$ x \in \R ^d$ we have that
\begin{equation}
\label{eq:ex-heat2-eqn}
c\|x\|_2 \leq \|x\| \leq C\|x\|_2.
\end{equation}
This and the fact that 
$
\inf_{ \gamma \in (0,\infty) } 
\sup_{ x \in \R^d }
\big(
\frac{|\varphi(x)|}{1+\| x \| ^{\gamma}}
\big)
< \infty
$
verify that  
\begin{equation}
\inf_{ \gamma \in (0,\infty) } 
\sup_{ x \in \R^d }
\Big(
\frac{|\varphi(x)|}{1+\| x \|_2 ^{\gamma}}
\Big)
< \infty
.
\end{equation}
Hence, we obtain that $ \varphi \colon \R^d  \to \R $ 
is an at most polynomially growing function.
The Feynman-Kac formula (cf., for example, 
Grohs et al.~\cite[Proposition 2.22(iii)]{HornungJentzen2018}
and 
Hairer et al.~\cite[Corollary~4.17]{hairer2015})
hence ensures 
that
for all 
$ t \in [0,T] $,
$ x \in \R^d $ 
we have that
\begin{equation}
\label{eq:ex-heat2-fk}
u(t,x) = \E[\varphi(x + \sqrt{2}W_t)]
.
\end{equation}
Next note that
the fact that for all $ t \in (0,T] $ 
we have that
$ W_t $ is a
$ \mathcal{N}_{0,tI_{\R^d}}$ distributed random variable
implies that
for all 
$ t \in (0,T] $,
$ x \in \R^d $ 
we have that 
$ x + \sqrt{2}W_t$ 
is a
$ \mathcal{N}_{x,2tI_{\R^d}} $ 
distributed random variable.
Combining this with 
\eqref{eq:ex-heat2-fk}
demonstrates 
that for all 
$ t \in (0,T] $,
$ x \in \R^d $ 
we have that
\begin{equation}
u(t,x)
=
\int_{\R^d}
\frac{1}{(4\pi t)^{\frac{d}{2}} }
e^{
-
\frac{(x_1-y_1)^2+\ldots+(x_d-y_d)^2}{4t}
}
\varphi(y)
\, 
dy.
\end{equation}
Lemma~\ref{lem:ex-heat1} hence proves that 
for all $ t \in (0,T]$, $ x \in \R^d $ 
we have that 
$ u|_{(0,T] \times \R^d} \in C^{1,2}((0,T]\times \R^d,\R) $
and
\begin{align}
(\tfrac{\partial }{\partial t}u)(t,x) 
&= 
(\Delta_x u)(t,x)
.
\end{align}
This completes the proof of Lemma~\ref{lem:ex-heat2}.
\end{proof}
\subsection{Qualitative error estimates for heat equations}
\label{ssec:ql-heat}
It is the subject of this subsection to state and prove Theorem~\ref{thm:heat-eq} below, which is the main result of this work.
Theorem~\ref{thm:heat-eq} establishes that 
ANNs do not suffer from the curse of dimensionality 
in the uniform numerical approximation of heat equations. 
Corollary~\ref{cor:heat-eq2} below specializes Theorem~\ref{thm:heat-eq} to the case in which the constants  $c \in \R$, $ p,q,v,\mathbf{v},w,\mathbf{w}, z,\mathbf{z}  \in [0,\infty)$, which are used to formulate the hypotheses in \eqref{eq:heat-eq-ass1}--\eqref{eq:heat-eq-ass2} below, all coincide.
\begin{theorem}
\label{thm:heat-eq}
Assume Setting~\ref{setting:ANNs}, let $c, a \in \R$, $b \in (a,\infty)$,   $r, T \in (0,\infty)$, 
$ p,q,v,\mathbf{v},w,\mathbf{w}, z,\mathbf{z}  \in [0,\infty)$, 
$
  \mathbf{p} 
= 
  2 +  
  \max\{  2z + 3\mathbf{z},
  2w + 3\mathbf{w} +1  \} 
$,
for every $d \in \N$ let $\left \| \cdot \right \|_{\R^d} \colon \R^d \to [0,\infty)$ be the standard norm on $\R^d$,
let $\varphi_d \in C(\R^d,\R)$, $d \in \N$, 
let $(\phi_{\varepsilon,d})_{(\varepsilon,d) \in (0,r]\times \N}  \subseteq \mathbf{N}$, 
and assume for all $\varepsilon \in (0,r]$, $d \in \N$, 
$ x \in \R^d$ that $(\mathcal R \phi_{\varepsilon,d}) \in C(\R^d,\mathbb R)$,
\begin{equation}
\label{eq:heat-eq-ass1}
\left| (\mathcal{R}\phi_{\varepsilon,d})(x)\right| \leq c  d^z(1+\|x\|_{\R^d}^{\mathbf{z}}), 
\qquad
\| (\nabla (\mathcal{R}\phi_{\varepsilon,d}))(x)\|_{\R^d} \leq  c  d^w(1+\|x\|_{\R^d}^{\mathbf{w}}),
\end{equation}
\begin{equation}
\label{eq:heat-eq-ass2}
\left| \varphi_d(x)-(\mathcal{R}\phi_{\varepsilon,d})(x)\right| \leq \varepsilon c  d^v(1+\|x\|_{\R^d}^{\mathbf{v}}),
\quad
\text{and}
\quad
\mathcal{P}(\phi_{\varepsilon,d}) \leq c  d^p\varepsilon^{-q}.
\end{equation}
Then
\begin{enumerate}[(i)]
\item \label{it:heat-eq-1}
there exist unique 
$u_d \in  C([0,T] \times \R^d, \R) $,
$d \in \N$, 
which satisfy for all
$d \in \N$, 
$ t \in (0,T] $,
$x \in \R^d$ 
that 
$ u_d|_{(0,T] \times \R^d } \in C^{1,2}((0,T] \times \R^d, \R ) $,
$u_d(0,x) = \varphi_d(x)$,
$
\inf_{\gamma \in (0,\infty)}  
\sup_{(s,y) \in [0,T] \times \R^d}\! 
\big( \frac{|u_d(s,y)|}{1 + \|y\|_{ \R^d } ^{\gamma}} \big)
< \infty
$,
and 
\begin{align}
\label{eq heat eq in t54}
(\tfrac{\partial }{\partial t}u_d)(t,x) 
&= 
(\Delta_x u_d)(t,x)
\end{align}
and
\item \label{it:heat-eq-2}
there exist $(\psi_{\varepsilon,d})_{(\varepsilon,d) \in (0,r]\times \N} \subseteq \mathbf{N}$, 
$C \in \R$ such that 
for all
$\varepsilon \in (0,r]$, 
$d \in \N$
we have that
$
(\mathcal{R}\psi_{\varepsilon,d}) \in C(\R^d,\R)
$,
$
\mathcal{N}(\psi_{\varepsilon,d}) \leq C d^{p + \mathbf{p} + q(v+\frac12 \mathbf v)}\varepsilon^{-(q+2)}
$,
$
\mathcal{P}( \psi_{ \varepsilon,d  } )
\leq C d^{p+ 2 \mathbf{p}+ q(v+\frac12 \mathbf v)} \varepsilon^{-(q+4)}
$,
$
\mathfrak{P}(\psi_{\varepsilon,d}) 
\leq C d^{p+ \mathbf{p}+ q(v+\frac12 \mathbf v)} \varepsilon^{-(q+2)}
$,
and
\begin{equation}
\sup_{x \in [a,b]^d} \left| u_d(T,x) -  ( \mathcal{R} \psi_{ \varepsilon, d  } )( x ) \right| \leq \varepsilon.
\end{equation}
\end{enumerate}
\end{theorem}
\begin{proof}[Proof of Theorem~\ref{thm:heat-eq}]
First, observe that for all $ d \in \N $ we have that
\begin{equation}
 \sqrt{\operatorname{Trace}(I_{\R^d})} 
 =
 \left[\sum_{i=1}^d 1 \right]^{\nicefrac{1}{2} }
 =
 d^{\nicefrac{1}{2}}
 \leq
 \max\{1,c\}d^{\nicefrac{1}{2}}
 .
\end{equation}
Corollary~\ref{cor:nnet-ap2} 
(applied with $ \alpha \leftarrow 0$,
$ \beta \leftarrow \frac{1}{2} $,
$ c \leftarrow \max\{1,c\} $,
$ \mu_d \leftarrow 0 $,
$ A_d \leftarrow I_{\R^d} $
for $ d \in \N $ in the notation of Corollary~\ref{cor:nnet-ap2})
hence implies that 
there exist unique $u_d \in  C([0,T] \times \R^d, \R)$, $d \in \N$,  which satisfy for all $d \in \N$, $x \in \R^d$ that $u_d(0,x) = \varphi_d(x)$,
which satisfy for all $d \in \N$ that 
$
\inf_{\gamma \in (0,\infty)}  
\sup_{(t,x) \in [0,T] \times \R^d}\! 
\big( \frac{|u_d(t,x)|}{1 + \|x\|_{ \R^d } ^{\gamma}} \big)
< \infty
$,
and which satisfy that for all $d \in \N$ we have that $u_d|_{(0,T)\times \mathbb R^d}$ is a viscosity solution of
\begin{align}
\label{eq equat heat eq visc}
(\tfrac{\partial }{\partial t}u_d)(t,x) 
&= 
(\Delta_x u_d)(t,x)
\end{align}
for $(t,x) \in (0,T) \times \R^d$
and
there exist $(\psi_{\varepsilon,d})_{(\varepsilon,d) \in (0,r]\times \N} \subseteq \mathbf{N}$, 
$C \in \R$ such that 
for all $\varepsilon \in (0,r]$, $d \in \N$ we have that
$
\mathcal{N}(\psi_{\varepsilon,d}) \leq C d^{p + \mathbf{p}+ q(v+\frac12 \mathbf v)}\varepsilon^{-(q+2)}
$,
$
\mathcal{P}( \psi_{ \varepsilon,d  } )
\leq C d^{p+ 2 \mathbf{p}+ q(v+\frac12 \mathbf v)} \varepsilon^{-(q+4)}
$,
$
\mathfrak{P}(\psi_{\varepsilon,d}) 
\leq C d^{p+ \mathbf{p}+ q(v+\frac12 \mathbf v)} \varepsilon^{-(q+2)}
$,
$
(\mathcal{R}\psi_{\varepsilon,d}) \in C(\R^d,\R)
$,
and
\begin{equation}
\sup_{x \in [a,b]^d} \left| u_d(T,x) -  ( \mathcal{R} \psi_{ \varepsilon, d  } )( x ) \right| \leq \varepsilon.
\end{equation}
This proves item (ii).
Next note that \eqref{eq:heat-eq-ass1}
and
\eqref{eq:heat-eq-ass2}
ensure that 
for all $ d \in \N $
we have that 
$ \varphi_d \colon \R^d \to \R $ 
is an at most polynomially growing function.
This reveals that for all
$ d \in \N $
we have that
\begin{equation}
\inf_{ \gamma \in (0,\infty) } 
\sup_{ x \in \R^d }
\Big(
\frac{|\varphi_d(x)|}{1+\| x \|_{\R^d} ^{\gamma}}
\Big)
< \infty
.
\end{equation}
Lemma~\ref{lem:ex-heat2} 
(applied with $ T \leftarrow T$, $ \varphi \leftarrow \varphi_d$, $ u \leftarrow u_d $ for $ d \in \N $ in the notation of Lemma~\ref{lem:ex-heat2})
hence 
shows 
that 
for all 
$ d \in \N $,
$ t \in (0,T]$,
$ x \in  \R^d$
we have that 
$ u_d|_{(0,T] \times \R^d} \in C^{1,2}((0,T] \times \R^d,\R) $ and
\begin{align}
(\tfrac{\partial }{\partial t}u_d)(t,x) 
&= 
(\Delta_x u_d)(t,x)
.
\end{align}
This, \eqref{eq equat heat eq visc}, and Hairer et al.~\cite[Remark 4.1]{hairer2015}) prove item (i).
This completes the proof of Theorem~\ref{thm:heat-eq}.
\end{proof}
\begin{cor}
\label{cor:heat-eq2}
Assume Setting~\ref{setting:ANNs}, let $a \in \R$, $b \in (a,\infty)$,   
$c,T \in (0,\infty)$, 
for every $d \in \N$ let $\left \| \cdot \right \|_{\R^d} \colon \R^d \to [0,\infty)$ be the standard norm on $\R^d$,
let $\varphi_d \in C(\R^d,\R)$, $d \in \N$, 
let $(\phi_{\varepsilon,d})_{(\varepsilon,d) \in (0,1]\times \N}  \subseteq \mathbf{N}$, 
and 
assume for all 
$\varepsilon \in (0,1]$, 
$d \in \N$, 
$ x \in \R^d$
that $(\mathcal R \phi_{\varepsilon,d}) \in C(\mathbb R^d,\mathbb R)$,
\begin{equation}
\label{eq:heat-eq2-ass0}
\mathcal{P}(\phi_{\varepsilon,d})  \leq c  d^c\varepsilon^{-c}\,,
\quad
| \varphi_d(x) - (\mathcal{R}\phi_{\varepsilon,d})(x)| 
\leq 
\varepsilon c  d^c(1+\|x\|_{\R^d}^{c})\,,
\end{equation}
and
\begin{equation}
\label{eq:heat-eq2-ass1}
\left| (\mathcal{R}\phi_{\varepsilon,d})(x)\right| 
+
\| (\nabla (\mathcal{R}\phi_{\varepsilon,d}))(x)\|_{\R^d}
\leq
c  d^c(1+\|x\|_{\R^d}^{c}).
\end{equation}
Then
\begin{enumerate}[(i)]
\item \label{it:heat-eq2-1}
there exist unique 
$u_d \in  C([0,T] \times \R^d, \R)$,
$d \in \N$, 
which satisfy for all
$d \in \N$, 
$ t \in (0,T] $,
$x \in \R^d$ 
that 
$ u_d|_{(0,T] \times \R^d } \in C^{1,2}((0,T] \times \R^d, \R ) $,
$u_d(0,x) = \varphi_d(x)$,
$
\inf_{\gamma \in (0,\infty)}  
\sup_{(s,y) \in [0,T] \times \R^d}\! 
\big( \frac{|u_d(s,y)|}{1 + \|y\|_{ \R^d } ^{\gamma}} \big)
< \infty
$,
and 
\begin{align}
(\tfrac{\partial }{\partial t}u_d)(t,x) 
&= 
(\Delta_x u_d)(t,x)
\end{align}
and
\item \label{it:heat-eq2-2}
there exist $(\psi_{\varepsilon,d})_{(\varepsilon,d) \in (0,1]\times \N} \subseteq \mathbf{N}$, 
$\kappa \in \R$ such that 
for all $\varepsilon \in (0,1]$, $d \in \N$ we have that
$
\mathcal{P}( \psi_{ \varepsilon,d  } )
\leq \kappa d^{\kappa} \varepsilon^{-\kappa}
$,
$
(\mathcal{R}\psi_{\varepsilon,d}) \in C(\R^d,\R)
$,
and
\begin{equation}
\sup_{x \in [a,b]^d} \left| u_d(T,x) -  ( \mathcal{R} \psi_{ \varepsilon, d  } )( x ) \right| \leq \varepsilon.
\end{equation}
\end{enumerate}
\end{cor}
\begin{proof}[Proof of Corollary~\ref{cor:heat-eq2}]
First, observe that~\eqref{eq:heat-eq2-ass0}, \eqref{eq:heat-eq2-ass1},
and item \eqref{it:heat-eq-1} in Theorem~\ref{thm:heat-eq} 
(applied with
$ r \leftarrow 1$,
$c \leftarrow c$, $p \leftarrow c$, $q \leftarrow c$, $v \leftarrow c$, $\mathbf{v} \leftarrow c $, $w \leftarrow c$, $\mathbf{w} \leftarrow c$, $z \leftarrow c$, $\mathbf{z} \leftarrow c$
in the notation of Theorem~\ref{thm:heat-eq})
establish item \eqref{it:heat-eq2-1}.
Next note that~\eqref{eq:heat-eq2-ass0}, \eqref{eq:heat-eq2-ass1},
and
item \eqref{it:heat-eq-2} in Theorem~\ref{thm:heat-eq}
(applied with
$ r \leftarrow 1$,
$c \leftarrow c$, $p \leftarrow c$, $q \leftarrow c$, $v \leftarrow c$, $\mathbf{v} \leftarrow c $, $w \leftarrow c$, $\mathbf{w} \leftarrow c$, $z \leftarrow c$, $\mathbf{z} \leftarrow c$
in the notation of Theorem~\ref{thm:heat-eq})
ensure that 
there exist $(\psi_{\varepsilon,d})_{(\varepsilon,d) \in (0,1]\times \N} \subseteq \mathbf{N}$, 
$C \in \R$ which satisfy that 
for all $\varepsilon \in (0,1]$, $d \in \N$ we have that
$
\mathcal{P}( \psi_{ \varepsilon,d  } )
\leq C  d^{\frac32 c^2 + 11c+6} \varepsilon^{-(c+4)}
$,
$
(\mathcal{R}\psi_{\varepsilon,d}) \in C(\R^d,\R)
$,
and
\begin{equation}
\label{eq:heat-eq2-sup}
\sup_{x \in [a,b]^d} \left| u_d(T,x) -  ( \mathcal{R} \psi_{ \varepsilon, d  } )( x ) \right| \leq \varepsilon.
\end{equation}
This reveals that for all $\varepsilon \in (0,1]$, $d \in \N$ we have that
\begin{equation}
\mathcal{P}( \psi_{ \varepsilon,d  } )
\leq \max\{C,\tfrac32 c^2 + 11c + 6\}  d^{\max\{C,\frac32 c^2 + 11c + 6\}} \varepsilon^{-\max\{C,\frac32 c^2 + 11c + 6\}}.
\end{equation}
Combining this with \eqref{eq:heat-eq2-sup} 
establishes item \eqref{it:heat-eq2-2}.
This completes the proof of Corollary~\ref{cor:heat-eq2}.
\end{proof}

\subsection{ANN approximations for geometric Brownian motions}
\label{ssec:heat-ex}
In this subsection we specialize Theorem~\ref{thm:heat-eq} above
in Corollary~\ref{cor:heat-eq5} below to an example in which the activation function is the softplus function ($\mathbb{R} \ni x \mapsto \ln(1+e^{x}) \in (0,\infty)$).

\begin{cor}
\label{cor:heat-eq5}
Let $c,a \in \R$, $b \in (a,\infty)$,   
$p\in [0,\infty)$,
$T \in (0,\infty)$, 
for every $d \in \N$ let $\left \| \cdot \right \|_{\R^d} \colon \R^d \to [0,\infty)$ be the standard norm on $\R^d$,
let $\mathbf{N}$ be the set given by 
\begin{equation}
\mathbf{N} = \cup_{L \in \N \cap [2,\infty)}\cup_{\substack{l_0,\ldots,l_L \in  \N 
}} \big( 
\times_{k=1}^L (\R^{l_k\times l_{k-1}}\times\R^{l_k})
\big)
,
\end{equation}
let
$\Af_d \colon \R^d \to \R^d$, $ d \in \N$, satisfy for all  $d \in \N$, $x=(x_1,\ldots,x_d) \in \R^d$ that 
$ \Af_d(x) = (\ln(1 + e^{x_1}),\ldots,\ln(1+e^{x_d})) $, 
let $\mathcal{P}\colon \mathbf{N} \to \N$ 
and
$\mathcal{R} \colon \mathbf{N} \to \cup_{m,n \in \N} C(\R^m,\R^n)$ satisfy for all
$L \in \N \cap [2,\infty)$, 
$l_0,\ldots,l_L \in \N $,  
$\Phi = 
(
(\WE_1,\BI_1),
\ldots,
(\WE_L,  \BI_L)
)
\in (\times_{k=1}^{L} (\R^{l_k \times l_{k-1}} \times \R^{l_k}))$, 
$x_0 \in \R^{l_0}$, $\ldots, x_{L-1}\in \R^{l_{L-1}}$ with 
$\forall \, k \in \N \cap (0,L) \colon x_k = \Af_{l_k}(\WE_k x_{k-1} + \BI_k)$ that
$\mathcal{P}(\Phi) =  \smallsum_{k=1}^{L} l_k(l_{k-1} + 1)$,
$(\mathcal{R} \Phi) \in C(\R^{l_0},\R^{l_L})$, 
and
\begin{equation}
\label{eq:heat-eq5-R}
(\mathcal{R}\Phi)(x_0) =  \WE_L x_{L-1} + \BI_L,
\end{equation}
and
let $ (K_d)_{ d \in \N } \subseteq \R $
satisfy
for all $ d \in \N $ 
that
$
|K_d|
\leq
cd^p
$.
Then
\begin{enumerate}[(i)]
\item \label{it:heat-eq5-1} 
there exist unique
$u_d \in  C([0,T] \times \R^d, \R)$, 
$d \in \N$, 
which satisfy for all 
$d \in \N$, 
$ t \in (0,T] $,
$x = (x_1,\ldots,x_d) \in \R^d$ 
that 
$ u_d|_{(0,T] \times \R^d } \in C^{1,2}((0,T] \times \R^d, \R ) $, 
$u_d(0,x) = \ln(1+e^{x_1+\ldots+ x_d - K_d})+K_d$,
$
\inf_{\gamma \in (0,\infty)}  
\sup_{(s,y) \in [0,T] \times \R^d}\! 
\big( \frac{|u_d(s,y)|}{1 + \|y\|_{ \R^d } ^{\gamma}} \big)
< \infty
$,
and
\begin{align}
(\tfrac{\partial }{\partial t}u_d)(t,x) 
&= 
(\Delta_x u_d)(t,x)
\end{align}
and
\item \label{it:heat-eq5-2}
there exist $(\psi_{\varepsilon,d})_{(\varepsilon,d) \in (0,1]\times \N} \subseteq \mathbf{N}$, 
$\kappa \in \R$ such that 
for all $\varepsilon \in (0,1]$, $d \in \N$ we have that
$
\mathcal{P}( \psi_{ \varepsilon,d  } )
\leq \kappa d^{11 + 4\max\{p,\nicefrac{1}{2}\}} \varepsilon^{-4}
$,
$
(\mathcal{R}\psi_{\varepsilon,d}) \in C(\R^d,\R)
$,
and
\begin{equation}
\sup_{x \in [a,b]^d} \left| u_d(T,x) -  ( \mathcal{R} \psi_{ \varepsilon, d  } )( x ) \right| \leq \varepsilon.
\end{equation}
\end{enumerate}
\end{cor}
\begin{proof}[Proof of Corollary~\ref{cor:heat-eq5}]
Throughout this proof 
let 
$ \varphi_d \colon \R^d \to \R $,
$ d \in \N $,
satisfy for all
$ d \in \N $,
$ x=(x_1,\ldots,x_d) \in \R^d $
that
\begin{equation}
\varphi_d(x)
=
\ln(1 + e^{x_1+\ldots+ x_d - K_d}) + K_d
\end{equation}
and
let 
$ (\phi_d)_{d \in  \N} \subseteq \mathbf{N} $
satisfy 
for all 
$ d \in \N $
that
\begin{equation}
\label{eq:heat-eq5-phi}
\phi_d =
( ((1,\ldots, 1), -K_d), (1,K_d)) 
\in
(\R^{1 \times d } \times \R)
\times
(\R \times \R)
.
\end{equation}
Observe that 
\eqref{eq:heat-eq5-phi}
assures that
for all 
$ d \in \N $
we have that 
\begin{equation}
\label{eq:heat-eq5-par}
\mathcal{P}(\phi_d)
=
1(d+1)
+
1(1+1)
=
d+3
\leq
4d
\leq
\max\{4,c\}d
.
\end{equation}
Next note that the fact that
$
(\R \ni x \mapsto \ln(1+e^x) \in \R)
\in
C^1(\R,\R)
$,
\eqref{eq:heat-eq5-R},
and
\eqref{eq:heat-eq5-phi}
imply that
for all 
$ d \in \N $, $x=(x_1,\ldots,x_d) \in \mathbb R^d$
we have that
$ (\mathcal{R} \phi_d) \in C(\R^d,\R) $
and
\begin{equation}
\label{eq:heat-eq5-rphi1}
(\mathcal{R} \phi_d)(x)
=
\ln(1+e^{x_1+\ldots+x_d - K_d})
+
K_d = \varphi_d(x)
.
\end{equation}
Next note that for all $d\in \mathbb N$ we have that $\ln(1+e^{-K_d}) \leq \ln 2 + |K_d|$ 
and for any $d\in \mathbb N$, $x\in \mathbb R^d$ we have that $\|(\nabla \varphi_d)(x)\|_{\R^d} \leq d^{1/2}$.
This and the hypothesis that for all 
$ d \in \N $ 
we have that
$
|K_d|
\leq
cd^p
$
hence yield that 
for all
$ d \in \N $,
$ x=(x_1,\ldots,x_d) \in \R^d $
we have that
\begin{equation}
\label{eq:heat-eq5-rphi}
\begin{split}
|(\mathcal{R} \phi_d)(x)| & = |\varphi_d(x)|
\leq \sup_{y\in \mathbb R^d}\left[\|(\nabla \varphi_d)(y)\|_{\R^d} \|x\|_{\R^d} + |\varphi_d(0)|\right]
\\ & \leq
d^{1/2} \|x\|_{\R^d} + \ln 2 + 2|K_d|
\leq
d^{\nicefrac{1}{2}} \|x\|_{\R^d}
+ 1 + 2cd^p\\
& \leq \max\{1,2c\} (1+d^p + d^{1/2}\|x\|_{\R^d})\\
& \leq 2\max\{1,2c\}d^{\max\{p,\nicefrac{1}{2}\}}
(1 + \|x\|_{\R^d})\,.
\end{split}
\end{equation}
Next note that
for all
$ d \in \N $,
$ x \in \mathbb R^d $ 
we have that 
\begin{equation}
\label{eq:heat-eq5-drphi}
\left\|\big(\nabla (\mathcal{R} \phi_d)\big)(x) \right\|_{\R^d}
= \|(\nabla \varphi_d)(x) \|_{\R^d}
\leq
d^{\nicefrac{1}{2}}
\leq
2\max\{1,2c\}
d^{\nicefrac{1}{2}}
(1+\|x\|_{\R^d}^0)
.
\end{equation}

Combining this,
\eqref{eq:heat-eq5-par},
\eqref{eq:heat-eq5-rphi1},
\eqref{eq:heat-eq5-rphi},
and
the fact that 
$ 
(\R \ni x \mapsto \ln(1 + e^{x}) \in \R) \in C^1(\mathbb R, \mathbb R)
$
with
Theorem~\ref{thm:heat-eq}
(applied with
$ c \leftarrow \max\{4,4c\}$,
$ r \leftarrow 1$,
$p \leftarrow 1$,
$q \leftarrow 0$,
$v \leftarrow 0$,
$\mathbf{v} \leftarrow 0$,
$w \leftarrow \nicefrac{1}{2}$,
$\mathbf{w} \leftarrow 0$,
$z \leftarrow \max\{p,\nicefrac{1}{2}\}$,
$\mathbf{z} \leftarrow 1$,
$ \mathbf{a} \leftarrow \big(\R \ni x \mapsto \ln(1 + e^{x}) \in \R\big) $
in the notation of Theorem~\ref{thm:heat-eq})
establishes items
\eqref{it:heat-eq5-1}--\eqref{it:heat-eq5-2}.
This completes the proof of Corollary~\ref{cor:heat-eq5}.
\end{proof}
\subsubsection*{Acknowledgements}
This project is based on the master thesis of DK written from March 2018 to September 2018 at ETH Zurich under the supervision of AJ.
The authors thank Diyora Salimova for helpful comments.

\bibliographystyle{acm}
\bibliography{references}

\begin{thebibliography}{10}

\bibitem{aliprantis:border:infinite}
{\sc Aliprantis, C.~D., and Border, K.~C.}
\newblock {\em Infinite dimensional analysis}.
\newblock Springer, Berlin, 2006.

\bibitem{MR2500068}
{\sc Amann, H., and Escher, J.}
\newblock {\em Analysis. {III}}.
\newblock Birkh\"auser, Basel, 2009.

\bibitem{MR0165148}
{\sc Artin, E.}
\newblock {\em The gamma function}.
\newblock Holt, Rinehart and Winston, New York-Toronto-London, 1964.

\bibitem{BeckBeckerCheridito2019}
{\sc Beck, C., Becker, S., Cheridito, P., Jentzen, A., and Neufeld, A.}
\newblock {Deep splitting method for parabolic PDEs}.
\newblock {\em arXiv:1907.03452\/} (2019), 40 pages.

\bibitem{Becker2018}
{\sc Beck, C., Becker, S., Grohs, P., Jaafari, N., and Jentzen, A.}
\newblock Solving stochastic differential equations and {Kolmogorov} equations
  by means of deep learning.
\newblock {\em arXiv:1806.00421\/} (2018), 56 pages.

\bibitem{BeckJentzenE2019}
{\sc Beck, C., E, W., and Jentzen, A.}
\newblock Machine {L}earning {A}pproximation {A}lgorithms for
  {H}igh-{D}imensional {F}ully {N}onlinear {P}artial {D}ifferential {E}quations
  and {S}econd-order {B}ackward {S}tochastic {D}ifferential {E}quations.
\newblock {\em J. Nonlinear Sci. 29}, 4 (2019), 1563--1619.

\bibitem{Becker2019}
{\sc Becker, S., Cheridito, P., and Jentzen, A.}
\newblock Deep optimal stopping.
\newblock {\em Journal of Machine Learning Research 20}, 74 (2019), 1--25.

\bibitem{BeckerCheriditoJentzen2019}
{\sc Becker, S., Cheridito, P., Jentzen, A., and Welti, T.}
\newblock Solving high-dimensional optimal stopping problems using deep
  learning.
\newblock {\em arXiv:1908.01602\/} (2019), 42 pages.

\bibitem{MR2641641}
{\sc Bellman, R.}
\newblock {\em Dynamic programming}.
\newblock Princeton University Press, Princeton, NJ, 1957.

\bibitem{Berg2018AUD}
{\sc Berg, J., and Nystr{\"o}m, K.}
\newblock A unified deep artificial neural network approach to partial
  differential equations in complex geometries.
\newblock {\em Neurocomputing 317\/} (2018), 28--41.

\bibitem{BernerGrohsJentzen2018}
{\sc Berner, J., Grohs, P., and Jentzen, A.}
\newblock {Analysis of the generalization error: Empirical risk minimization
  over deep artificial neural networks overcomes the curse of dimensionality in
  the numerical approximation of Black-Scholes partial differential equations}.
\newblock {\em Revision requested from SIAM Journal on Mathematics of Data
  Science; arXiv:1809.03062\/} (2018), 35 pages.

\bibitem{Buehler2018}
{\sc Buehler, H., Gonon, L., Teichmann, J., and Wood, B.}
\newblock Deep hedging.
\newblock {\em Quantitative Finance 19}, 8 (2019), 1271--1291.

\bibitem{ChanMikaelWarin2019}
{\sc Chan-Wai-Nam, Q., Mikael, J., and Warin, X.}
\newblock Machine learning for semi linear {PDE}s.
\newblock {\em J. Sci. Comput. 79}, 3 (2019), 1667--1712.

\bibitem{chen2019deep}
{\sc Chen, Y., and Wan, J. W.~L.}
\newblock Deep neural network framework based on backward stochastic
  differential equations for pricing and hedging american options in high
  dimensions.
\newblock {\em arXiv:1909.11532\/} (2019), 35 pages.

\bibitem{MR3098996}
{\sc Cohn, D.~L.}
\newblock {\em Measure theory}.
\newblock Birkh\"auser, New York, 2013.

\bibitem{Coxetal2016}
{\sc Cox, S., Hutzenthaler, M., Jentzen, A., van Neerven, J., and Welti, T.}
\newblock Convergence in {H}\"older norms with applications to {M}onte {C}arlo
  methods in infinite dimensions.
\newblock {\em Revision requested from IMA. Journal on Numerical Analysis;
  arXiv:1605.00856\/} (2017), 48 pages.

\bibitem{Dockhorn2019}
{\sc Dockhorn, T.}
\newblock A discussion on solving partial differential equations using neural
  networks.
\newblock {\em arXiv:1904.07200\/} (2019), 9 pages.

\bibitem{EHanJentzen2017CMStat}
{\sc E, W., Han, J., and Jentzen, A.}
\newblock Deep learning-based numerical methods for high-dimensional parabolic
  partial differential equations and backward stochastic differential
  equations.
\newblock {\em Communications in Mathematics and Statistics 5}, 4 (2017),
  349--380.

\bibitem{weinan2018deep}
{\sc E, W., and Yu, B.}
\newblock The deep {R}itz method: {A} deep learning-based numerical algorithm
  for solving variational problems.
\newblock {\em Commun. Math. Stat. 6}, 1 (2018), 1--12.

\bibitem{ElbraechterSchwab2018}
{\sc Elbr{\"a}chter, D., Grohs, P., Jentzen, A., and Schwab, C.}
\newblock {DNN Expression Rate Analysis of High-dimensional PDEs: Application
  to Option Pricing}.
\newblock {\em arXiv:1809.07669\/} (2018), 50 pages.

\bibitem{EvansPDEs}
{\sc Evans, L.~C.}
\newblock {\em Partial differential equations}, second~ed., vol.~19 of {\em
  Graduate Studies in Mathematics}.
\newblock American Mathematical Society, Providence, RI, 2010.

\bibitem{Farahmand2017DeepRL}
{\sc Farahmand, A.-m., Nabi, S., and Nikovski, D.}
\newblock Deep reinforcement learning for partial differential equation
  control.
\newblock {\em 2017 American Control Conference (ACC)\/} (2017), 3120--3127.

\bibitem{FujiiTakahashi2019}
{\sc Fujii, M., Takahashi, A., and Takahashi, M.}
\newblock Asymptotic expansion as prior knowledge in deep learning method for
  high dimensional {BSDE}s.
\newblock {\em Asia-Pacific Financial Markets\/} (Mar 2019).

\bibitem{MR2343341}
{\sc Garling, D. J.~H.}
\newblock {\em Inequalities: a journey into linear analysis}.
\newblock Cambridge University Press, Cambridge, 2007.

\bibitem{MR1814364}
{\sc Gilbarg, D., and Trudinger, N.~S.}
\newblock {\em Elliptic partial differential equations of second order}.
\newblock Springer, Berlin, 2001.
\newblock Reprint of the 1998 edition.

\bibitem{GoudenegeMolent2019}
{\sc Gouden{\`e}ge, L., Molent, A., and Zanette, A.}
\newblock {Machine Learning for Pricing American Options in High Dimension}.
\newblock {\em arXiv:1903.11275\/} (2019), 11 pages.

\bibitem{graves2013speech}
{\sc Graves, A., Mohamed, A.-r., and Hinton, G.}
\newblock Speech recognition with deep recurrent neural networks.
\newblock In {\em Proceedings of the IEEE Conference on Acoustics, Speech and
  Signal Processing, {ICASSP}\/} (2013), pp.~6645--6649.

\bibitem{HornungJentzen2018}
{\sc Grohs, P., Hornung, F., Jentzen, A., and von Wurstemberger, P.}
\newblock A proof that artificial neural networks overcome the curse of
  dimensionality in the numerical approximation of {B}lack-{S}choles partial
  differential equations.
\newblock {\em To appear in Memoires of the American Mathematical Society;
  arXiv:1809.02362\/} (2018), 124 pages.

\bibitem{GrohsHornungJentzen2019}
{\sc Grohs, P., Hornung, F., Jentzen, A., and Zimmermann, P.}
\newblock Space-time error estimates for deep neural network approximations for
  differential equations.
\newblock {\em arXiv:1908.03833\/} (2019), 86 pages.

\bibitem{grohs2019deep}
{\sc Grohs, P., Jentzen, A., and Salimova, D.}
\newblock Deep neural network approximations for {M}onte {C}arlo algorithms.
\newblock {\em arXiv:1908.10828\/} (2019), 45 pages.

\bibitem{hairer2015}
{\sc Hairer, M., Hutzenthaler, M., and Jentzen, A.}
\newblock Loss of regularity for {K}olmogorov equations.
\newblock {\em Annals of Probabability 43}, 2 (2015), 468--527.

\bibitem{Han2018}
{\sc Han, J., Jentzen, A., and E, W.}
\newblock Solving high-dimensional partial differential equations using deep
  learning.
\newblock {\em Proceedings of the National Academy of Sciences 115}, 34 (2018),
  8505--8510.

\bibitem{HanLong2018}
{\sc Han, J., and Long, J.}
\newblock {Convergence of the Deep BSDE Method for Coupled FBSDEs}.
\newblock {\em arXiv:1811.01165\/} (2018), 26 pages.

\bibitem{henry2017deep}
{\sc Henry-Labord\`ere, P.}
\newblock Deep {Primal}-{Dual} {Algorithm} for {BSDEs}: {Applications} of
  {Machine} {Learning} to {CVA} and {IM}.
\newblock (November 15, 2017), 16 pages.
\newblock Available at SSRN: \url{https://ssrn.com/abstract=3071506}.

\bibitem{38131}
{\sc Hinton, G., Deng, L., Yu, D., Dahl, G., rahman Mohamed, A., Jaitly, N.,
  Senior, A., Vanhoucke, V., Nguyen, P., Sainath, T., and Kingsbury, B.}
\newblock Deep neural networks for acoustic modeling in speech recognition.
\newblock {\em Signal Processing Magazine\/} (2012).

\bibitem{Huang2017DenselyCC}
{\sc Huang, G., Liu, Z., and Weinberger, K.~Q.}
\newblock Densely connected convolutional networks.
\newblock {\em 2017 IEEE Conference on Computer Vision and Pattern Recognition
  (CVPR)\/} (2017), 2261--2269.

\bibitem{HurePhamWarin2019}
{\sc Hur{\'e}, C., Pham, H., and Warin, X.}
\newblock {Some machine learning schemes for high-dimensional nonlinear PDEs}.
\newblock {\em arXiv:1902.01599\/} (2019), 33 pages.

\bibitem{HutzenthalerJentzenKruse2019}
{\sc Hutzenthaler, M., Jentzen, A., Kruse, T., and Nguyen, T.~A.}
\newblock A proof that rectified deep neural networks overcome the curse of
  dimensionality in the numerical approximation of semilinear heat equations.
\newblock {\em Revision requested from SN Partial Differential Equations and
  Applications; arXiv:1901.10854\/} (2019), 29 pages.

\bibitem{MR3766391}
{\sc Hutzenthaler, M., Jentzen, A., and Wang, X.}
\newblock Exponential integrability properties of numerical approximation
  processes for nonlinear stochastic differential equations.
\newblock {\em Mathematics of Computation 87}, 311 (2018), 1353--1413.

\bibitem{MR3752640}
{\sc Hyt\"onen, T., van Neerven, J., Veraar, M., and Weis, L.}
\newblock {\em Analysis in {B}anach spaces. {V}ol. {II}}.
\newblock Springer, 2017.

\bibitem{JacquierOumgari2019}
{\sc Jacquier, A., and Oumgari, M.}
\newblock {Deep PPDEs for rough local stochastic volatility}.
\newblock {\em arXiv:1906.02551\/} (2019), 21 pages.

\bibitem{MR2856611}
{\sc Jentzen, A., and Kloeden, P.~E.}
\newblock {\em Taylor approximations for stochastic partial differential
  equations}, vol.~83 of {\em CBMS-NSF Regional Conference Series in Applied
  Mathematics}.
\newblock Society for Industrial and Applied Mathematics (SIAM), Philadelphia,
  PA, 2011.

\bibitem{JentzenSalimovaWelti2018}
{\sc Jentzen, A., Salimova, D., and Welti, T.}
\newblock A proof that deep artificial neural networks overcome the curse of
  dimensionality in the numerical approximation of {K}olmogorov partial
  differential equations with constant diffusion and nonlinear drift
  coefficients.
\newblock {\em arXiv:1809.07321\/} (2018), 48 pages.

\bibitem{karatzas:shreve:brownian}
{\sc Karatzas, I., and Shreve, S.~E.}
\newblock {\em Brownian motion and stochastic calculus}, second~ed.
\newblock Springer, New York, 1991.

\bibitem{klenke:probability}
{\sc Klenke, A.}
\newblock {\em Probability theory}.
\newblock Springer, 2014.

\bibitem{NIPS2012_4824}
{\sc Krizhevsky, A., Sutskever, I., and Hinton, G.~E.}
\newblock Imagenet classification with deep convolutional neural networks.
\newblock In {\em Advances in Neural Information Processing Systems 25},
  F.~Pereira, C.~J.~C. Burges, L.~Bottou, and K.~Q. Weinberger, Eds. Curran
  Associates, Inc., 2012, pp.~1097--1105.

\bibitem{KutyniokPeterseb2019}
{\sc Kutyniok, G., Petersen, P., Raslan, M., and Schneider, R.}
\newblock A theoretical analysis of deep neural networks and parametric {PDE}s.
\newblock {\em arXiv:1904.00377\/} (2019), 43 pages.

\bibitem{LongLuMaDong2018}
{\sc Long, Z., Lu, Y., Ma, X., and Dong, B.}
\newblock {PDE-Net: Learning PDEs from Data}.
\newblock In {\em Proceedings of the 35th International Conference on Machine
  Learning\/} (2018), pp.~3208--3216.

\bibitem{LyeMishraRay2019}
{\sc Lye, K.~O., Mishra, S., and Ray, D.}
\newblock Deep learning observables in computational fluid dynamics.
\newblock {\em arXiv:1903.03040\/} (2019), 57 pages.

\bibitem{Magill2018NeuralNT}
{\sc Magill, M., Qureshi, F., and de~Haan, H.~W.}
\newblock Neural networks trained to solve differential equations learn general
  representations.
\newblock In {\em Advances in Neural Information Processing Systems\/} (2018),
  pp.~4071--4081.

\bibitem{MR3736595}
{\sc Mizuguchi, M., Tanaka, K., Sekine, K., and Oishi, S.}
\newblock Estimation of {S}obolev embedding constant on a domain dividable into
  bounded convex domains.
\newblock {\em Journal of Inequalities and Applications\/} (2017).

\bibitem{petersen2017optimal}
{\sc Petersen, P., and Voigtlaender, F.}
\newblock {Optimal approximation of piecewise smooth functions using deep ReLU
  neural networks}.
\newblock {\em arXiv:1709.05289\/} (2017), 54 pages.

\bibitem{PhamWarin2019}
{\sc Pham, H., and Warin, X.}
\newblock {Neural networks-based backward scheme for fully nonlinear PDEs}.
\newblock {\em arXiv:1908.00412\/} (2019), 15 pages.

\bibitem{ReisingerZhang2019}
{\sc Reisinger, C., and Zhang, Y.}
\newblock Rectified deep neural networks overcome the curse of dimensionality
  for nonsmooth value functions in zero-sum games of nonlinear stiff systems.
\newblock {\em arXiv:1903.06652\/} (2019), 34 pages.

\bibitem{rockafellar:convex}
{\sc Rockafellar, R.~T.}
\newblock {\em Convex Analysis}.
\newblock Princeton University Press, 1970.

\bibitem{Silver2016}
{\sc Silver, D., Huang, A., Maddison, C.~J., Guez, A., Sifre, L., van~den
  Driessche, G., Schrittwieser, J., Antonoglou, I., Panneershelvam, V.,
  Lanctot, M., Dieleman, S., Grewe, D., Nham, J., Kalchbrenner, N., Sutskever,
  I., Lillicrap, T., Leach, M., Kavukcuoglu, K., Graepel, T., and Hassabis, D.}
\newblock Mastering the game of go with deep neural networks and tree search.
\newblock {\em Nature 529\/} (Jan. 2016), 484.

\bibitem{Silver2017}
{\sc Silver, D., Schrittwieser, J., Simonyan, K., Antonoglou, I., Huang, A.,
  Guez, A., Hubert, T., Baker, L., Lai, M., Bolton, A., Chen, Y., Lillicrap,
  T., Hui, F., Sifre, L., van~den Driessche, G., Graepel, T., and Hassabis, D.}
\newblock Mastering the game of go without human knowledge.
\newblock {\em Nature 550\/} (Oct. 2017), 354.

\bibitem{DBLP:journals/corr/SimonyanZ14a}
{\sc Simonyan, K., and Zisserman, A.}
\newblock Very deep convolutional networks for large-scale image recognition.
\newblock {\em arXiv:1409.1556\/} (2014).

\bibitem{Sirignano2019}
{\sc Sirignano, J., and Cont, R.}
\newblock Universal features of price formation in financial markets:
  perspectives from deep learning.
\newblock {\em Quantitative Finance 19}, 9 (2019), 1449--1459.

\bibitem{Sirignano2018}
{\sc Sirignano, J., and Spiliopoulos, K.}
\newblock D{GM: A} deep learning algorithm for solving partial differential
  equations.
\newblock {\em J. Comput. Phys. 375\/} (2018), 1339--1364.

\bibitem{wu2016stimulated}
{\sc Wu, C., Karanasou, P., Gales, M.~J., and Sim, K.~C.}
\newblock Stimulated deep neural network for speech recognition.
\newblock In {\em Interspeech 2016\/} (2016), pp.~400--404.

\end{thebibliography}
\end{document}